\documentclass[oneside, 12pt]{amsart}
\usepackage{amscd,amssymb,enumerate, amsmath,mathrsfs}
\usepackage{url}
\usepackage{hyperref}
\usepackage{color}

\newcommand{\stefan}[1]{{\color{red}#1}} 

\newtheorem{theorem}{Theorem}[section]

\newtheorem{lemma}[theorem]{Lemma}
\newtheorem{proposition}[theorem]{Proposition}
\newtheorem{corollary}[theorem]{Corollary}

\theoremstyle{definition}
\newtheorem{definition}[theorem]{Definition}
\newtheorem{example}[theorem]{Example}

\newtheorem{remark}[theorem]{Remark}
\newtheorem{emp}[theorem]{}
\newtheorem*{acknowledgement}{Acknowledgement}

\theoremstyle{remark}

\DeclareFontFamily{U}{wncy}{}
\DeclareFontShape{U}{wncy}{m}{n}{<->wncyr10}{}
\DeclareSymbolFont{mcy}{U}{wncy}{m}{n}
\DeclareMathSymbol{\Sh}{\mathord}{mcy}{"58}

\newcommand{\m}{\mu}
\newcommand{\mm}{\mu}

\newcommand{\MRone}{MR$\,$1}
\newcommand{\MRtwo}{MR$\,$2}
\newcommand{\MRthree}{MR$\,$3}
\newcommand{\MRfour}{MR$\,$4}
\newcommand{\MRfive}{MR$\,$5}

\newcommand{\mnr}	{{\text{\rm mnr}}}
\newcommand{\ci}	{\text{\rm ci}}

\newcommand{\NN}	{\mathbb{N}}
\newcommand{\ZZ}	{\mathbb{Z}}
\newcommand{\QQ}	{\mathbb{Q}}

\newcommand{\FF}	{\mathbb{F}}

\renewcommand{\AA}	{\mathbb{A}}
\newcommand{\GG}	{\mathbb{G}}

\newcommand{\idealq}    {\mathfrak{q}}
\newcommand{\ideala}    {\mathfrak{a}}
\newcommand{\idealb}    {\mathfrak{b}}

\newcommand  {\shG}     {\mathscr{G}}

\newcommand  {\shM}     {\mathscr{M}}

\newcommand  {\shL}     {\mathscr{L}}

\newcommand  {\Ass}     {\operatorname{Ass}}
\newcommand  {\Aut}     {\operatorname{Aut}}

\newcommand  {\Cl}      {\operatorname{Cl}}

\newcommand  {\Cov} 	{{\rm\text{Cov}}}
\renewcommand{\cong}    {\equiv}

\newcommand  {\depth}   {\operatorname{depth}}

\newcommand  {\edim}    {\operatorname{edim}}

\newcommand  {\End}   	{\operatorname{End}}

\newcommand  {\Ext}     {\operatorname{Ext}}

\newcommand  {\Fil}	{\operatorname{Fil}}

\newcommand  {\Frac}    {\operatorname{Frac}}

\newcommand  {\Gal}     {\operatorname{Gal}}
\newcommand  {\GL}      {\operatorname{GL}}
\newcommand  {\FFGrp}  	{{\text{{\rm FFG}}}}
\newcommand  {\Gr}      {\operatorname{Gr}}

\newcommand  {\Hom}     {\operatorname{Hom}}

\newcommand  {\height}  {\operatorname{ht}}

\newcommand  {\id}      {{\operatorname{id}}}

\newcommand  {\lra}     {\longrightarrow}

\newcommand  {\Mat}     {\operatorname{Mat}}

\newcommand  {\maxid}   {\mathfrak{m}}

\newcommand  {\primid}  {\mathfrak{p}}
\renewcommand{\O}       {\mathscr{O}}

\newcommand  {\ord}     {\operatorname{ord}}

\newcommand  {\Pic}     {\operatorname{Pic}}

\newcommand  {\pr}      {\operatorname{pr}}

\newcommand  {\quadand} {\quad\text{and}\quad}

\newcommand  {\ra}      {\rightarrow}

\newcommand  {\Set}     {{\text{\rm Set}}}

\newcommand  {\Spec}    {\operatorname{Spec}}

\newcommand  {\Trp}     {\text{{\rm Trp}}}

\newcommand  {\Trace}   {\operatorname{Tr}}

\newcommand  {\val}  	{\operatorname{val}}

\newcommand{\SR}{S}

\def\mydate{\number\day\space\ifcase\month \or January\or February\or March\or 
April\or May\or June\or July\or
August\or September\or October\or November\or December\fi \space\number\year}

\DeclareFontFamily{U}{wncy}{}
\DeclareFontShape{U}{wncy}{m}{n}{<->wncyr10}{}
\DeclareSymbolFont{mcy}{U}{wncy}{m}{n}
\DeclareMathSymbol{\Sh}{\mathord}{mcy}{"58}
\numberwithin{equation}{section}
\numberwithin{figure}{section}
\numberwithin{table}{section}
\if false
\makeatletter
    \let\c@equation\c@thm
    \let\c@figure\c@thm
   \let\c@table\c@thm
\makeatother
\fi
\usepackage{amsmath,amssymb}
\makeatletter
\newsavebox\myboxA
\newsavebox\myboxB
\newlength\mylenA

\newcommand*\xoverline[2][0.75]{%
    \sbox{\myboxA}{$\m@th#2$}%
    \setbox\myboxB\null
    \ht\myboxB=\ht\myboxA%
    \dp\myboxB=\dp\myboxA%
    \wd\myboxB=#1\wd\myboxA
    \sbox\myboxB{$\m@th\overline{\copy\myboxB}$}
    \setlength\mylenA{\the\wd\myboxA}
    \addtolength\mylenA{-\the\wd\myboxB}%
    \ifdim\wd\myboxB<\wd\myboxA%
       \rlap{\hskip 0.5\mylenA\usebox\myboxB}{\usebox\myboxA}%
    \else
        \hskip -0.5\mylenA\rlap{\usebox\myboxA}{\hskip 0.5\mylenA\usebox\myboxB}%
    \fi}
\makeatother

\numberwithin{equation}{section}

\numberwithin{figure}{section}
\numberwithin{table}{section}

\makeatletter
    \let\c@equation\c@theorem
    \let\c@figure\c@theorem
    \let\c@table\c@theorem
\makeatother

\begin{document}

\title[Moderately ramified actions]
      {Moderately ramified  actions in positive characteristic}

\author[Dino Lorenzini]{Dino Lorenzini}
\address{Department of Mathematics, University of Georgia, Athens, GA 30602, USA}
\curraddr{}
\email{lorenzin@uga.edu}

\author[Stefan Schr\"oer]{Stefan Schr\"oer}
\address{Mathematisches Institut, Heinrich-Heine-Universit\"at,
40204 D\"usseldorf, Germany}
\curraddr{}
\email{schroeer@math.uni-duesseldorf.de}

 

\dedicatory{\today}

\begin{abstract}
In characteristic $2$ and dimension $2$,  wild ${\mathbb Z}/2{\mathbb Z}$-actions on $k[[u,v]]$ ramified precisely at the origin
were classified by Artin, who showed in particular that they induce hypersurface singularities. 
We introduce in this article a new class of  wild quotient singularities in any characteristic $p>0$ and dimension $n\geq 2$
arising from  certain  non-linear actions of $\ZZ/p\ZZ$ on the formal power series ring $k[[u_1,\dots,u_n]]$.
These actions are ramified precisely at the origin, and their rings of invariants in dimension $2$ are hypersurface singularities,
with an equation of a  form similar to the form found by Artin when $p=2$. In higher dimension, the rings of invariants are not local complete intersection in general, but remain quasi-Gorenstein.
We establish several structure results for such actions and their corresponding rings of invariants.
\end{abstract}

\maketitle
\tableofcontents

\section{Introduction}
\label{Introduction}

Given a smooth quasi-projective algebraic variety  $X$ over a field $k$ and a finite subgroup $G$ of automorphisms
of $X$, the quotient $X/G$ exists, and a precise understanding of the singularities of $X/G$ is often crucial in many problems in algebraic geometry, 
including in the geography of surfaces of general type, and in the study of the automorphisms of K3 surfaces and Enriques surfaces. The initial study 
of a quotient singularity on $X/G$ is  local, and consists in the analysis of the action of the isotropy subgroup of a point $x \in X$ on the completion of the regular  local ring ${\mathcal O}_{X,x}$.  

In characteristic zero, the action of a finite group $G$ on the power series ring $A=k[[u_1,\dots,u_n]]$ is always linearizable, an observation going
back to H.\ Cartan \cite{Cartan 1957}. The ring of invariants $A^G$ has then good algebraic properties, 
such as being Cohen--Macaulay \cite{H-E}, and the singularity of $\Spec A^G$ is even rational (\cite{Boutot 1987}, \cite{Bri}).
Watanabe gave in \cite{WatI} and \cite{WatII} an explicit criterion for $A^G$ to be Gorenstein.

When $k$ has positive characteristic $p$ and  the order of  $G$ is divisible by $p$, 
an action of $G$ on $k[[u_1,\dots,u_n]]$ is called {\it wild}. Such actions are much more delicate to study. 
Most wild actions of $G$ on  $k[[u_1,\dots,u_n]]$ are not linearizable,   in which case 
the group $G$  acts via true  power series substitutions. 
The  resulting rings $A^G$ are usually not Cohen--Macaulay when $n\geq 3$ (\ref{properties invariant ring}).  
The elements of order $p$ in the group ${\rm Aut}_k(k[[u_1,\dots,u_n]])$ are completely understood only when $n=1$ and $k$ is finite (\cite{Klo}, \cite{Lub}).
In this article, we will focus on the geometric case where the action is ramified precisely at the origin (\ref{ramifiedprecisely}). 
This type of wild action is never induced by a linear action on the variables (\ref{no invariant parameter}), and thus the large body of results in modular invariant theory is not immediately applicable to their study.

\if false
\stefan{
Most wild actions of $G$ on  $k[[u_1,\dots,u_n]]$ are not linearizable,   in which case 
the group $G$  acts via true \emph{power series substitutions}. In fact, this is  inevitable if the 
action if free in codimension one, and  allows for very complicated  actions.
For this setting, the many results from modular representation theory are often of limited use.
Note that the  resulting rings $A^G$ are usually not     Cohen--Macaulay when $n\geq 3$, an observation due to Fogarty \cite{Fog}.
}
The elements of order $p$ in the group ${\rm Aut}_k(k[[u_1,\dots,u_n]])$ are completely understood only when $n=1$ and $k$ is finite \cite{Klo}.
\fi

The starting point of our article is a  result  of Artin \cite{Artin 1975}, who analyzed all wild actions of $G=\ZZ/ 2\ZZ$, ramified precisely at the origin,
in dimension $n=2$ when $k$ is algebraically closed, and showed that
$$
A^G=k[[x,y,z]]/(z^2-abz-ya^2+xb^2)
$$
for some system of parameters $a,b\in k[[x,y]]$. In particular, this explicit description of $A^G$ shows that it is a complete intersection and, hence, Gorenstein. 
Artin noted  in \cite{Artin 1975} that it would be interesting to extend his result to wild  $\ZZ/ p\ZZ$-actions when $p>2$. 
Peskin \cite{Peskin 1983} subsequently described an explicit class of wild $\ZZ/ 3\ZZ$-actions ramified precisely at the origin with  ring of invariants  also described by an explicit equation when $n=2$. In general, though, most wild $\ZZ/ p\ZZ$-actions on $k[[u_1,u_2]]$ ramified precisely at the origin do not have Gorenstein rings of invariants (see \ref{rem.notGorenstein}).

 In this article, inspired by the work of Artin, 
we introduce a new class of wild  $\ZZ/ p\ZZ$-actions on $k[[u_1,\ldots, u_n]]$ ramified precisely at the origin, for all $n \geq 2$ and all primes $p$,
whose rings of invariants are complete intersection when $n=2$, and {\it quasi-Gorenstein} (\ref{canonical class trivial}) in general. 
Our main motivation for introducing these actions, in addition to their intrinsic interest 
as new explicit elements in the mysterious group of automorphisms of $k[[u_1,\ldots, u_n]]$, lies in the fact that  already in dimension $n=2$, we believe that this class of actions is rich enough to potentially solve two open problems concerning resolutions of singularities 
of $\ZZ/ p\ZZ$-quotient singularities of surfaces: determine whether in the class of minimal resolutions of $\ZZ/ p\ZZ$-quotient singularities, 
there is no bound on the number of vertices of valency at least $3$  that the minimal resolution graph can have; 
and determine whether every power of $p$ occurs as the determinant of the intersection matrix of a resolution of some $\ZZ/ p\ZZ$-quotient singularity when $p>2$.
We  address the latter problem in  \cite{Lorenzini Schroeer 2}.

An analysis of  Artin's arguments in \cite{Artin 1975} for $p=n=2$ 
shows that it is possible to formulate  five  axiomatic conditions 
on a $\ZZ/ p\ZZ$-action,
such that any $\ZZ/ p\ZZ$-action satisfying all five of these axioms can be described in a particularly simple way, 
similar to the description obtained by Artin in the case $p=n=2$.  We call such $\ZZ/ p\ZZ$-actions \emph{moderately ramified}. 
As a byproduct, we also obtain a class of $(\ZZ/ p\ZZ)^{n-1}$-actions on $k[[u_1,\ldots, u_n]]$
whose rings of invariants are local complete intersections (\ref{maximal subgroups}).

Our new class of wild  $\ZZ/ p\ZZ$-actions can be described as follows. Start with an action of  $G:=\ZZ/ p\ZZ$ on $A:=k[[u_1,\ldots, u_n]]$ which is 
ramified precisely at the origin.
We call $x_i:= \prod_{\sigma\in G}\sigma(u_i)$ a   norm  element, and we show in \ref{exist admissible} that it is always possible to find 
a regular system of parameters $u_1,\dots, u_n$ in $A$ such that $A$ is finite of rank $p^n$ over the {\it norm subring} $R:=k[[x_1,\dots,x_n]]$.
Our axioms  ensure that the ring $A$ comes with a norm subring $R$
such that $\Frac(A)/\Frac(R)$ is Galois with elementary abelian Galois group $H$ of order $p^n$,
and that $n$ distinguished  subgroups $G_i^\perp$ of index $p$ yield   intermediate rings 
of invariants $R\subset A^{G_i^\perp}\subset A$ that result in a decomposition $A= A^{G_1^\perp}\otimes_R\ldots\otimes_R A^{G_n^\perp}$.
Our description then relies on an analysis of the resulting schemes $Y_i:= \Spec A^{G_i^\perp} \to S:=\Spec R$ endowed with the natural action of $H/G_i^\perp$. 

Extending to base schemes $S$ of dimension bigger than one a construction of Raynaud \cite{Ray1999}, 
which was already further extended in dimension one by Romagny \cite{Romagny 2012} to groups which are not necessarily finite and flat over $S$, 
we attach to the action of $H/G_i^\perp$ on $Y_i$  a finite flat group scheme $\shG_i/S$ and a natural action $\shG_i\times_S Y_i\ra Y_i$.
In keeping with the terminology introduced by Romagny, we call the action $\shG_i \times_S Y_i\to Y_i$ the {\it effective model} 
of the action of the constant group scheme $H/G_i^\perp$ on $Y_i$. 
Our final axiom 
imposes that $Y_i/S$ be a torsor under the action of  $\shG_i/S$, for $i=1,\dots, n$. 
We then use the \emph{Tate--Oort classification} of such group schemes $\shG_i/S$
\cite{Tate; Oort 1970}, and obtain explicit equations describing a moderately ramified $G$-action on $A$ in Theorem \ref{structure moderately ramified}.

Our terminology ``moderately ramified $G$-action on $A$" refers to the fact that the ramification
of the associated morphism $\Spec A \to \Spec A^H$ is ``as small as possible". 
Moderately ramified actions are in some sense ``as free as possible''
outside ramification at the closed point.
When $p=n=2$,  all five axioms are automatically satisfied when $k$ is algebraically closed.
For arbitrary $p\geq 2$, we obtain in dimension $n=2$ a description in \ref{2-dimensional invariant ring} of the ring of invariants $A^G$ which 
generalizes Artin's description when $p=2$ in \cite{Artin 1975}.

In dimension $n>2$, the ring $A^G$ is never a complete intersection since it is known that this ring is not 
Cohen--Macaulay when the action is ramified precisely at the origin. 
Thus an explicit description of $A^G$  by generators and relations is  in general out of reach in higher dimensions.
When $n=3$, we are able to provide  generators for $A^G$ in \ref{thm.n=3} when the system of parameters associated with the $G$-action consists of polynomials rather than general power series. We are also able to provide  a formula
for the embedding dimension.
Using methods from non-commutative algebra, we show for all $n\geq 2$ in \ref{canonical class trivial}   that $A^G$ is quasi-Gorenstein, that is,
that the dualizing module $K_{A^G}$ is trivial in the class group $\Cl(A^G)$.

\medskip
The paper is organized as follows.
In Section \ref{norm subrings}  we study general properties of  norm subrings $R$ inside $A=k[[u_1,\ldots,u_n]]$.
Section \ref{ramified precisely origin} focuses on wild $G$-actions on $A$ that are ramified precisely at 
the origin, and presents a criterion for the extension $\Frac(A)/\Frac(R)$ to be Galois in terms of the extension 
$\Frac(A^G)/\Frac(R)$. The proof of this criterion uses results on fundamental groups.
Section \ref{extensions torsors} deals with  extensions of group schemes and torsors, and proves the existence 
of the effective model of a group action  over higher-dimensional bases.
These concepts are further developed explicitly in dimension $1$ in Section \ref{wild actions in dimension one}.
Section \ref{moderately ramified} introduces the  central notion of moderately ramified action
and   the main structure results for such action.
In Section \ref{complete intersection subrings} we study various auxiliary invariant rings that
are attached to moderately ramified actions. We treat the case $n=3$ in Section \ref{linear situation}, 
where we present a set of generators for the ring $A^G$ for certain moderately ramified actions.
We show that the canonical class $[K_{A^G}]\in\Cl(A^G)$ is trivial in Section \ref{class group}.

\begin{acknowledgement}
The authors thank the referee for a careful reading of the manuscript and for useful comments. 
The authors gratefully acknowledges funding support from the 
 research training group \emph{Algebra, Algebraic Geometry, and Number Theory}
at the University of Georgia, from the National Science Foundation RTG grant DMS-1344994, from the Simons Collaboration Grant 245522,
and the research training group
\emph{GRK 2240: Algebro-geometric Methods in Algebra, Arithmetic and Topology}
of the Deutsche Forschungsgemeinschaft.
\end{acknowledgement}

\section{Norm subrings}
\label{norm subrings}

Let $A$ denote a complete local  noetherian ring that is regular, of dimension $n\geq 1$,
and  with  maximal ideal $\maxid_A$. 
Recall that a field of representatives for $A$ is a subfield $k$ of $A$
such that the composition $k \subset A \to A/\maxid_A$ is an isomorphism of fields. 
When $A/\maxid_A$ is perfect, such a subfield   is unique. 
We will always in this article assume that $A$ contains a field, and we fix a field of 
representatives $k\subset A$ and regard it as ground field.
A set of $n$  elements $x_1,\ldots, x_n\in A$ which generate an  $\maxid_A$-primary
ideal in $A$ is called  \emph{system of parameters}. The following facts are well-known.

\begin{lemma}
\label{parameter system}
If $x_1,\ldots,x_n\in A$ is a system of parameters, then the ensuing
homomorphism of complete local $k$-algebras 
$R:=k[[x_1,\ldots,x_n]]\ra A$ 
is finite and flat. Its degree is the length of $A/(x_1,\ldots,x_n)A$.
\end{lemma}

\proof
The ring $A$, viewed as an $R$-module is finite, according
to \cite{ZS}, Corollary 2 and Remark on page 293.
Since $R$ is regular and $A$ is Cohen--Macaulay, the 
$R$-module $A$ is free of finite rank,
by \cite{Serre 1965}, Proposition 22, page IV-37. 
This rank is the vector space dimension of $A/(x_1,\ldots,x_n)A$
over $k=R/(x_1,\ldots,x_n)$, which coincides with the length of
$A/(x_1,\ldots,x_n)A$.
\qed

\medskip
A set of $n$ elements $u_1,\dots, u_n$ in $A$ which  generate $\maxid_A$
is called a \emph{regular system of parameters}. The resulting homomorphism
$k[[u_1,\dots, u_n]]\ra A$ is then bijective.

Let $G\subset \Aut(A)$ be a finite group of    automorphisms such that
 $k$ lies in the ring of invariants $A^G$. Clearly, $\maxid_A \cap A^G$ is a maximal ideal with residue field $k$. 
 It follows from \cite{Mol}, Th\'eor\`eme 2, 
 that $A^G$ is a complete {\it noetherian} local ring and that $A$ is an $A^G$-module of finite type.
Choose   a regular system of parameters $u_1,\ldots,u_n\in A$, 
and  consider the \emph{norm elements}
$$
x_i:=N_{A/A^G}(u_i)=\prod_{\sigma\in G} \sigma(u_i),\quad 1\leq i\leq n.
$$
These elements are obviously $G$-invariant, and we can consider the complete local $k$-subalgebra $R\subset  A^G$
generated by $x_1,\dots,x_n$. 
Let us call $R$ a {\it norm subring} of $A$.
The ring extensions $R\subset A^G\subset A$ play a crucial role in this article.
In this section, we determine under what conditions the extension $R\subset A$ is finite.

\begin{definition} \label{admissible}
A regular system of parameters $u_1,\ldots,u_n\in A$ is called  \emph{weakly admissible} (resp.  \emph{admissible}) with respect
to the $G$-action  if,  
for each $(\sigma_1,\ldots,\sigma_n)\in G^n$, the elements
$\sigma_1(u_1),\ldots,\sigma_n(u_n) \in A$ form a  system of parameters (resp., form a regular system of parameters).
\end{definition}
The main justification for introducing this notion is the following result:

\begin{proposition} \label{norm elements}
Let $u_1,\ldots,u_n$ be a regular system of parameters of $A$.
Then the associated norm elements $x_1,\ldots,x_n $ 
form a   system of parameters in $A$ if and only if the elements $u_1,\ldots,u_n$ are weakly admissible with respect to the action of $G$.
If the above equivalent conditions hold, then the homomorphism of $k$-algebras 
$k[[x_1,\ldots,x_n]]\ra A$ 
is finite and flat of degree at least $|G|^n$, with equality when the regular system is admissible. 
\end{proposition}

\proof
Let $g:=|G|$. Choose an enumeration $\sigma_1, \dots, \sigma_g$ of the elements of $G$ such that $\sigma_1:=e$
is the neutral element.
Define $f_{i,j}:=\sigma_j(u_i)$, with $1\leq j\leq g$ and $1\leq i\leq n$. In particular,
$f_{i,1}=u_i$ and $x_i=\prod_{j=1}^g f_{i,j}$.
Let $X:=\Spec(A)$ and consider the   closed subsets $V(f_{i,j})\subset X$. Clearly, 
$$
V(x_1,\dots,x_n)= \bigcap_{i=1}^n \left( \bigcup_{j=1}^g V(f_{i,j}) \right).
$$
Let $J:=\{1,\ldots,g\}$. Using the distributive properties of $\cap $ and $\cup$, we find that 
\begin{equation} \label{cupcap}
V(x_1,\dots,x_n)=  \bigcup_{(j_1,\ldots,j_n)\in J^n} \left(\bigcap_{i=1}^n V(f_{i,j_i}) \right) = \bigcup_{(j_1,\ldots,j_n)\in J^n} V(f_{1,j_1}, \dots, f_{n,j_n}).
\end{equation}
If the elements $x_1,\ldots,x_n\in A$ form a   system of parameters,
then $V(x_1,\dots,x_n)$ contains only the closed point and, 
thus, each $V(f_{1,j_1}, \dots, f_{n,j_n})$
consists only of the closed point. 
Hence, for each tuple $(j_1,\ldots,j_n)$, the elements  $f_{i,j_1}, \dots, f_{i,j_n} \in A$ form a   system of parameters.
In other words, the elements $u_1,\ldots,u_n\in A$ are weakly admissible. Conversely, when the elements $u_1,\ldots,u_n\in A$ are weakly admissible, 
it immediately follows from \eqref{cupcap} that the elements $x_1,\ldots,x_n\in A$ form a   system of parameters.

Now suppose that the regular system of parameters $u_1,\ldots,u_n\in A$ is weakly admissible,
such  that the associated norm elements $x_1,\ldots,x_n $ form a system of parameters in $A$.
According to Lemma \ref{parameter system}, 
the ring $A$ is a finite and flat module over $R=k[[x_1,\ldots,x_n]]$.
To determine its rank, we need to 
compute the length of  $A/I$, where  $I:=(x_1,\ldots,x_n)A$. We now establish  the 
inequality
$\text{\rm length}(A/I)\geq g^n$ using general facts on the \emph{Hilbert--Samuel multiplicity} $e(I,A)$ of $I$. 
Since we assume that the elements $u_1,\ldots,u_n\in A$ are weakly admissible, 
the ideal $I\subset A$ is generated by a system of parameters.
Since the ring $A$ is Cohen--Macaulay, the formula  
$\text{\rm length}(A/I)= e(I,A)$ holds (\cite{Auslander; Buchsbaum 1958}, Proposition 5.9).
For any $\maxid_A$-primary ideal $J\subset A$,  we have 
$e(J^s,A)=s^n e(J,A)$ (\cite{Matsumura 1989}, page 108). 
Applying this to $J:={\mathfrak m}_A^g$ and using the fact that 
$e({\mathfrak m}_A, A)=1$ since $A$ is regular, 
we find that $e(J, A)=g^n$. 
Finally, the inclusion $I\subset J$ gives   $e(I, A) \geq e(J,A)$ (\cite{Matsumura 1989}, page 109),
and the desired inequality follows.

Finally, assume that the regular system of parameters $u_1,\ldots,u_n\in A$ is admissible.
To proceed, it is convenient 
to consider the slightly more general situation where we omit the precise definition of the element $f_{i,j}$, 
and keep only the following hypothesis: for each element  $(j_1,\ldots,j_n) \in J^n$, the 
elements $f_{1,j_1},\ldots, f_{n,j_n}$ form a regular system of parameters of $A$. 
In particular, 
$f_{i,j}\in \maxid_A \setminus \maxid_A^2$ for each $i \in [1,n] $ and $j \in [1,g]$.
Letting $u_i:=f_{i,1}$, we find that $u_1,\dots, u_n$ form a regular system of parameters of $A$.

For each $\ell \in [1,n]$ and $r \in [1,g]$, define $g_{\ell,r}:=\prod_{j=1}^{r}f_{\ell,j}$, and denote as previously $x_\ell:=g_{\ell,g}$.
When $2 \leq \ell \leq n-1$, consider the   ideal
$$
I_{\ell,r}:=(x_1,\ldots,x_{\ell-1}, g_{\ell,r},u_{\ell+1},\dots, u_{n})\subset A.
$$
Define similarly the ideals $I_{1,r}=(g_{1,r},u_{2},\dots, u_{n})$ and $I_{n,r}=(x_1,\ldots,x_{n-1}, g_{n,r})$ for each $r \in [1,g]$. With respect to the lexicographic ordering
on the set of all pairs $(\ell, r)$ in $[1,n]\times [1,g]$, the $gn$ ideals $
I_{\ell,r} $ form a decreasing sequence
$$
I_{\ell,r}=(x_1,\ldots,x_{\ell-1}, g_{\ell,r},u_{\ell+1},\dots, u_{n}) \supseteq
(x_1,\ldots,x_{\ell-1}, g_{\ell,r+1},u_{\ell+1},\dots, u_{n})= I_{\ell,r+1}
$$
between  $I_{1,1}=(u_1,\ldots,u_n)$ and $I_{n,g}=(x_1,\ldots,x_n)$. Note  that this
sequence of ideals contains the repetitions $I_{\ell,g}=I_{\ell+1,1}$.

Since the module $I_{\ell,r} / I_{\ell,r+1}$ is generated by $g_{\ell,r}$, 
and because $g_{\ell,r+1}= g_{\ell,r} f_{\ell,r+1}$, we find that $I_{\ell,r} / I_{\ell,r+1}$ is annihilated by 
$$
J_{\ell,r+1}:=(x_1,\ldots,x_{\ell-1}, f_{\ell,r+1},u_{\ell+1},\dots, u_{n}).
$$
It follows that $\text{\rm length} (A/J_{\ell,r+1}) \geq \text{\rm length} (I_{\ell,r} / I_{\ell,r+1})$.
We claim that 
$$
\text{\rm length} (A/J_{\ell,r+1}) \leq g^{\ell-1} 
$$ 
for all $r \in [0,g-1]$ and $\ell \in [1,n]$. Assuming this claim, we find the upper bound
\begin{equation} \begin{array}{rcl}
\label{eq.length}
\text{\rm length}(A/(x_1,\ldots,x_n)) 
&=& 
1+\sum_{\ell=1}^n \sum_{r=1}^{g-1} \text{\rm length} (I_{\ell,r} / I_{\ell,r+1}) \\
&\leq& 1+\sum_{\ell=1}^n (g-1)g^{\ell-1}
= g^n.
\end{array}
\end{equation}
To verify the claim, we  proceed by   induction on $n=\dim(A)$. 
Assume that $n=1$. By construction, $J_{1,r+1}=(f_{1,r+1})$. By hypothesis, $f_{1,r+1} \in \maxid_A \setminus \maxid_A^2$.
Hence, $f_{1,r+1}$ is a uniformizer and
$A/(f_{1,r+1})$ has length one, as desired. 

Assume now that $n>1$, and that the assertion holds for $n-1$. Since $f_{\ell,r+1}$ is part of a regular system of parameters, 
we find that the ring $\overline{A}:=A/(f_{\ell,r+1})$ is a complete regular local ring of dimension $n-1$. 
When $a \in A$, denote by $\overline{a}$ its class in $\overline{A}$. 
When $2 \leq \ell \leq n-1$, we have 
\begin{equation} \label{length}
\text{\rm length} (A/J_{\ell,r+1}) =  
\text{\rm length} (\overline{A}/(\overline{x}_1,\ldots,\overline{x}_{\ell-1}, \overline{u}_{\ell+1},\dots, \overline{u}_{n})).
\end{equation}
Similarly, $
A/J_{1,r+1}$  and $\overline{A}/(\overline{u}_{2},\dots, \overline{u}_{n}))$ have the same length,
and
$A/J_{n,r+1}$  and $\overline{A}/(\overline{x}_1,\ldots,\overline{x}_{n-1})$ also have same length.
We now observe that we can apply our induction hypothesis to bound the right-hand side of \eqref{length}.
Indeed, the set
of elements $ \overline{f}_{i,r}\in\overline{A} $ with $ i \in [1,n]$ and $i \neq \ell$, and with $r \in [1,g]$, inherits the property
that every sequence $\overline{f}_{1,j_1}, \ldots, \overline{f}_{n,j_n}$ of $n-1$ elements
(where no term  $\overline{f}_{\ell,j_\ell}$ appears) is a  regular system of parameters in $\overline{A}$. 
In case  $\ell< n$, we can view $\overline{u}_{\ell+1}$ as $\overline{f}_{\ell +1 ,1}$  and  conclude by induction.
In the boundary case   $ \ell =n$, it suffices to show that  $\text{\rm length} (\overline{A}/(\overline{x}_1,\ldots,\overline{x}_{n-1})\leq g^{n-1}$.
To prove this inequality we use the argument in \eqref{eq.length} and use induction to justify each inequality found in that argument.
\qed

\medskip
Recall that a $G$-action   on $A=k[[u_1,\ldots,u_n]]$ is called \emph{linear} if each $\sigma\in G$
acts as a substitution of variables $u_j\mapsto \sum_{i=1}^n\lambda_{ij}^\sigma u_i$ for some
linear representation $G\ra\GL_n(k)$, $\sigma\mapsto (\lambda_{ij}^\sigma)$.
For instance, the ${\mathbb Z}/2{\mathbb Z}$-action on $k[[u,v]]$ which permutes $u$ and $v$ is linear, 
and our next lemma implies that $u,v$ is not weakly admissible for this action.

\begin{lemma}
Suppose that the $G$-action on $A=k[[u_1,\ldots,u_n]]$ is linear.
If there is an element $\sigma\in G$ whose matrix $(\lambda_{ij}^\sigma)$
has a zero on the diagonal, then the regular system of parameters
$u_1,\ldots,u_n\in A$ is not weakly admissible with respect to the action of $G$.
\end{lemma}

\proof
After reordering the $u_1,\ldots,u_n\in A$,
we may assume that  the non-zero diagonal entry is $\lambda_{11}^\sigma=0$.  
For the  tuple $(\sigma, e,\ldots, e)\in G^n$, where $e\in G$ denotes the neutral element,
the resulting elements  $\sigma(u_1),u_2, \ldots, u_n\in A$ 
do not form a system of parameters, because they generate the  ideal $(u_2,\ldots,u_n)\subset A$.
\qed

\medskip
Any regular system of parameters  $u_1,\ldots,u_n\in A$  induces a basis $\overline{u}_1,\dots, \overline{u}_n$ of the cotangent space $\maxid_A/\maxid_A^2$.
Fixing this basis of $\maxid_A/\maxid_A^2$, we obtain an induced linear representation $G\ra\GL(\maxid_A/\maxid_A^2)= \GL_n(k)$.

\begin{lemma} \label{lem.Borel}
Let $u_1,\ldots, u_n\in A$ be a regular system of parameters
with induced linear representation $G\ra\GL(\maxid_A/\maxid_A^2) = \GL_n(k)$.
If the image of $G$ is contained in the Borel subgroup of upper triangular matrices,
then $u_1,\ldots,u_n$ is admissible with respect to the action of $G$.
\end{lemma}
\proof Given $(\sigma_1,\ldots,\sigma_n)\in G^n$, we need to show that the elements
$\sigma_1(u_1),\ldots,\sigma_n(u_n)$ form a regular system of parameters of $A$. For this, it suffices to show that
the images $\overline{\sigma_1(u_1)},\ldots,\overline{\sigma_n(u_n)} $ in $\maxid_A/\maxid_A^2$ form a basis.
That they form a basis is clear from the assumption
that each $\sigma_i$ induces in the basis $\overline{u}_1,\dots, \overline{u}_n$ a matrix which is upper triangular.
\qed

\medskip
Note in particular that it follows from Lemma \ref{lem.Borel} that
if the induced representation of $G$ on the cotangent space is trivial, 
then any regular system of parameters is automatically admissible.  

\begin{proposition} \label{exist admissible} 
If the ring $A$ has characteristic $p>0$ and the group $G$ 
is a finite $p$-group, then  $A$ contains an admissible regular system of parameters.
\end{proposition}

\proof Let $u_1, \dots, u_n$ be any regular system of parameters of $A$ with induced linear representation $G\ra\GL(\maxid_A/\maxid_A^2)= \GL_n(k)$.
Since $G$ is a finite $p$-group, it is possible to change basis in $\maxid_A/\maxid_A^2$ so that in the new basis, every element in the image of $G$
is an upper triangular matrix whose diagonal coefficients are all $1$ (\cite{Ser}, Proposition 26 on  page 64).
The proposition then follows from \ref{lem.Borel}. \qed

\section{Actions ramified precisely at the origin}
\label{ramified precisely origin}

Let $A$ be a complete   local noetherian ring that is regular, of dimension $n\geq 2$,  
and characteristic $p>0$, with maximal ideal $\maxid_A$ and field of representatives $k$.
Let $G$ be a finite cyclic group of order $p$, and assume that $A$ is endowed with a 
faithful     action of $G$ so that $k$ lies in the ring of invariants $A^G$.
For each prime ideal $\primid\subset A$, we write $I_\primid\subset G$
for the inertia subgroup
consisting of all $\sigma \in G$ with $ \sigma(\primid)=\primid $
and such that the induced morphism $\overline{\sigma}:A/\primid \to A/\primid $ is the identity.

\begin{emp} \label{ramifiedprecisely} We say that the $G$-action is \emph{ramified precisely at the origin} if
$I_{\maxid_A}=G$ and  $I_\primid=\{{\rm id} \}$ for every non-maximal prime ideal $\primid$.
\end{emp}

We discuss in this section several algebraic properties of the ring  of invariants $A^G$ and of the  field extension $\Frac(A)/\Frac(R)$
induced by a norm subring $R\subset A$.
Let us start by recalling the following well-known facts.

\begin{proposition}
\label{properties invariant ring} 
Suppose that the $G$-action on $A$ is  ramified precisely at the origin.
\begin{enumerate}
[\rm (i)]
\item The ring  of invariants $A^G$ is a complete     local noetherian domain that is normal, 
with $\dim(A^G)=n$ and $\depth(A^G)=2$.
\item The ring extension $A^G\subset A$ is   local, finite of  degree $p$, but not flat.
The induced map  on residue fields is an isomorphism.  
\item
The singular locus of $\Spec(A^G)$ consists of 
the closed point $\maxid_{A^G}$. The   morphism $\Spec(A)\setminus \{ \maxid_A\} \ra\Spec(A^G)\setminus \{ \maxid_{A^G}\}$ is a $G$-torsor.
\end{enumerate}
\end{proposition}

\proof It is standard that $A^G$ is an integrally closed domain, and that 
the ring extension $A^G\subset A$ is integral.
In light of the Going-Up Theorem, the ring $A^G$ must be local,
the ring extension $A^G\subset A$ is local, and $\dim(A^G)=\dim(A)$.

Let us show that $A^G$ is noetherian (see also \cite{Mol}, Th\'eor\`eme 2).
According to \ref{exist admissible}, there exists a regular system of parameters $u_1,\ldots,u_n\in A$
that is admissible with respect to the $G$-action. Let 
$R=k[[x_1,\ldots,x_n]]$ be the resulting norm subring. Then the
extension $R\subset A$ if finite of degree $p^n$ (\ref{norm elements}). It follows that $A^G$ is  also
a finitely generated $R$-module and, hence,  $A^G$ is
noetherian.

Now that we know that $A^G$ is noetherian,   Proposition 2 and Proposition 4 in \cite{Fog} show that  $A^G$ has depth $2$ (see also \cite{Ellingsrud; Skjelbred 1980}, 2.4,  or \cite{Peskin 1983}, Corollary 1.6).
Since $A$ is complete, the completion $B$ of $A^G$ with respect to its maximal ideal maps to $A$, and the image of $B$ in $A$ is $A^G$. 
Thus, $A^G$ is complete since it is the image of a complete ring. This completes the proof of (i).

By our overall assumption we have
$k\subset A^G$, so the inclusion $A^G\subset A$ has trivial residue field extension. 
If $A^G$ is regular, then $A $ would be flat over $A^G$ (\cite{Matsumura 1989}, Theorem 23.1). 
But then the Zariski--Nagata Theorem on the Purity of the Branch Locus (\cite{A-K}, Chapter VI, Theorem 6.8)
would imply that the branch locus is pure of codimension $1$, which would contradict our hypothesis that $A^G\subset A$ is ramified precisely at the origin.
Hence, the ring $A^G$ is not regular, and it follows from \cite{EGA}, IV.6.5.1 (i), that $A^G\subset A$ is not flat. This shows (ii).
\if false. Otherwise, let $L_n\ra \ldots L_0\ra k\ra 0$
be a resolution of the residue field with free $A^G$-modules of finite rank;
the kernel of $L_n\ra L_{n-1}$ would become free after tensoring with $A$, whence
must be free itself, contradicting that $A^G$ is not regular. This shows (ii).
\fi
Part (iii) follows from \cite{SGA 1}, Expos\'e V, Proposition 2.6.
\qed

\medskip
The following lemma will be used in \ref{proofstructure}. 

\begin{lemma}
\label{no invariant parameter} Suppose that the $G$-action on $A$ is ramified precisely at the origin.
Then no element $v\in\maxid_A\smallsetminus \maxid_A^2$ is $G$-invariant. In particular, the action of $G$ on $A$ is not linear.
\end{lemma}

\proof
Suppose that there exists $v\in\maxid_A\smallsetminus \maxid_A^2$ which is $G$-invariant.
Extend $v$ to a regular system of parameters $ u_1,\ldots,u_n$ of $ A$, with $u_1=v$.
Let $\sigma$ denote a generator of $G$, and
consider the ideal $\ideala\subset A$ generated by the elements $\sigma(u_i) - u_i$, $i=1,\dots,n$.
\if false
This defines the largest $G$-invariant closed subscheme of $\Spec(A)$ on which the $G$-action is trivial ($\Spec A/\ideala$ is called the 
the \emph{fixed scheme} of the action). From this we see that $\ideala$ does not depend
on the choice of the regular system of parameters. 
\fi 
Since any prime ideal $\mathfrak p$ of $A$ which contains $\ideala$ 
is such that $I_{\mathfrak p} \neq (0)$, we find that $V(\ideala) = \{ {\mathfrak m}_A\}$. In particular, $\dim(A/\ideala)=0$.
On the other hand, by construction, the ideal $\ideala$ is generated by at most $n-1$ elements because $\sigma(u_1)-u_1=0$.
Thus, since $\dim(A)=n$, we find that $\dim(A/\ideala)>0$, a contradiction.

\if false 
Recall that a  $G$-action   on $A=k[[u_1,\ldots,u_n]]$ is  linear if one generator $\sigma\in G$ acts via
substitution of variables $u_j\mapsto \sum\lambda_{ij}u_i$ for some
invertible matrix $(\lambda_{ij})\in\GL_n(k)$.
\fi

It follows that an action of $G$ on $A$ which is ramified precisely at the origin is \emph{never}
linear, because a matrix of order $p$ 
always has an eigenvector for the eigenvalue $\lambda=1$ in characteristic $p$. Note
that this   statement about non-linearity was noted already in \cite{Peskin 1983}, Proposition 2.1. 
\qed

\medskip
Our next theorem is used in \ref{artin case}.

\if false
\begin{theorem}  \label{galois extension}
Let $A$ be endowed with an action of $G$ which is ramified precisely at the origin. 
Choose a regular system of parameters $u_1,\ldots,u_n \in A$ which is admissible with respect to the action of $G$,
and consider the corresponding norm subring $R=k[[x_1,\ldots,x_n]]$. Then

\noindent {\rm (a)} The extension $\Frac(A^G)/ \Frac(R)$ is not purely inseparable. 

\smallskip
Assume that $k$ is such that there exist no separable extension $k'/k$ of degree at most $p$. Then

\smallskip
\noindent {\rm (b)} If $\Frac(A^G)/\Frac(R)$ is Galois, 
then the  extension $\Frac(A)/\Frac(R)$ is Galois.

\noindent {\rm (c)} If $\varphi: A^G \to A^G$ is an automorphism of $k$-algebras, then there exists  an automorphism $\varphi': A \to  A$ of $k$-algebras such that  $\varphi'_{\mid A^G}=\varphi$. 
\end{theorem}
\fi

\begin{theorem}  \label{galois extension}
Assume that the $G$-action on $A$ is ramified precisely at the origin. 
\begin{enumerate}
[\rm (i)]
\item
Choose a regular system of parameters $u_1,\ldots,u_n \in A$ which is admissible with respect to the $G$-action,
and consider the corresponding norm subring $R=k[[x_1,\ldots,x_n]]$. Then
the extension $\Frac(A^G)/ \Frac(R)$ is not purely inseparable. 
\item
Assume that the field $k$ has no Galois extension of degree $p$.
Then any $k$-automorphism $\varphi: A^G \to A^G$ extends to an automorphism $\varphi': A \to  A$.
Furthermore, if $A^G$ contains a subring $R$ such that the extension $R\subset A^G$ is finite and  $\Frac(A^G)/\Frac(R)$ is Galois, then
the  extension $\Frac(A)/\Frac(R)$ is Galois as well.
\end{enumerate}
\end{theorem}

\proof
(i) Consider the local schemes
$
X:=\Spec(A)$,  $Y:=\Spec(A^G)$, and $ S:=\Spec(R)$,
and write $U\subset X$, $V\subset Y$, and $W\subset S$ for the complements of their closed points.
Suppose  that $\Frac(R)\subset \Frac(A^G)$ is   purely 
inseparable, so that the finite morphism $Y\ra S$ is 
a universal homeomorphism. 
According to \cite{SGA 1}, Expos\'e IX, Theorem 4.10, the \'etale covering $U\ra V$ would be the base-change
of an \'etale covering $W'\ra W$. The latter is the restriction of some \'etale covering $S'\ra S$
by the Zariski--Nagata Purity Theorem. Consequently, $U\ra V$ is the restriction of some \'etale
covering $Y'\ra Y$. Since the schemes $X$ and $Y'$ are normal, 
both restriction maps $\Gamma(Y',\O_{Y'})\ra\Gamma(U,\O_U)\leftarrow\Gamma(X,\O_X)$
are bijective, whence $Y'=X$, contradicting the fact that $X\ra Y$ is not \'etale.

(ii)
In order to extend $\varphi:A^G\ra A^G$ to $A$, let
$\tau:V\ra V$ be the morphism induced by $\varphi$.
Define the scheme $\tau^*(U)$ by requiring that the following diagram be Cartesian: 
$$
\begin{CD}
\tau^*(U) @>>> U\\
@VVV @VVV\\
V @>>\tau> V.
\end{CD}
$$
Our task is to show 
the existence of a $V$-isomorphism  $U\to \tau^*(U)$. Composing this isomorphism 
with the given morphism $\tau^*(U) \to U$ produces an  extension of the morphism $\tau$,
giving the desired automorphism $\varphi'$ of $A=\Gamma(U,\O_U)$.

Choose a separable closure   $\Frac(A^G)\subset\Omega$. In turn, we get  a geometric point $b:\Spec(\Omega)\ra V$. 
Recall that the algebraic fundamental group $\pi_1(V,b)$ is defined as the group of automorphisms for the fiber functor
$(\Cov/V)\ra(\Set)$ that sends a finite \'etale $V'\ra V$
to the underlying set of the base-change $V'\otimes_V\Omega$. Here $(\Cov/V)$ denotes the category of all finite \'etale $V$-schemes,
which thus becomes equivalent to the category of finite sets endowed with
a continuous action of $\pi_1(V,b)$ (\cite{SGA 1}, Expos\'e V, Section 7).
By abuse of notation, we also write $\pi_1(V,\Omega)$ for the fundamental group with respect to the base-point $b:\Spec(\Omega)\ra V$.

Up to isomorphism, the  two finite \'etale $G$-coverings $U_1=U$ and $U_2=\tau^*(U)$ correspond   
to the finite set $G$, endowed with  actions via 
group homomorphisms $h_i:\pi_1(V,\Omega)\ra G$ for $i=1,2$. Let $H_i\subset\pi_1(V,\Omega)$ be their kernels.
The task is to show that the two actions on $G$ are isomorphic, in other words, that $H_1=H_2$.
To proceed, choose a lifting of the geometric point $b:\Spec(\Omega)\ra V$ along $U_i\ra V$.
This gives an identification $H_i=\pi_1(U_i,\Omega)$.
Clearly, $\Gamma(U_i,\O_{U_i})$ is a complete local noetherian ring that is regular. Let $X_i$ be its spectrum.
The commutative diagram
$$
\begin{CD}
U_i	@>>>	V\\
@VVV		@VVV\\
X_i	@>>>	\Spec(k)
\end{CD}
$$
induces a commutative diagram
of fundamental groups 
$$
\begin{CD}
\pi_1(U_i,\Omega)	@>>>	\pi_1(V,\Omega)\\
@VVV			@VVV\\
\pi_1(X_i,\Omega)	@>>>	\pi_1(k,\Omega).
\end{CD}
$$
The map on the left is bijective, because $X_i$ is regular and $X_i\smallsetminus U_i$ has codimension two (\cite{SGA 1}, Expos\'e X, Corollary 3.3).
The lower map is bijective as well, because $X_i$ is local henselian (\cite{SGA 1}, Expos\'e X, Th\'eor\`eme 2.1).
It follows that the kernel $N\subset \pi_1(V,\Omega)$ for the map on the right is isomorphic to $G$,
and that both $H_i=\pi_1(U_i,\Omega)$ are sections for   $\pi_1(V,\Omega)\ra\pi_1(k,\Omega)$.
Any two   sections differ by a homomorphism $\pi_1(k,\Omega)\ra N$. The latter must be zero, by assumption on
the field $k$. In turn, $H_1=H_2$.

Now suppose that $R\subset A^G$ is finite and $\Frac(R)\subset\Frac(A^G)$ is Galois, say of degree $d\geq 1$.
To check that $\Frac(R)\subset\Frac(A)$ is Galois, it suffices to extend
each $\Frac(R)$-automorphism $\varphi:\Frac(A^G)\ra\Frac(A^G)$ to an automorphism $\varphi':\Frac(A)\ra\Frac(A)$.
Indeed, the group $H$ of automorphisms of $\Frac(A)$ over $\Frac(R)$ then contains
at least $dp$ elements, obtained as products of extension $\varphi'$ and $\sigma\in G$.
Since $[\Frac(A):\Frac(R)]=dp$, we find that $\Frac(R)\subset\Frac(A)$ is Galois.
The desired extensions $\varphi'$ do exist, because then the morphism $\varphi$ extends to the 
integral closure $A^G$ of $R$ with respect to $R\subset A^G$, and thus to $A$.
\qed

\begin{remark} (i) 
We do not know whether the conclusion of Theorem \ref{galois extension} (ii) holds without 
the additional hypothesis on the field of representatives $k$.

(ii) We do not know of an example of an action of $G={\mathbb Z}/p{\mathbb Z}$ on $A=k[[u_1,\ldots,u_n]]$  ramified precisely at the origin
where $\Frac(A^G)/\Frac(R)$ is not separable. 
\if false
We  present in \cite{Lorenzini Schroeer 2} 
an example inspired by the work of Peskin \cite{Peskin 1983} of an action ramified precisely at the origin 
where $\Frac(A^G)/\Frac(R)$ is separable, but not Galois.
\fi 
When the action is not ramified precisely at the origin, it is easy to find examples where $\Frac(A^G)/\Frac(R)$ is not separable. Indeed, 
if $\sigma(u_i)=u_i$, then $u_i \in A^G$ and $x_i:={\rm Norm}(u_i)=u_i^p \in R$.

\if false
(iv) The condition that $\Frac(A)/\Frac(R)$ is separable depends on the choice of the admissible regular system of parameters
$u_1,\ldots,u_n\in A$, at least when the action is not ramified only at the origin. 
Here I do not see anymore why I wrote that. It seems to me that the following is not clear:
Suppose that the extension $\Frac(A^G)/\Frac(R)$ is purely inseparable. So we know that there is ramification somewhere, 
so the ideal of the fixed scheme is not primary to the maximal ideal. Does it follow that there exists $u$ in a regular system of parameters of $A$ 
such that $\sigma(u)=u$?
\fi
\end{remark}

\if false
\begin{example} Let $k$ be of characteristic $2$.
We exhibit in this example  a faithful action of $G:={\mathbb Z}/4{\mathbb Z}$ on $A:=k[[u,v]]$
where the extension $\Frac(A^G)/\Frac(R)$ is not separable. 
In the ring $k[[v,y]][U]$, let 
$$
g:=U(U+v)(U+y)(U+v+y)-y,
$$ 
and consider the ring $A:=k[[v,y]][U]/(g)$. Let $u$ denote the class of $U$ in $A$. Then 
$\sigma: A \to A$, with $\sigma(u)=u+v$ and $\sigma(v)=v+y$, is a  $k$-automorphism of $A$ of order $4$. We let $G=\langle \sigma \rangle$.
The elements $u,v$ form a regular system of parameters, so that $A$ is isomorphic to $k[[u,v]]$. 
\if false
The ideal of the fixed point scheme is $(\sigma(u)-u, \sigma(v)-v)=(v,y)$, and since this ideal is primary to the maximal ideal $(u,v,y)$, the action is ramified precisely at the origin. We note in passing that the action of $\langle \sigma^2\rangle$
is not ramified precisely at the origin.
\fi

By construction, we have $N_{A/A^G}(u)=y$. We set $x:=N_{A/A^G}(v)=v^2(v+y)^2$, and $R:=k[[x,y]]$.
A simple computation shows that $u,v$ is an admissible system of parameters of $A$.
It follows from Proposition \ref{norm elements} that 
$[\Frac(A):\Frac(R)]=16$.
Clearly, $A^G$ contains $v(v+y)$ and $y$. Thus, we find that $\Frac(A^G)$ contains the root $v(v+y)$  of the inseparable irreducible polynomial $z^2-x \in \Frac(R)[z]$.
The element $t:= u^2y + uy^2 + (v^3 +  vy^2)$ is also invariant, and thus $\Frac(R)(v(v+y),t) = \Frac( A^G)$.
For completeness, we note that $t$ satisfies the equation $t^2+yv(v+y)t+(v(v+y))^3+y^3=0$, and that the inertia group of the prime ideal $(u)$ contains $\sigma^2$.
\end{example}
\fi 

\if false
We consider in section \ref{moderately ramified} a class of $G$-actions where the extension 
$\Frac(A)$ is Galois over $k((x_1,\ldots,x_n))$.
This hypothesis alone however does not seem  to be strong enough to allow for a nice 
classification of the   rings of invariants $A^G$. To be able to state 
a further hypothesis leading to a nice classification, 
we introduce in the next section the notion of effective model of a  group action.
\fi

\section{The effective model of a group action}
\label{extensions torsors}

Let $R:=k[[x_1,\dots, x_n]]$, and let $L/\Frac(R)$ be a Galois extension of degree $p$. Let $B$ denote the integral closure of $R$ in $L$.
It is natural to wonder what conditions can be imposed on the Galois extension  $L/\Frac(R)$ to guarantee that the ring $B$ is simple,
say $B\simeq R[u]/(f(u))$ for some easily described $f(u) \in R[u]$. To give an answer to this question that will be useful in \ref{MR4} in our study
of normal forms of ${\mathbb Z}/p{\mathbb Z}$-actions on $A=k[[u_1,\dots, u_n]]$,  we introduce below 
the notion of effective model of a group action. 

Let $G$ denote the Galois group of $L/\Frac(R)$. Clearly, $G$ also acts on $B$. Letting $S=\Spec R$ and $Y:=\Spec B$,  we obtain an action of the constant group scheme $G_S/S$ on $Y/S$. Assume now that $Y/S$ is flat. Then, associated with this action on $Y/S$ is a uniquely defined group scheme ${\mathcal G}/S$
with an $S$-action ${\mathcal G} \times_S Y \to Y$. The group scheme ${\mathcal G}/S$ is called {\it the effective model of the action of $G$ on $Y$}.
It was considered first by Raynaud in \cite{Ray1999}, Proposition 1.2.1, in cases where $R$ is a discrete valuation ring. This theory was extended, still in the case where $R$ is a discrete valuation ring, to cases where $G_S$ is not constant, and even not finite over $S$, in \cite{Rom2006} and \cite{Romagny 2012}. 
In the case where $G_S$ is constant, some cases where the base scheme $S$ is not of dimension $1$ are considered in \cite{Abr}, 2.2. 
The main result in this section is Theorem \ref{effective model} below, which proves the existence of the effective model when $G_S/S$ is constant but the base $S$ is not necessarily of dimension 1. 

Fix a noetherian base scheme $S$ and some prime number $p>0$.
Let $Y \to S$ be a finite flat morphism of degree $p$. Denote by ${\rm Aut}(Y/S)$ the group of $S$-automorphisms of $Y$.
Consider the group-valued functor $\Aut_{Y/S}$, with $\Aut_{Y/S}(T):=\Aut(Y_T/T)$ for any $S$-scheme $T$.
Using the existence of Hilbert schemes 
(\cite{FGA IV}, 221-19), such functor is shown to be representable, by a scheme denoted ${\rm Aut}_{Y/S}/ S$. In our situation, 
we will need to use the fact that the structure morphism ${\rm Aut}_{Y/S}\ra S$ is affine. 
For convenience, we recall a proof of this fact below. This proof can be easily modified to 
also give a proof of the existence of ${\rm Aut}_{Y/S}/ S$.

\begin{lemma} 
\label{lemma.affine}
The structure morphism ${\rm Aut}_{Y/S}\ra S$ is affine.
\end{lemma}

\proof
The question is local in $S$, so it suffices to treat the case where the schemes  $S=\Spec(R)$ and $Y=\Spec(B)$ are  affine.
Furthermore, we may assume that  $B$ admits a basis $b_1,\ldots,b_p\in B$ as $R$-module, and  that $b_1=1_B$.
Let $\underline{\GL}(B)=\underline{\GL}_{p,R}$ be the group scheme of $R$-linear automorphisms.
Its underlying scheme is the spectrum of the localization $R[T_{ij}]_{\det}$, where $T_{11},\ldots, T_{pp}$ are $p^2$ indeterminates
and $\det=\det(T_{ij})$ is the determinant of the matrix $(T_{ij})_{1 \leq i,j\leq p}$.
We have a canonical monomorphism ${\rm Aut}_{Y/S}\subset\underline{\GL}_{p,R}$.
To show that the scheme ${\rm Aut}_{Y/S}\ra S$ is affine, 
it suffices to show that the monomorphism is a closed embedding. 

Let $\tau=(\tau_{ij})$ be an $R$-linear map $B\ra B$.
Write  $\mu:B\otimes B\ra B$ for  the algebra multiplication, 
with 
$\mu(b_i \otimes b_j)=\sum \mu_{kij} b_k$.
Then the linear map $\tau $ is an algebra homomorphism if and only 
if $\mu \circ (\tau\otimes\tau) =\tau\circ\mu$ and  $\tau(b_1)=b_1$.
The latter means $\tau_{21}=\ldots=\tau_{p1}=0$. 
As to the former condition, we have
$
\left(\mu\circ(\tau\otimes\tau)\right)(b_i\otimes b_j)= \sum_k (\sum_{r,s}\tau_{ri}\tau_{sj}\mu_{krs}) b_k
$
and 
$
\left(\tau\circ\mu\right)(b_i\otimes b_j) = \sum_k (\sum_t\mu_{ijt}\tau_{kt})b_k$.
Comparing coefficients gives $\sum_{r,s}\tau_{ri}\tau_{sj}\mu_{krs}=\sum_t\mu_{ijt}\tau_{kt}$
for all $1\leq k\leq n$. Replacing the scalars $\tau_{ri}$ by the indeterminates $T_{ij}$,
we obtain equations that define a  closed subscheme inside $\underline{\GL}_{p,R}$.
Applying the above computations over arbitrary $R$-algebras $R'$, we infer that
this closed subscheme represents the functor ${\rm Aut}_{Y/S}$.
\qed

\medskip
Let now $G$ be an abstract group of order $p$ acting on $Y$ via $S$-automorphisms. In other words,
$G$ is endowed with a group homomorphism $G \to {\rm Aut}_S(Y/S)$.
The quotient scheme $Y/G $ exist, and {\it we assume that the structure map $Y/G \to S$ is an isomorphism}.
Write  $G_S/S$ for the constant group scheme associated to $G$. The action of $G$ on $Y/S$ yields a natural   homomorphism of group schemes \begin{equation} 
\label{action}
f:G_S\ra\Aut_{Y/S}.
\end{equation}
Since $\Aut_{X/S}\ra S$ is separated and $G_S\to S$ is proper,
the morphism $G_S\ra \Aut_{Y/S}$ is proper, and its schematic image $f(G_S)\subset \Aut_{Y/S}$
is finite over $S$ (use \cite{EGA}, Chapter II, Corollary 5.4.3).

Let $s \in S$ be any point, and consider the natural morphism $\Spec \O_{S,s} \to S$. 
We denote by $f(G_S)\otimes\O_{S,s}$ the base change of $f(G_S) \to S$ by   $\Spec \O_{S,s} \to S$.
Both the construction of $\Aut_{Y/S}$ and of the schematic closure
of $f(G_S)$ commute with the base change by $\Spec \O_{S,s} \to S$.
Let now $s \in S$ be a regular point of codimension $1$, so that $\O_{S,s}$ is a discrete valuation ring. We are in a position 
to apply a theorem of  Romagny (\cite{Romagny 2012}, Theorem 4.3.4)
to obtain that the
scheme  $f(G_S)\otimes\O_{S,s}$ is in fact a subgroup scheme of the group scheme $ \Aut_{Y/S}\otimes\O_{S,s}$. The group scheme $f(G_S)\otimes\O_{S,s}$ 
over $\Spec \O_{S,s}$
is called the \emph{effective model} for the $G$-action at the 
point $s\in S$. Our next theorem slightly extends this result. Note that Theorem 4.3.4 in \cite{Romagny 2012} is stated for  group schemes $\mathcal G\to \Spec \O_{S,s}$ that need not be finite. The hypothesis in Theorem 4.3.4 that $\mathcal G/ \Spec \O_{S,s}$ is universally affinely dominant (see \cite{Romagny 2012}, 3.1.1) is automatically satisfied in our situation where we require $ \mathcal G\to \Spec \O_{S,s}$ to be finite and flat. The same is true for the hypothesis that $\mathcal G\to \Spec \O_{S,s}$ is pure (see \cite{Romagny 2012}, 2.1.1 and 2.1.3). 

Let $\zeta\in\ZZ_p$ denote a primitive $(p-1)$-th root of unity,
and define  $\Lambda_p$ to be the subring $\ZZ[\zeta,1/p(p-1)] \cap \ZZ_p$ of $\QQ_p$.
Note that all schemes $S \to \Spec \FF_p$ of characteristic $p$ are in fact schemes over $\Spec \Lambda_p$
in a unique way, after composing with the natural homomorphism $\Lambda_p\subset\ZZ_p\ra\FF_p$.

\begin{theorem}
\label{effective model} Let $p>0$ be prime. Let $S$ be a noetherian locally factorial scheme over $  \Lambda_p$. Keep the above notation. In particular,
let $Y \to S$ be a finite flat morphism of degree $p$. Let $G$ be an abstract group of order $p$ acting on $Y$ via $S$-automorphisms.
Then there is a finite flat $S$-group scheme $\shG$ of degree $p$ and
a $S$-homomorphism $h:\shG\ra\Aut_{Y/S}$ such that, given any point $s \in S$ of codimension at most $1$, 
the base change of $h$ over $\Spec(\O_{S,s})$, $\shG\otimes\O_{S,s}\ra \Aut_{Y/S} \otimes\O_{S,s}$, induces   an isomorphism of $\Spec(\O_{S,s})$-group schemes
$$
\shG\otimes\O_{S,s}\lra f(G_S)\otimes\O_{S,s}.
$$
The pair $(\shG,h)$ is unique up to unique isomorphism.
Moreover, $Y/S$ is a torsor for the $\shG$-action if  and only if 
the fiber $Y\otimes k(s) $ is a torsor for the action of $\shG\otimes\kappa(s)$ for each point $s\in S$
of codimension at most $ 1$.
\end{theorem}

\medskip
Extending Romagny's terminology, we call the finite flat $S$-group scheme $\shG$, together with its action on $Y$,
the \emph{effective model} for the $G$-action on $Y$. The proof of Theorem \ref{effective model} is postponed to \ref{proof.effective model}.
We start with several preliminary propositions.

\begin{lemma} \label{affine} Let $S$ be a noetherian scheme.
Let  $U\subset S$ be a dense open subscheme 
which contains all points of $S$ which do not satisfy Serre's Condition 
$(S_1)$. Let $G_1\to S$ and $G_2\to S$ be two 
finite and flat 
morphisms. Let $\varphi: G_1 \to G_2$ and $\psi:G_1 \to G_2$ be two $S$-morphisms
such that $\varphi_{\mid U}$ and $\psi_{\mid U}$ are equal as morphisms from $G_1 \times_S U$ to $G_2 \times_S U$. 
Then $\varphi=\psi$.
\end{lemma}
\proof The equality can be checked locally on the base. Thus we may assume that $S=\Spec R$ is affine, in which case
$G_i=\Spec(A_i)$ for some $R$-algebra $A_i$ which are free of finite rank as $R$-module, and the two given morphisms corresponds to morphisms $\varphi^*,\psi^*: A_2 \to A_1$ of $R$-algebras.
Since $S$ is noetherian, the set $\Ass(R)$ is finite. Thus the Prime Avoidance Lemma lets us  find $f \in R$ such that $\Ass(R)$ is contained in the special open set $D(f)$ of $S$, and such that $D(f) \subseteq U$. 
Then $f\in R$ is regular  
(\cite{EGA}, Chapter IV, Corollary 3.1.9). Since $A_1/R$ and $A_2/R$ are free,   the localization map
$$
\Hom_R(A_2,A_1)\lra \Hom_R(A_2,A_1)_f=\Hom_{R_f}((A_2)_f, (A_1)_f)
$$
is injective. It follows that $\varphi^*=\psi^*$.
\qed

\smallskip
Let $(\FFGrp/S)$ denote the category of finite flat group schemes over $S$.
\begin{proposition}
\label{full and faithful}
Let $S$ be a noetherian scheme. Let  $U\subset S$ be a dense open subscheme and 
consider the restriction functor $(\FFGrp/S)\ra(\FFGrp/U)$.
\begin{enumerate}[\rm (i)]
\item 
If $U$ contains 
all points of $S$ which do not satisfy Serre's Condition 
$(S_1)$, then the restriction functor is faithful.
\item 
If $U$ contains all points of $S$ which have codimension at most $ 1$ and all points which
do not satisfy Serre's condition $(S_2)$, then the restriction functor is full.
\end{enumerate}
\end{proposition}

\proof
(i) Follows immediately from \ref{affine}.
\if false
  Let $\varphi:G_1\ra G_2$ be a homomorphism between  finite flat group schemes
over $S$ so that $\varphi_U$ is the zero map.
We have to check that $\varphi$ is the zero map.
This is a local problem, so we may assume that $S=\Spec(R)$ is a local scheme.
Then $G_i=\Spec(A_i)$ for some Hopf $R$-algebra $A_i$ such that the underlying $R$-modules are free
of finite rank. The morphism $\varphi$ corresponds to an $R$-linear Hopf algebra map $A_2\ra A_1$.
Choose an element $f\in R$ so that $\Spec(R_f)\subset S$ is contained in $U$ and contains the finite set $\Ass(\O_S)\subset U$.
Then $f\in A$ is regular  
(\cite{EGA}, Chapter IV, Corollary 3.1.9), whence the localization map
$$
\Hom_R(A_2,A_1)\lra \Hom_R(A_2,A_1)_f=\Hom_{R_f}((A_2)_f, (A_1)_f)
$$
is injective. In turn, $\varphi$ is determined by $\varphi_U$.
\fi

(ii) Let $G_1,G_2$ be two finite flat group schemes
over $S$, and $\varphi_U:(G_1)_U\ra (G_2)_U$ be a homomorphism over $U$. We have to extend it to $S$.
Since $S$ is noetherian, there is a maximal open
subset over which $\varphi_U$ extends. 
It suffices to treat the case where $U\subset S$ itself is maximal.
Seeking a contradiction, we assume $U\neq S$  and choose
a generic point $s$ in the closed subset $S\smallsetminus U$.
One easily sees that if $(\varphi_U)_{U\cap\Spec(\O_{S,s})}$ extends over $\Spec(\O_{S,s})$, then it extends
to some $\varphi_V$ over some open neighborhood $V$ of $s\in S$.
Using (i), we may shrink $V$ so that  $\varphi_U$ and $\varphi_V$ coincide on the overlap $U\cap V$.
Hence, $\varphi_U$ extends to $U\cup V$, which is a contradiction.
This reduces us to the case where $S=\Spec(R)$ is local,
and $U\subset S$ is the complement of the closed point $s\in S$.
Then $G_i=\Spec(A_i)$ for some Hopf $R$-algebra $A_i$, such that the underlying $R$-modules are free
of finite rank. The morphism $\varphi$ that we seek corresponds to an $R$-linear Hopf algebra map $A_2\ra A_1$.
Consider the free $R$-module $M=\Hom_R(A_2,A_1)$ of finite rank
and the corresponding coherent sheaf $\shM$.
The long exact sequence of local cohomology yields
\begin{equation}
\label{local cohomology}
\Gamma_s(S,\shM)\lra \Gamma(S,\shM)\lra \Gamma(U,\shM)\lra H^1_s(S,\shM).
\end{equation}
By assumption, the local ring $R$ has   depth $\geq 2$.
Whence the outer terms vanish, and the homomorphism of Hopf algebras
$\varphi_U$ extends to a homomorphism of $R$-modules $\varphi:A_2\ra A_1$.
Arguing in the same way as for  Lemma \ref{affine}, one sees that the morphism $\varphi$
is compatible with the Hopf  structures.
\qed

\begin{proposition}
\label{extension torsors} Let $S$ be a noetherian scheme. Let  $\shG$ be a finite flat $S$-group scheme, and let $Y$ be a finite flat $S$-scheme.
Assume that there exists an open subscheme $U$ of $S$ such that $Y_U$ is endowed with a $U$-group scheme action 
 $\mu_U: \shG_U\times_U Y_U\ra Y_U$ of $\shG_U$.
If $U $ contains all points of $ S$ that are of codimension at most $1$ and all points of $S$  which do not satisfy Serre's Condition $(S_2)$,
then the given $U$-group scheme action $\mu_U $ extends to a unique $S$-group scheme action $\mu: \shG\times_S Y\ra Y$ of 
$\shG$ on $Y$. 
\end{proposition}

\proof
As in the proof for Proposition \ref{full and faithful}, it suffices to treat the case
that  $S=\Spec(R)$ is local, and $U\subset S$ is the complement of the closed point.
Then  $Y=\Spec(B)$ and $\shG=\Spec(A)$ are given
by $R$-algebras that are free of finite rank as $R$-modules.
The desired group scheme action $\mu:\shG\times Y\ra Y$ corresponds
to a linear map $B\ra A\otimes_R B$ satisfying certain axioms.
Consider the free $R$-module $M=\Hom_R(B,A\otimes_RB)$; arguing with
a short exact sequence in local cohomology like (\ref{local cohomology}), 
we see that the desired homomorphism exists.
The uniqueness of the extension follows from Lemma \ref{affine}.
\qed

\begin{proposition} \label{torsors} 
Let $S$ be a noetherian scheme. Let  $\shG$ be a finite flat $S$-group scheme, and let $Y$ be a finite flat $S$-scheme endowed with 
 a $S$-group scheme action $\mu: \shG\times_S Y\ra Y$.
If for each point $s\in S$ of codimension at most $ 1$, the fiber $Y\otimes\kappa(s)$ is a  torsor over $\shG\otimes\kappa(s)$, then $Y$ is a $\shG$-torsor.
\end{proposition}

\proof
We have to show that the canonical map $\mu\times \pr_2:\shG\times_S Y\ra Y\times_S Y$ is
an isomorphism.
The question is local, so we may assume that $S=\Spec(R)$ is the spectrum of a local ring,
and write $Y=\Spec(B)$ and $\shG=\Spec(A)$.  Let $f:B\otimes_R B\ra A\otimes_R B$ be the resulting
map between free modules of finite rank,
and  consider its determinant $\det(f)\in R$. Suppose the latter is not a unit,
and choose a minimal prime ideal $\primid\subset R$ annihilating $R/\det(f)R$.
Replacing $R$ by the localization  $R_\primid$, we arrive at the situation 
that $\det(f)$ is invertible precisely outside the closed point $s\in S=\Spec(R)$.
By Krull's Principal Ideal Theorem, the point $s\in S$ is of codimension at most $ 1$.
By assumption, the class of  $\det(f)$ in $\kappa(s)$ is nonzero, contradiction.
\qed

\begin{remark} The above propositions might be known to the experts, but we did not find an appropriate reference for them in the literature. 
Lemme 2 in \cite{Mor}, stated without proof, asserts that the restriction functor induces an equivalence of categories between the category of torsors
under $\shG/S$ and the category of torsors under $\shG_U/U$, for any dense open set $U$ of a regular noetherian scheme $S$ which contains all points of codimension 1 in $S$. This statement is proved in \cite{Mar}, 3.1.
\end{remark}

\begin{emp} \label{FFG}
Denote by $(\FFGrp^p/S)$ the full subcategory of $(\FFGrp/S)$ consisting of all finite flat $S$-group schemes
of  degree $p$. According to \cite{Tate; Oort 1970}, Theorem 1,
such  group schemes $G/S$ take as values commutative groups annihilated by $p$.
To apply further results of \cite{Tate; Oort 1970}, we assume now 
that the base scheme $S$ is a scheme over $\Spec \Lambda_p$.

Certain elements  $w_i\in\Lambda_p$, $1\leq i\leq p$  are introduced
in \cite{Tate; Oort 1970} on page 9,  with the property that $w_i\cong i!$ modulo $p$.
When $S \to \Spec \Lambda_p$ is given, we also denote by $w_i$ the image of $w_i$ in $\Gamma(S,\O_S)$.
Consider now triples $(\shL,\alpha,\beta)$, where $\shL$ is an invertible sheaf on the scheme $S$, endowed with global sections
$\alpha\in\Gamma(S,\shL^{\otimes (p-1)})$ and  $ \beta\in\Gamma(S,\shL^{\otimes(1-p)})$
 so that $\alpha\otimes \beta = w_p$. Here the equality is obtained after identifying in the natural way $\shL^{\otimes (p-1)} \otimes \shL^{\otimes(1-p)}$
with $\O_S$.

An abelian sheaf $\shG_{\alpha,\beta}^\shL$,  annihilated by $p$, is associated to such triple $(\shL,\alpha,\beta)$ as follows.  
For any $T/S$, let 
$$
\shG_{\alpha,\beta}^\shL(T):=\left\{x\in \Gamma(T,\shL\otimes_{\O_S}\O_T) \mid x^{\otimes p}=\alpha\otimes x\right\},
$$
and the group law is given by the formula
$$
x\star x' := x+x'+\frac{\beta}{w_{p-1}}D_p(x\otimes 1, 1\otimes x'),
$$
where $D_p$ is a certain polynomial in two variables described in \cite{Tate; Oort 1970}, page 14.
It satisfies $D_p(X_1,X_2)\cong \sum_{i=1}^{p-1} \frac{(p-1)!}{i! (p-i)!} X_1^iX_2^{p-i}$ modulo $p$.
The sheaf $\shG_{\alpha,\beta}^\shL$ is representable by a finite flat $S$-group scheme of degree $p$,
denoted by the same letter $\shG_{\alpha,\beta}^\shL $. 
We can   assemble a category $(\Trp/S)$ whose objects are   triples $(\shL,\alpha,\beta)$ as above,
and whose morphisms are homomorphisms between invertible sheaves respecting the sections.
According to \cite{Tate; Oort 1970}, Theorem 2, we have:
\end{emp}
\begin{proposition}
\label{tate--oort}
The functor $(\Trp/S)\ra (\FFGrp^p/S)$, given by  
$(\shL,\alpha,\beta)\mapsto \shG_{\alpha,\beta}^\shL$, is essentially surjective.
\end{proposition}

From this one deduces:

\begin{proposition}
\label{extension group schemes}
Suppose that   $S$ is a locally factorial scheme over $ \Lambda_p$, and let  $U\subset S$ be an open subscheme
containing  all points of codimension at most $ 1$. Then the restriction functor $(\FFGrp^p/S)\ra(\FFGrp^p/U)$ 
is an equivalence of categories.
\end{proposition}
  
\proof
The restriction functor is fully faithful by Proposition \ref{full and faithful},
and we have to check that it is essentially surjective.
In light of Proposition \ref{tate--oort}, it suffices to extend an invertible sheaf $\shL_U$
and the sections $\alpha_U,\beta_U\in \Gamma(U,\shL_U)$ corresponding to a triple in $(\Trp/U)$ to a triple in $(\Trp/S)$.
Write $\shL_U=\O_U(D_U)$ for some Cartier divisor $D_U$. Since $S$ is locally factorial, we can view $D_U$ as a Weil divisor and
write it as a difference
$D_U=A_U-B_U$ of effective Weil divisors. Denote by $A$ and $B$ the closures in $S$ of $A_U$ and $B_U$.
These Weil divisors over $S$ correspond to invertible sheaves, giving the desired extension
of the invertible sheaf $\O_U(D_U)$. The sections are extended as in the proof of Proposition \ref{full and faithful}.
\qed


\noindent
\begin{emp} \label{proof.effective model}
{\it Proof of Theorem \ref{effective model}.}
Since  $S$ is a disjoint union of integral schemes, it suffices to prove the theorem when $S$ is irreducible. 
Let $f: G_S \to {\rm Aut}_{Y/S}$ be the morphism introduced in \eqref{action}. Let $f(G_S) \subseteq {\rm Aut}_{Y/S}$ denote the schematic image, 
which is finite   over $S$.
According to 
\cite{EGA}, Chapter IV, Theorem 11.1.1, the set $W\subset f(G_S)$ of all points 
where the  morphism $f(G_S) \to S$ is flat  is open.
Let $V'$ be the complement in $S$ of the image under $f$ of the closed set $f(G_S)\smallsetminus W$. Since $f$ is finite, $V'$ is open.
Then the induced map $f(G_S)_{V'}\ra V'$ is flat.
 Consider now
$$
V:=\left\{s\in S\mid 
\text{$f(G_S)\otimes_{\O_S}\O_{S,s} \subset\Aut_{Y/S}\otimes_{\O_S}\O_{S,s}$ is a  subgroup scheme}\right\}.
$$
We claim that $V$ is open in $S$, and that $f(G_S)_V $ is a subgroup scheme of $(\Aut_{Y/S})_V$.
This is a local question, and  we may assume that $S=\Spec(R)$ is affine,
with $f(G_S)=\Spec(A)$ and $\Aut_{Y/S}=\Spec(B)$ (we use here the fact noted in \ref{lemma.affine} that $\Aut_{Y/S} \to S$ is affine). 
Write $A=B/I$, and let $\primid\subset R$ be the prime ideal corresponding to a point $s\in V$.
The  comultiplication map $B\ra B\otimes B\ra A\otimes A$ factorizes over $A=B/I$ when
localized at $\primid$. Since $I$ is finitely generated, there is an element $a\in R\smallsetminus \primid$
so that the map factorizes when inverting $a$.
Then $\Spec(R_a)\subset S$ defines an open neighborhood over which
the group law for $\Aut_{Y/S}$ defines a composition $f(G_S)\times f(G_S)\ra f(G_S)$.
In a similar way one construct an open neighborhood over which the inversion
map for $\Aut_{Y/S}$ maps $f(G_S)$ to itself.
Summing up, $V\subset S$ is open.

Now consider the open subset $U:=V\cap V'$ of $S$.
Define $\shG_U:=f(G_S)_U$. Then the structure morphism $\shG_U\ra U$ is finite and flat,
and the inclusion $\shG_U\subset(\Aut_{Y/S})_U$ is a   subgroup scheme, necessarily closed.
We claim that $\shG_U \to U$ has degree $p$.
Since $U$ is integral, it suffices to check this over the generic point $\eta\in U$.
By construction of the effective model,  $\shG_\eta=G$ is the constant group of order $p$.

According to Romagny's result (\cite{Romagny 2012}, Theorem A), $U$ contains each point $s\in S$
of codimension one.
Using Proposition \ref{extension group schemes}, we   extend
$\shG_U$ to a finite flat group scheme $\shG$ over $S$.  
The canonical inclusion $\shG_U\subset\Aut{Y/S}_U$ gives
an action on $Y_U$. By Proposition \ref{extension torsors}, this action extends to
an action of $\shG$ on $Y$.
This shows the existence of $\shG$ and $h:\shG\ra\Aut_{Y/S}$.

The uniqueness also follows from Proposition \ref{extension group schemes} and
\ref{extension torsors}, because for any other $\shG'$ and $h':\shG\ra\Aut_{Y/S}$ as in 
the assertion, there is an open subset $U_0\subset U$ containing all points of codimension one
with $\shG'|U_0=f(G_S)|U_0$.
If the  fibers $Y_s$ are $\shG_s$-torsors for each point $s\in S$ of codimension one, then
$Y$ is a $\shG$-torsor, by Proposition \ref{extension torsors}.
\qed
\end{emp}

\medskip
Now let $R$ be a noetherian domain of characteristic $p>0$ and write $S:=\Spec(R)$.
Fix some elements $a,b\in R$ with $a\neq 0$, and consider the $R$-algebra 
$$
B:=R[u]/(u^p-a^{p-1}u-b),
$$
endowed with the automorphism $u\mapsto u+a$ of order $p$. This automorphism 
induces an action of $G=\ZZ/p\ZZ$
on the $S$-scheme $Y:=\Spec(B)$. 
Let $\shG/S$ denote the group scheme $\shG_{a,0}=\Spec R[z]/(z^p-a^{p-1}z)$ in the Tate--Oort classification.

\begin{lemma} \label{naturalaction} There is a natural $S$-action $\shG \times_S Y \to Y$ of $\shG/S$ on  $Y/S$ such that $Y/S$ is a torsor under  $\shG/S$. Moreover, any torsor $Z/S$ under $\shG/S$ is isomorphic to a torsor of the form $Y/S$.
\end{lemma}

\proof Indeed, the $p$-polynomial $P(z)=z^p-a^{p-1}z$ defines an isogeny of the additive group scheme,
and yields a short exact sequence
$$
0\lra\shG\lra \GG_{a,S}\stackrel{P}{\lra}\GG_{a,S}\lra 0
$$
in the fppf-topology. In turn, we get a long exact sequence
\begin{equation}
\label{torsor sequence}
H^0(S,\O_S)\stackrel{P}{\lra} H^0(S,\O_S)\stackrel{\delta}\lra H^1(S,\shG)\lra H^1(S,\GG_{a,S}) \lra H^1(S,\GG_{a,S}).
\end{equation}
The element $b\in R=H^0(S,\O_S)$ yields, via the coboundary map $\delta$, a $\shG$-torsor.
According to \cite{Giraud 1971}, Chapter III, Definition 3.1.3, the torsor is defined as the fiber for the $\shG$-torsor $P:\GG_a\ra\GG_a$
over the section $b:S\ra\GG_a$. We thus see that $Y$ is the $\shG$-torsor coming from $b\in R$
via the coboundary map.
The group scheme action $\shG\times_S Y\ra Y$ is just induced by the addition in $\GG_a$, hence
given by the homomorphism $u\mapsto z+u$.

The etale cohomology group $H^1(S,\GG_{a,S})$
is isomorphic to the coherent sheaf cohomology group $H^1(S,\O_S)$, and since $S$ is affine, this latter group is trivial. 
Hence, since torsors under  $\shG$ are in bijection with the elements of $H^1(S,\shG)$, any torsor $Z/S$ under $\shG/S$ is isomorphic to a torsor of the form $Y/S$.
\qed 

\begin{proposition}
\label{effective model torsor}
Suppose that $R$ is a locally factorial noetherian domain of characteristic $p>0$.
Then $\shG/S$ and
the $S$-action $\shG \times_S Y \to Y$ in {\rm \ref{naturalaction}}  is the effective model for the action of $G=\ZZ/p\ZZ$
on $Y/S$. 
\end{proposition}

\proof
Since the $\shG$-action on the torsor $Y$ is faithful, the canonical map $\shG\ra\Aut_{Y/S}$ is 
a monomorphism. The structure morphism $\shG\ra S$ is finite, and $\Aut_{Y/S}\ra S$ is affine (\ref{lemma.affine}).
It follows that $\shG\ra\Aut_{Y/S}$ is proper, hence a closed embedding by \cite{EGA}, Chapter IV, Proposition 18.12.6.
Being flat, the morphism $\shG\ra S$ has the going-down property, hence the generic fiber $\shG_\eta$ is dense in $\shG$.
By assumption, the scheme $S=\Spec(R)$ is normal, whence contains no embedded component, and the same
holds for $\shG$, because it is the spectrum of  $R[z]/(z^p-a^{p-1}z)$, whose  underlying $R$-module is free.
It follows that $G=\shG_\eta$ is schematically dense in $\shG$.

Our assumption that $S$ is locally
factorial and Theorem \ref{effective model} show that the effective model for the $G$-action exists.
Theorem \ref{effective model} also shows that the closed subscheme $\shG$ of $\Aut_{Y/S}$ is the effective model 
if we show that when $S$ has dimension $1$, then $\shG$ is equal to the schematic image $\shG':=f(G_S)$ of the canonical homomorphism $f:G_S\ra\Aut_{Y/S}$. 
Since $G_S$ is reduced and $G_\eta\subset G_S$ is dense, the same properties hold for $\shG'$. 
Both $\shG$ and $\shG'$ are the closures of their generic fibers inside $\Aut_{Y/S}$.
Hence, $\shG=\shG'$.
\qed

\section{Wild actions in dimension one}
\label{wild actions in dimension one}

We further discuss in this section
the notion of effective model for a group action
in the case where the 
base scheme $S$ has dimension $1$.
The main result of this section is Theorem \ref{characterization moderately ramified}.

Let $A:=k[[u]]$ be a formal power series ring over a field $k$ of characteristic $p>0$.
Let $\sigma:A\ra A$ be a $k$-linear automorphism of order $p$, generating a cyclic group $G\subset\Aut_k(A)$
of order $p$. 
Let $x:=N_{A/A^G}(u)=\prod_{\sigma\in G}\sigma(u)$ be the   norm element of the uniformizer $u\in A$.
It is obviously invariant under the action of $G$, so that $k[[x]]\subset A^G$. Samuel's Theorem \cite{Samuel 1966} 
ensures that this canonical inclusion $k[[x]]\subset A^G$ is   an equality.

Since the field $k$ contains only $\zeta=1$ as $p$-th root of unity, 
we find that $\sigma(u)=u+\text{\rm (higher order terms)}$.
It is standard to define the \emph{ramification break} of the action as the largest   integer $m>0$ such that
the induced action of $\sigma$ on   $A/\maxid_A^{m+1}$ is trivial. 
In particular, $\sigma(u)=u+u^{m+1}({\rm unit})$, and $m>0$ 
completely determines the higher ramification groups $G=G_1= G_2 =\dots=G_m\supsetneqq G_{m+1}= (0)$.

Let now $Y:=\Spec A$ and $S:=\Spec A^G$. 
Using the $G$-action on $A$, we associate an action of the constant group scheme $G_S$ on $Y$,
given by the morphism $G_S \times_S Y \to Y$ corresponding to the ring homomorphism
$A \ra A \otimes_{A^G} A^G[s]/(s^p-s)$ defined by
\begin{equation*} 
a\longmapsto \sum_{i=0}^{p-1} \left( \sigma^i(a)\otimes \prod_{j \in {\mathbb F}_p, \;  j\neq \overline{i}}\frac{(s-j)}{(\overline{i}-j)}\right).
\end{equation*}
Consider now the effective model $\shG/S$ for the action of $G_S$ on $Y$, as recalled in Section \ref{extensions torsors}.
Since $\dim(S)=1$, the existence of the effective model is proved already in \cite{Romagny 2012}, Theorem 4.3.4,
where one finds that $\shG/S$ is a finite flat group scheme of degree $p$ with an $S$-action $\shG \times_S Y \to Y$. It is natural to ask
when $Y/S$ is a torsor for this action. 
The answer to this question is given in our next theorem, which is the main result of this section.

\begin{theorem}
\label{characterization moderately ramified}
Keep the above notation. Let $\shG/S$ be the effective model for the action of $G_S$ on $Y$. The following are equivalent:
\begin{enumerate}[\rm (i)]
\item
The scheme  $Y/S$ is a torsor for the action of $\shG/S$.
\item
The integer $m+1$ is divisible by $p$.
\item
There exists a uniformizer $u\in A$ and  an invariant element $a\in A^G$ such that $\sigma(u)=u+a$.
\end{enumerate}
Let $\rho:=(m+1)(p-1)/p$. When these equivalent conditions are satisfied, 
the group scheme $\shG $ is isomorphic to the group scheme $\shG_{x^{\rho},0} $ in the Tate--Oort
classification \cite{Tate; Oort 1970}.
\end{theorem}

Note that in characteristic $p=2$, the equivalent conditions in Theorem \ref{characterization moderately ramified} automatically hold, because the ramification
break $m$ is known to be always odd in this case (the proof of this fact is recalled 
in \ref{RamificationBreak}).
Theorem \ref{characterization moderately ramified} illustrates the fact that the geometric condition {\it $Y/S$ is a torsor for the action of $\shG/S$}
implies an interesting condition on the equation defining the automorphism $\sigma$. 
When the three equivalent conditions in Theorem \ref{characterization moderately ramified} are satisfied, we call the action of $G$ on $A$ {\it moderately ramified}.  
We consider the case where  $\dim(A)>1$ in \ref{definition moderately ramified} 
and impose  an analogous geometric condition (MR5) in our general definition of moderately ramified action. 
The proof of Theorem \ref{characterization moderately ramified} is postponed to 
\ref{Proof characterization moderately ramified}. We start by reviewing the relevant parts of the Tate--Oort
classification \cite{Tate; Oort 1970}.

\begin{emp} \label{TateOortReview}
Let $R$ be any ring of characteristic $p>0$,  with $\Pic(R)=0$. Set $S:=\Spec R$.
The Tate--Oort classification \cite{Tate; Oort 1970},
which we already
discussed in \ref{FFG}, now takes the following simpler form:
a group  scheme $\shG/S$ whose structure sheaf is locally free of rank $p$  is isomorphic to
$\shG_{\alpha,\beta}:=\Spec(R[T]/(T^p-\alpha T))$
for some elements $\alpha,\beta\in R$ with $\alpha \beta=0$.
As abelian sheaves,  
$$
\shG_{\alpha,\beta}(T):=\left\{s\in\Gamma(T,\O_T)\mid s^p=\alpha s\right\},
$$
with group law given by the formula
$s\star s' = s+s'+\beta\sum_{i=1}^{p-1} \frac{1}{i! (p-i)!} s^is'^{p-i}$.
Moreover, two group schemes  $\shG_{\alpha,\beta}/S$ and $ \shG_{\alpha',\beta'}/S$ 
are isomorphic if and only if there exists $\epsilon\in R^\times$
such that $\alpha'=\epsilon^{p-1}\alpha$ and $\beta'=\epsilon^{1-p}\beta$.
\end{emp}

\begin{proposition}
\label{generically split group scheme}
Suppose that $R$ is a domain. Then the generic fiber of $\shG_{\alpha,\beta}$ is 
\'etale if and only if $\alpha\neq 0$ and   $\beta=0$.
Assume in addition that $R$ is normal. 
Then  the generic fiber of $\shG_{\alpha,\beta}$ is constant if and only if  $\beta=0$ and $\alpha=a^{p-1}$
for some non-zero $a\in R$.  
\end{proposition}

\proof
Taking the derivative of $T^p-\alpha T$, we see that $\alpha$ defines the locus
where the structure morphism $\shG_{\alpha,\beta}\ra\Spec(R)$ is not smooth.
Suppose that  the generic fiber of $\shG_{\alpha,\beta}$ is \'etale. Then $\alpha\neq 0$.
Since $R$ is a domain, the condition $\alpha\beta=0$ ensures $\beta=0$.
The converse is   also clear.

Now assume that $R$ is normal.
Suppose that $\alpha=a^{p-1}$ for some non-zero $a\in R$.
Then we have a factorization $T^p-\alpha T=\prod_{i=0}^{p-1}(T-ia)$. The roots $ia$ are pairwise different.
Thus the generic fiber of $\shG_{\alpha,\beta}$ must be constant.
Conversely, suppose that $\shG_{\alpha,\beta}$ is generically constant.
Then the polynomial $T^p-\alpha T$ is separable, and its roots are contained
in the  field of fractions
$F=\Frac(R)$.
In particular, there is a non-zero element $a\in F$ with $a^{p-1}=\alpha$. Again it follows that
$\beta=0$. Since $R$ is normal, we already have $a\in R$.  
\qed

\begin{corollary} \label{cor.structure}  \label{effective model generically constant}
Let $\shG/S$ be  a finite flat group scheme of degree $p$ that  is
generically constant. Then  $\shG/S$ is isomorphic
to $\shG_{x^{r(p-1)},0}/S$ for some unique integer $r\geq 0$.
\end{corollary}

\proof
According to  Proposition \ref{generically split group scheme}, the group scheme $\shG$ is isomorphic to $\shG_{\alpha,0}$ for some 
power series of the form $\alpha=a^{p-1}$ with $a\in R$.
Write $\alpha=\epsilon\cdot x^e$ for some unit $\epsilon\in R$ and some exponent $e\geq 0$.
Taking valuations of both sides, we see that $e=(p-1)r$ for some $r$, so $x^e$, and whence the unit $\epsilon$, are $(p-1)$-th powers.
Since $\alpha$ is unique up to invertible $(p-1)$-th powers, we may assume $\alpha=x^{r(p-1)}$,
so that
$$
\shG=\shG_{x^{r(p-1)},0}= \Spec R[s]/(s^p-x^{r(p-1)} s).
$$
The integer $r\geq 0$ is unique, because it is the length of the
intersection, inside the scheme $\shG$, of the zero-section (defined by $s=0$)
and  any non-zero section (defined by $s=ix^r$ for some $i\in\FF_p^\times$).
\qed

\medskip
We now return to the initial set-up of this section where $ A=k[[u]]$ and $R=A^G=k[[x]]$.
Recall that since $\Frac(A)/\Frac(A^G)$ is a Galois extension of degree $p$, we can find $z \in \Frac(A)$, a power series 
$f(x)=\sum_{i=0}^\infty \lambda_i x^i \in k[[x]]^\times$, 
and an integer $\mu $, such that $\Frac(A)=\Frac(A^G)(z)$ and $z$ satisfies the Artin--Schreier
equation $z^p-z=x^{-\mu} f(x)$.
Since the extension $A/A^G$ is totally ramified, 
we find that $\mu>0$. It is well-known that there exists such $z$ with $\mu$ coprime to $p$.
Indeed, write $\mu=pr-i$ for some uniquely defined non-negative integers $r$ and $0\leq i<p$.
The ring $A$ contains the element $s:=zx^{r}$, which satisfies the equation 
$s^p -x^{r(p-1)}s = x^i f(x)$. 
If $i=0$, then the residue field of $A$ contains the class of $s$, 
and since the residue field extension is trivial, we find 
that $\lambda_0$ is a $p$-th power modulo $(x)$, say $\lambda_0=\ell^p$ for some $\ell \in k$. 
It follows that we can consider $z':=z+ \ell/x^r$, 
which satisfies $z'^p-z'= x^{-\mu} f(x) - \lambda_0/x^\mu + \ell/x^r$. 
After possibly finitely many similar steps, we obtain a generator $z$
for $\Frac(A)$ whose associated integer $\mu>0$ is coprime to $p$.
The following lemma is well-known, and the notation introduced in its proof
will be used in \ref{Proof characterization moderately ramified}.

\begin{lemma}
\label{RamificationBreak} 
Write as above $\Frac(A)=\Frac(A^G)(z)$ such that $z$ satisfies the Artin--Schreier
equation $z^p-z=x^{-\mu} f(x)$  with $\mu>0$ coprime to $p$. Then $\mu=m$.
\end{lemma}
\proof
Write $\mu=pr-i$ with $0<i<p$.
Let $0<c<p$ be the unique integer such that $ci$ is congruent to $1$ modulo $p$,
and let $d:=(ci-1)/p$, so that $ci-dp=1$.
The element $s=zx^r\in  \Frac(A)$ satisfies the integral equation
$$
s^p - x^{r(p-1)}s = x^if(x).
$$
Recall that $A=k[[u]]$, so that 
${\rm ord}_u(s)={\rm ord}_u(z)+pr=-\mu+pr =i$. Consequently, the element
$s^c/x^d$ has valuation
${\rm ord}_u(s^c/x^d)= ci-pd=1$. 
Thus the element $s^c/x^d$ is a uniformizer in $A$, and  hence  $k[[s^c/x^d]]=A$.
The group action $z\mapsto z+1$ sends $ s$ to $s+x^r$.
To determine the ramification break $m$, it remains to compute the valuation of
$$
\frac{(s+x^{r})^c}{x^d}-\frac{s^c}{x^d}=\sum_{n=0}^{c-1} \binom{c}{n} s^n x^{r(c-n)-d},
$$
and we leave this computation to the reader.
\if false
Since $c<p$, the $n$-th summand has valuation $ni+ (r(c-n)-d)p = -dp+ n(i-rp) +rpc$.
This expression takes its minimal value exactly when $n=c-1$, because  $0<i<p\leq rp$.
In turn, this unique minimal value is the  valuation of the sum
(\cite{Bourbaki AC 1-7}, Chapter VI, \S 3, No.\ 1, Proposition 1):
$$
-dp+ (c-1)(i-rp)+rcp = -dp+ci   - i +rp   = 1+\mu.
$$
It follows that $\mu=m$ is indeed the ramification break.
\fi
\qed

\begin{emp} 
\label {Proof characterization moderately ramified}
\noindent
{\it Proof of Theorem {\rm \ref{characterization moderately ramified}}.}
The implication (iii)$\Rightarrow$(ii) is   immediate: Suppose that $\sigma(u)=u+a$ for some invariant $a$.
Then $m+1:=\val_u(\sigma(u)-u)= \val_u(a)$, and the latter is a multiple of $p$ since $A^G\subset A$ has degree $p$.

(ii)$\Rightarrow$(i): Suppose that $m=pr-1$. As above, we can find $z \in \Frac(A)$ and  a power series 
$f(x)=\sum_{i=0}^\infty \lambda_i x^i \in k[[x]]^\times$  such that $\Frac(A)=\Frac(A^G)(z)$ and $z$ satisfies the Artin--Schreier
equation $z^p-z=x^{-m} f(x)$.
With the notation from the proof of Lemma \ref{RamificationBreak},
we have $i=c=1$ and $d=0$, and we see that $s:=zx^r\in A$ is a uniformizer, satisfying
the equation
$$
s^p-x^{r(p-1)}s - xf(x) =0.
$$
Moreover, the $G$-action is given by $s\mapsto s+x^r$.
In light of the exact sequence \eqref{torsor sequence}, the scheme $Y=\Spec(A)$ is a torsor
for the group scheme $\shG_{x^{r(p-1)},0}$. By Proposition \ref{effective model torsor},
the group scheme $\shG_{x^{r(p-1)},0}$ and its associated action on $Y$ is the effective model for the action of $G$ on $Y$.

(i)$\Rightarrow$(iii): Suppose that $Y$ is a $\shG$-torsor. 
Corollary \ref{cor.structure}
shows that since $\shG/S$ is  a finite flat group scheme of degree $p$ that  is
generically constant, then it is isomorphic
to $\shG_{x^{r(p-1)},0}/S$ for some unique integer $r\geq 0$.
The case $r=0$ is impossible, for then the closed fiber of $\shG$ is \'etale, whereas the closed fiber of $Y \to S$ is non-reduced. 
Since $Y$ is a $\shG$-torsor, Lemma \ref{naturalaction} allows us to write that 
$$
A=k[[x]][s]/(s^p-x^{r(p-1)}s - g(x))
$$
for some power series $g(x)\in k[[x]]$. 
The induced $G$-action is given by $s\mapsto s+x^r$.
Using the hypothesis that $A$ is formally smooth over $k$,
we infer that the  partial derivative of $s^p-x^{r(p-1)}s - g(x)$ with respect to $x$ is a unit,
and we deduce from the Implicit Function Theorem
that one may express $x$ as a formal power series in $s$.
In turn, the canonical map $k[[s]]\ra A$ is bijective.
Setting $u=s$ and $a=x^r$, we see that (iii) holds.
\qed
\end{emp}

\if false
\begin{remark} Let $S=\Spec k[[x]]$.
The group scheme $\shG_{1,0}/S$ is nothing but the constant group scheme $G_S/S$.
When $r'\leq r$, we have a canonical  $S$-group scheme homomorphism 
$ \shG_{x^{r'(p-1)},0} \to \shG_{x^{r(p-1)},0}$, 
which on the level of $k[[x]]$-algebras
is given by 
$$
k[[x]][s]/(s^p-x^{r(p-1)} s) \longrightarrow k[[x]][t]/(t^p-x^{r'(p-1)} t),\quad s \longmapsto x^{r-r'}t.
$$
We will view below the $S$-morphism $ \shG_{1,0} \to \shG_{x^{r(p-1)},0}$ as given by the inclusion of the $k[[x]]$-subalgebra
$k[[x]][x^rs] $ into $k[[x]][s]/(s^p-s)$. 
Let now $Y=\Spec A$, and consider the action $G_S \times_S Y \to Y $ given explicitly in \eqref{eq.action}. 
We have seen in Theorem \ref{characterization moderately ramified}
that when $m+1$ is divisible by $p$, then this action naturally factors as
$$G_S \times_S Y \lra \shG_{x^{r(p-1)},0} \times_S Y \lra Y,
$$
with $rp=m+1$. We show below that such factorization can be seen using the explicit homomorphism in \eqref{eq.action}, and that in general, 
a similar factorization occurs when $rp(p-1)\leq m+1$. 
Indeed, let us note first the following two identities involving  expressions
appearing in \eqref{eq.action}:
$$
\sum_{i=0}^{p-1} \left( \prod_{j \in {\mathbb F}_p, \;  j\neq \overline{i}}\frac{(s-j)}{(\overline{i}-j)}\right) = 1,
  \text{ \rm and }
\sum_{i=0}^{p-1} \left( \overline{i} \prod_{j \in {\mathbb F}_p, \;  j\neq \overline{i}}\frac{(s-j)}{(\overline{i}-j)}\right) = s.
$$
The left hand sides of these equalities are polynomials in $s$ of degree at most $p-1$, and thus they are completely determined by their values 
at the $p$ elements $j \in {\mathbb F}_p$. Assume now that $rp(p-1)\leq m+1$. We can then write
$$ \sum_{i=0}^{p-1} \left( \sigma^i(a)\otimes \prod_{j \in {\mathbb F}_p,\; j\neq \overline{i}}\frac{(s-j)}{(\overline{i}-j)}\right)
 = a \otimes 1+
 \sum_{i=0}^{p-1} \left( (\sigma^i(a)-a)x^{-r(p-1)}\otimes \prod_{j \in {\mathbb F}_p, \; j\neq \overline{i}}\frac{(sx^r-jx^r)}{(\overline{i}-j)}\right),
 $$
 which justifies the existence of a factorization $G_S \times_S Y \lra \shG_{x^{r(p-1)},0} \times_S Y$.
 When $rp=m+1$, there exists a uniformizer $u\in A$ such that $\sigma(u)=u+\epsilon x^r$ for some $\epsilon \in (A^G)^\times$. It follows in this case
 that 
 $$ \sum_{i=0}^{p-1} \left( \sigma^i(u)\otimes \prod_{j \in {\mathbb F}_p,\; j\neq \overline{i}}\frac{(s-j)}{(\overline{i}-j)}\right)
 =
 u \otimes 1 +1\otimes \epsilon x^r s.
 $$
 \end{remark}
\fi
\if false
\begin{proposition} \label{torsor1}
Keep the above notation. In particular,  $Y=\Spec A$, $S =\Spec k[[x]]$, and $\shG/S$ is the effective model for the action of $G$ on $Y$.
Let $m>0$ be the ramification break, and write $m+1=pr+i$ with some unique integers $r\geq 0$ and  $0\leq i<p$.
Then 
$\shG/S$ is isomorphic to $ \shG_{x^{r(p-1)},0}/S$.
\end{proposition}

\proof
By definition of the ramification break $m > 0$, the $m$-th infinitesimal neighborhood
$Y_m:=\Spec(A/\maxid^{m+1})$ is the largest subscheme of $Y$ on which $G$ acts trivially.
If follows that $Y_{pr}:=\Spec(A/\maxid^{pr})$ is the largest infinitesimal neighborhood
on which $G$ acts trivially that is the preimage of some subscheme  of $S$ under the morphism $Y\ra S$.

In particular, the  closure
$\overline{\left\{\sigma\right\}}\subset\shG\subset\underline{\GL}_S(\O_Y)$
of the generator $\sigma\in G$ has the property that 
its schematic intersection with the zero-section  $S\ra\underline{\GL}_S(\O_Y)$
is an Artin scheme of length precisely $r$.
Among all finite flat group scheme of order  $p$ over $S=\Spec(k[[x]])$,
only $\shG_{x^{r(p-1)},0}= \Spec k[[x]][s]/(s^p-x^{r(p-1)} s)$ has this property.
If follows that the effective model $\shG$ is isomorphic to $\shG_{x^{r(p-1)},0}$.
\qed
\fi

We can generalize the last statement of Theorem {\rm \ref{characterization moderately ramified}} as follows.
\begin{proposition} \label{torsor1}
Let $G$ and $A$ be as at the beginning of this section. In particular, $A^G=R=k[[x]]$, $Y=\Spec A$, and $S
=\Spec k[[x]]$.
Let $m>0$ be the ramification break, and write $m+1=pr+i$ with some
unique integers $r\geq 0$ and  $0\leq i<p$.
Then the effective model $\shG/S$ of the action of the constant group scheme $G_S$ on $Y$ is isomorphic to $ \shG_{x^{r(p-1)},0}/S$.
\end{proposition}

\proof
According to Corollary \ref{effective model generically constant}, the
effective model of the $G$-action is of the form
$$
\shG=\shG_{x^{\ell(p-1)},0}=\Spec R[s]/(s^p-x^{\ell(p-1)} s)
$$
for some integer $\ell \geq 1$. We have to prove that $\ell=r$.
In the group scheme $\shG_{x^{\ell(p-1)},0}/S$, the zero-section is given by the ideal $(s)$. 
Let $P$ denote the section defined by the ideal $(s-x^\ell)$.
By definition of the effective model when ${\rm dim}(S)=1$ (see proof of \ref{effective model} in \ref{proof.effective model}), we have a closed immersion $\shG=\shG_{x^{\ell(p-1)},0} \subset \underline{\GL}_S(\O_Y)$.
Thus the section $P$ gives an $S$-automorphism $P_0:Y\ra Y$,
and this automorphism becomes the identity
after tensoring
with $R/\maxid_R^\ell$, but is not the identity when tensoring with
$R/\maxid_R^{\ell+1}$.

Consider now the morphism $G_S \to \shG$. Since this morphism is an isomorphism 
after tensoring with $\Frac(R)$, we find that $P \in \shG(S)$ lifts to a generator of $G_S(S)$.
By definition of the ramification break $m\geq 0$, 
any generator of $G_S(S)$ 
corresponds to an automorphism $Y\to Y$ such that
this automorphism becomes the identity
after tensoring
with $A/\maxid_A^{m+1}$ over $A$, but is not the identity when tensoring with
$A/\maxid_A^{m+2}$. It follows that this automorphism 
becomes the identity
after tensoring
with $R/\maxid_R^{r}$ over $R$, but is not the identity when tensoring with
$R/\maxid_R^{r+1}$. Applying this to the automorphism $P_0$ shows that $\ell=r$.
\qed
\if false

On the other hand, $P\subset\shG$ yields a generator $\sigma\in G$
after tensoring with $\Frac(R)$.
One may also view $P$   as the image of the constant section $\sigma_S$
under the  homomorphism
$G_S\ra\shG$. It follows that the $S$-automorphism $P:Y\ra Y$ coincides
with the automorphism $\sigma:Y\ra Y$.

the infinitesimal
neighborhood
$Y_m=\Spec(A/\maxid_A^{m+1})$ is the largest closed subscheme on which
$\sigma$ and hence $P$ act trivially.
It follows that $Y_{pr}=\Spec(A/\maxid_R^{r}A)$ is the largest such
infinitesimal neighborhood
that is also the preimage of some closed subscheme in $S=\Spec(R)$ under
the structure morphism $Y\ra S$.
We conclude $l=r$.
\fi

\section{Moderately ramified actions}
\label{moderately ramified}

Let $A$ be a complete   local noetherian ring that is regular,  of dimension $n\geq 2$ and 
characteristic $p>0$, with maximal ideal $\maxid_A$ and field of representatives $k$.
Let $G$ be a finite cyclic group of order $p$, and assume that $A$ is endowed with a 
faithful    action of $G$ ramified precisely at the origin (\ref{ramifiedprecisely}), such that  $k\subset A^G$. 
Choose an admissible regular system of parameters $u_1,\ldots,u_n\in A$ with respect to the $G$-action   (\ref{exist admissible}).
Write as before $x_i:=N_{A/A^G}(u_i)$, $i=1,\dots, n$, for the norm elements, 
and consider the norm subring $R:=k[[x_1,\ldots,x_n]]$ of $A=k[[u_1,\ldots,u_n]]$.
The ring extension  $R\subset A$ is flat and finite of degree $p^n$ (\ref{norm elements}). 

In this section, we introduce  \emph{five cumulative assumptions} on 
the $G$-action and on the choice   of parameters,
which ensure that the ramification in $\Spec(A)\ra\Spec(R)$ is ``as small  and simple as possible''.
This will lead to the notion of \emph{moderately ramified} $G$-action in \ref{definition moderately ramified}, for which we shall
obtain  structure results in \ref{structure moderately ramified} and \ref{structure moderately ramified2}. We show in \ref{artin case} that 
when $n=2$ and $p=2$, and $k$ has no separable quadratic extensions, then every $G$-action that is ramified precisely at the origin is moderately ramified. 
The first condition on the $G$-action that we want to consider is:

\medskip
\begin{list}{-}{\leftmargin2em}
\item[\textbf{(\MRone)}] \emph{The field extension $\Frac(A)/\Frac(R)$ is Galois.}
\end{list}

\medskip \noindent
As we note in Theorem \ref{galois extension}, this condition often reduces to checking that the smaller extension $\Frac(A^G)/\Frac(R)$  is Galois.
Assuming from now on that (\MRone) holds, we write  $H$ for the Galois group of $\Frac(A)/\Frac(R)$, 
which  has order $p^n$ and contains $G$. Given $\primid \in \Spec A$, the notation $I_\primid$ always refers in this section 
to the inertia group  inside $H$.

\begin{proposition}
\label{generated by inertia}
If \rm{(\MRone)} holds, then the Galois group $H$ is generated by the inertia subgroups
$I_\primid\subset H$, where $\primid$ runs through the height one prime ideals in $A$.
\end{proposition}

\proof
Consider the subgroup $H'\subset H$ generated by the   inertia subgroups $I_\primid$ for all height one prime ideals $\primid $ of $ A$.
Set $R':=A^{H'}$ and consider the morphism $\Spec R' \to \Spec R$.
Seeking a contradiction, we assume $H'\neq H$, so that $R'\neq R$.
Let $\mathfrak q\subset R$ be a prime ideal of height one.
Then there exists a prime ideal $\mathfrak p\subset A$ of height one such that $\mathfrak q = R \cap \mathfrak p$. 
It follows that the morphism $\Spec R' \to \Spec R$
is unramified over $\mathfrak q$.  On the other hand, the morphism is ramified at the maximal ideal of $R'$, because the 
morphism is finite and the residue field extension at the maximal ideal is trivial. Since $R'$ is normal and $R$ is regular  and the morphism is finite, we can apply the 
Zariski--Nagata Purity Theorem to $\Spec R' \to \Spec R$ and find that the non-empty branch locus is  pure of codimension one in $\Spec R$. 
This is a contradiction.
\qed

\medskip
A similar argument appears in \cite{Serre 1967}, Th\'eor\`eme 2', and we discuss this result further in \ref{pseudoreflection}.
The next condition ensures that the inertia groups for the quotient map 
$\Spec(A)\ra\Spec(R)$ are as  small and uniform as possible:

\medskip
\begin{list}{-}{\leftmargin2em}
\item[\textbf{(\MRtwo)}] \emph{
The inertia group $I_\primid$ of every ramified prime ideal $\primid \subset A$  of height one is cyclic of order $p$, and normal in $H$.}
\end{list}

\noindent
Our next lemma shows that this condition is automatic in dimension two.

\begin{lemma} \label{n=2}
If {\rm {(\MRone)}} holds and $n=2$, then  Condition {\rm {(\MRtwo)}}  also holds, and $H$ is elementary abelian of order $p^2$.
\end{lemma}

\proof When $n=2$, we find that $|H|=p^2$ and so $H$ is abelian. 
By the Zariski--Nagata Purity Theorem applied to the ramified morphism $\Spec A \to \Spec R$, 
there is at least one prime ideal $\primid$ of height one in $A$ whose inertia group $I_\primid\subset H$ is non-trivial.
The group $I_\primid$ cannot be the whole group $H$, otherwise the morphism $\Spec A \to \Spec A^G$ is ramified at $\primid$, a contradiction.
Hence, all non-trivial inertia subgroups $I_\primid$ are cyclic of order $p$.
Note that the group $H$ cannot be cyclic, since otherwise, it contains only one subgroup of order $p$,
which must coincide with $G$. This is not possible, since then $G = I_\primid$, contradicting the assumption that
$G$ acts freely outside the closed point. Hence, $H$ is elementary abelian of order $p^2$. \qed


\begin{proposition}
\label{elementary abelian} 
If {\rm {(\MRone)}} and  {\rm {(\MRtwo)}} hold, then the Galois group $H$ is elementary abelian of order $p^n$. 
\end{proposition}

\proof We use here Proposition \ref{generated by inertia} and the fact that if $H_1$ and $H_2$ are two abelian subgroups of $H$ such that $H_1 \cap H_2 =(0)$ and such that both $H_1$ and $H_2$ are normal in $H$,
then the subgroup of $H$ generated by $H_1$ and $H_2$ is also abelian and normal in $H$. To see this, simply note that a commutator $h_1 h_2 h_1^{-1}h_2^{-1}$
with $h_i \in H_i$ belongs to both $H_1$ and $H_2$.
\qed

\medskip 
Assume that (\MRone) and (\MRtwo) hold.
Since $H$ is elementary abelian of order $p^n$, it
contains exactly $(p^n-1)/(p-1)$ distinct subgroups of order $p$. Proposition \ref{generated by inertia} let us choose
 $n$ distinct such cyclic subgroups $G_1,\ldots,G_n$
that are equal to the inertia subgroup of some ramified prime ideal of height one in $A$ and such that the natural map
$$
G_1\times\ldots \times G_n\longrightarrow H
$$
is an isomorphism.
Each subgroup $G_i\subset H$ has
a canonical complement, namely the subgroup $G_i^\perp$ of $H$ generated by the subgroups  $ G_j$, $j \neq i$.

Consider the corresponding ring of invariants $A^{G_i^\perp}\subset A$.
Then $\Frac(A^{G_i^\perp})/\Frac(R)$ is a cyclic extension of degree $p$,
with Galois group $H/G_i^\perp$, which we identify with the group $G_i$. By hypothesis, $G_i$ is the inertia group of a prime ${\mathfrak p}_i$ of $A$.
It follows that the morphism $\Spec(A^{G_i^\perp}) \to \Spec(R) $ is ramified at ${\mathfrak p}_i \cap A^{G_i^\perp}$.
By the Zariski--Nagata Purity Theorem, 
the branch locus of $\Spec A^{G_i^\perp} \to \Spec R $ is of the form $V(a_i)$ for some non-invertible element $a_i\in R=k[[x_1,\ldots,x_n]]$.

Consider the canonical homomorphism of $R$-algebras:
\begin{equation} \label{map}
f:A^{G_1^\perp}\otimes_R\cdots\otimes_R A^{G_n^\perp}\lra  A.
\end{equation}

\begin{lemma} \label{flat, system parameters} 
Keep the above notation, and assume that {\rm (\MRone)} and {\rm (\MRtwo)} hold.
If the above map $f$ is an isomorphism, then the extension $R\subset A^{G_i^\perp} $ is flat and the ring $A^{G_i^\perp}$ is regular for $i=1,\dots, n$. 
Moreover, the elements $a_1,\ldots,a_n$ form a system of parameters of $R$.
\end{lemma}

\proof 
Let $d_i$ denote the dimension of $A^{G_i^\perp} \otimes R/\maxid_R$ as an $R/\maxid_R$-vector space. 
By Nakayama's Lemma, there exists a surjective homomorphism  $\varphi_i: R^{\oplus d_i} \to A^{G_i^\perp}$ of $R$-modules.
Tensoring with $\Frac(A^{G_i^\perp})$, which has degree $p$ over $\Frac(R)$, we find that  $d_i\geq p$.
 Since $A$ is free of rank $p^n$ and $f$ is an isomorphism, we have $\prod_{i=1}^n d_i = p^n$, which implies that
 $d_i=p$ for all $1\leq i\leq n$.
Since the $R$-module $R^{\oplus p}$ is torsion-free, we conclude that the morphism $\varphi_i: R^{\oplus p} \to A^{G_i^\perp}$ is an isomorphism for each $i=1,\dots, n$.
Hence, the extension $R\subset A^{G_i^\perp}$ is flat for all $i=1,\dots, n$.
It follows that the finite ring extensions $A^{G_i^\perp} \subset \bigotimes_{j=1}^n A^{G_j^\perp}$
are flat. Since $A$ is regular, 
\cite{EGA}, IV.6.5.1  ensures  that the rings $A^{G_i^\perp}$
are regular. 

If $a_1,\dots, a_n$ does not form a system of parameters in $R$, 
then the ideal $(a_1,\ldots,a_n)$ of $R$ is contained in some non-maximal prime ideal
$\idealq$ of $R$.
Let $s\in \Spec(R)$ be the corresponding non-closed point.
Since $s$ belongs to the branch locus of each morphism $\Spec(A^{G_i^\perp}) \to \Spec R$ by hypothesis, the preimage of $s$ in each $\Spec(A^{G_i^\perp})$ is a singleton,
with purely inseparable (possibly trivial) residue field extension.
Since by hypothesis, $\Spec(A)$ is the fiber product of the $\Spec(A^{G_i^\perp})$, we find that the preimage $t$ of $s$  in $\Spec(A)$ is also a singleton with purely inseparable residue field extension. Hence, the inertia group at $t$ for the action of $G$ is not trivial, contradicting the hypotheses that the action of $G$ on $A$ is ramified precisely at the origin. \qed

\begin{remark} \label{IntegralClosureAbelianExt} Let $ Y_i:=\Spec(A^{G_i^\perp}) $ and  $S:=\Spec(R)$.
When $n=2$, the morphism $Y_i\to S$ is flat since $Y_i$ is normal and, hence, Cohen--Macaulay. When $n>2$, Condition (\MRfour) discussed below 
will imply that  $Y_i\to S$ is flat. 
Note that the flatness would be automatic if the order of the group were prime to the residue characteristic.
Indeed, it is proved in \cite{Rob} that if $R$ is a regular local ring and $L/{\rm Frac}(R)$ is a Galois extension with abelian Galois group $H$
of order coprime to the residue characteristic $p$ of $R$, then the integral closure $B$ of $R$ in $L$ is Cohen--Macaulay. It follows then that $B/R$ is flat.
An example is given in \cite{Rob} that shows that the hypothesis $\gcd(|H|,p)=1$ is needed in the proof of the statement. We note here that this hypothesis
is also needed for the statement to hold in the equicharacteristic case. For this, let $G={\mathbb Z}/p{\mathbb Z}$, and consider a moderately ramified $G$-action on $A=k[[u_1,\dots, u_n]]$, with norm ring $R$, as in \ref{structure moderately ramified2}.
Then the extension ${\rm Frac}(A^G)/{\rm Frac}(R)$ is elementary abelian with Galois group $H$ of order $p^{n-1}$, and  $A^G$ is the integral closure of $R$ in ${\rm Frac}(A^G)$. 
But $A^G $ is not Cohen--Macaulay when $\dim(A)>2$, 
in view of the facts reviewed in \ref{properties invariant ring}. 

We do not know a counter-example to the Cohen--Macaulayness of the integral closure $B$
when $H$ is cyclic of order $p>3$. When $p=3$, such a counter-example is presented in \cite{Koh}, 2.4, in the mixed characteristic case. In \ref{generalizedreflectionCM}, we exhibit such a counter-example when $R$ is only assumed to be Cohen--Macaulay 
and $H$ is generated by a generalized reflection (definition recalled in \ref{pseudoreflection}). Under such weaker assumption on $R$, the case where $|H|$ is coprime to the residue characteristic of $R$ is treated in \cite{H-E}, Propositions 13 and 15.

\end{remark}

Our third condition below requires that the set of   inertia subgroups for the morphism $\Spec(A) \to \Spec(R)$ be as small as possible. 
This   condition is motivated by the statement of Proposition \ref{tensor product}.
 
\medskip
\begin{list}{-}{\leftmargin2em}
\item[\textbf{(\MRthree)}] 
\emph{There are only $n$   subgroups
$G_1,\ldots,G_n\subset H$
which are equal to the inertia subgroup of a  ramified prime ideal of height one in   $A$.}
\end{list}
When {(\MRthree)} holds, up to the enumeration of the subgroups, the  morphism \eqref{map} does not anymore depend on the choice of subgroups.
\begin{proposition}
\label{tensor product} Keep the above notation, and assume that {\rm (\MRone)} -- {\rm (\MRthree)} hold.
If the rings $A^{G_i^\perp} $ are Cohen--Macaulay for $1\leq i\leq n$, then the morphism $f$ is an isomorphism.
\end{proposition}

\proof 
First, we verify that the ring $A':=A^{G_1^\perp}\otimes_R\cdots\otimes_R A^{G_n^\perp}$
is normal:  Since $R$ is regular and $A^{G_i^\perp}$ is Cohen--Macaulay, 
the extension $R\subset A^{G_i^\perp}$ is flat for all $i=1,\ldots,n$ (\cite{Matsu}, (21.D) Theorem 51).
In turn, $A'$ is free as an $R$-module, hence Cohen--Macaulay. It remains to check that $A'$ is regular in codimension one.
Let $\idealq\subset R$ be a prime of height one.
We claim that {\it all  but possibly one of the extensions $R\subset A^{G_i^\perp}$ are unramified over $\idealq$}.
If this holds, we may assume without restriction  that $R_{\idealq}\subset (A^{G_i^\perp})_{\idealq}$ is unramified for $i=2,\dots, n$. 
Then $A'_{\idealq}$ is \'etale over $(A^{G_1^\perp})_{\idealq}$ and the latter ring is normal, 
which ensures that $A'_{\idealq}$ is normal too (\cite{Matsu}, (21.E) (iii)).
Summing up, $A'$ is normal.

To verify the claim,  let us assume {\it ab absurdo} that there exist two indices $i \neq j$ such that $R\subset A^{G_i^\perp}$ and $R\subset A^{G_j^\perp}$ 
are ramified over $\idealq$.
Upon renumbering if necessary, we may assume that $i=1$ and $j=2$. 
Since $\Frac(A^{G_i^\perp})/\Frac(R)$ is cyclic of degree $p$,
there is only one prime ideal $\primid_1\subset A^{G_1^\perp}$ and $\primid_2\subset A^{G_2^\perp}$ 
lying over $\idealq$. Choose some prime ideal $\primid\subset A$ lying over $\idealq$.
Clearly,   $\primid$ lies over both  $\primid_1$ and $\primid_2$. The extension $R\subset A$ is ramified at $\primid$ since $R\subset A^{G_1^\perp}$ 
is ramified at $\primid_1$.
In particular, the inertia group $I_\primid$ is non-trivial.
By Condition (\MRthree), we have $I_\primid=G_j$ for some index $j \in [1,n]$. Thus we can always assume that either $I_\primid \neq G_1$ or $I_\primid \neq G_2$.
Let us consider the case where $I_\primid \neq G_1$, the other one being similar. Since $I_\primid=G_j$ for some index $j \in [2,n]$, we find that by definition
$I_\primid\subset G_1^\perp$. Hence, $A^{G_1^\perp}\subset A^{I_\primid}$. Since $R\subset A^{I_\primid}$ is unramified at $\primid \cap A^{I_\primid}$,
we find that  $R\subset A^{G_1^\perp}$ is unramified at $\primid_1$, which is a contradiction.

We proceed by  proving that the localized morphism $f_\idealq$ is bijective for the minimal prime $\idealq=0$ inside $R$. We need to show that the map
$$
{\rm Frac}(A)^{G_1^\perp} \otimes_{{\rm Frac}(R)} \ldots\otimes_{{\rm Frac}(R)} {\rm Frac}(A)^{G_n^\perp} = A'_\idealq \lra A_\idealq={\rm Frac}(A)$$
is an isomorphism. 
It follows from the Galois correspondence that the image of this map is the subfield of ${\rm Frac}(A)$
fixed by the subgroup $G_1^{\perp}\cap\ldots\cap G_n^{\perp} $. By construction, this latter subgroup is trivial
so the map is surjective. Since both source and targets are vector spaces over ${\rm Frac}(R)$ of the same dimension, the map is an isomorphism, as desired.

Since $A'$ is a free $R$-module, the canonical map $A'\ra A'_\idealq $ is injective, and we conclude
that $f:A'\ra A$ is injective. In particular, $A'$ is integral,
and we saw in the preceding paragraph that $\Frac(A')\ra\Frac(A)$ is bijective. 
Since $A'$ is normal, the finite extension $A'\subset A$ must be an  equality.
\qed

\medskip
\noindent
Proposition \ref{tensor product} motivates  our next condition.

\medskip
\begin{list}{-}{\leftmargin2em}
\item[\textbf{(\MRfour)}] 
\emph{The rings $A^{G_i^\perp}$ are Cohen--Macaulay, for all $1\leq i\leq n$.}   
\end{list}

\smallskip \noindent
When (\MRfour) holds, we can use $f$ to identify  $A$ with $A^{G_1^\perp}\otimes_R\cdots\otimes_R A^{G_n^\perp}$.

\begin{emp} \label{MR4}
Let $ Y_i:=\Spec(A^{G_i^\perp}) $ and  $S:=\Spec(R)$.
Recall that we identify $H/G_i^{\perp}$ with $G_i$ through the canonical bijection   $G_i\ra H/G_i^{\perp}$, 
so that we get a   $G_i$-action on $A^{G_i^\perp}$  and $Y_i/S$. 
Let $\shG_i/S$ be the effective model of the $G_i$-action on $Y_i$, as introduced in Theorem \ref{effective model}.
This effective model is a finite flat group scheme over $S$, 
and comes with an  $S$-action
$\shG_i\times_S Y_i\ra Y_i$. 
We  can now formulate the last of our cumulative conditions.

\medskip
\begin{list}{-}{\leftmargin2em}
\item[\textbf{(\MRfive)}] 
The    scheme $Y_i/S$ is a torsor for $\shG_i/S$, for all $1\leq i\leq n$.
\end{list}
\end{emp}

\noindent When $\dim(S)=1$, a similar condition is considered in Theorem \ref{characterization moderately ramified}.
Our next lemma shows that Condition (\MRfive) is automatic when $p=2$.

\begin{lemma} \label{MR1-2impliesMR5}
If {\rm {(\MRone)}} -- {\rm {(\MRfour)}} hold  and $p=2$, then  Condition {\rm {(\MRfive)}} also holds.
\end{lemma}

\proof
Set $B=A^{G_i^\perp}$ for convenience. Condition {\rm {(\MRfour)}} ensures that the extension $R\subset B$ is flat.
Thus $B$ is a free $R$-module of rank $2$, and since $R$ is local, we can find a basis for $B$ as an $R$-module
of the form $1,u$, where $1$ is the unit element in $R$ and $u \in B$.  It follows that there exist $a, \xi \in R$ such that 
$u^2=au+\xi$.
 Then $B=R[u]/(u^2-au-\xi)$, and the action of $G/G_i^\perp$ is given by $u\mapsto u+a$.
Proposition \ref{effective model torsor} describes the effective model $\shG_i$ of the action of ${\mathbb Z}/2{\mathbb Z}$ on  $Y_i=\Spec(B)$, 
and shows that the latter scheme is a torsor under $\shG_i$. 
\qed

\begin{definition}
\label{definition moderately ramified}
We say that a $G$-action on $A$ is \emph{moderately ramified}
if it is ramified precisely at the origin, and 
there exists an admissible regular system of parameters $u_1,\ldots,u_n\in A$ with associated norm subring $R$
such that the five conditions (\MRone) -- (\MRfive)  hold.
\end{definition}

Theorem \ref{structure moderately ramified}  below is our main structure result for moderately ramified actions, and shows that such actions are easily described by explicit equations. Its converse in \ref{structure moderately ramified2} shows in addition that moderately ramified actions are abundant.
\begin{theorem}
\label{structure moderately ramified} 
Let $A$ be a complete   local ring that is regular of dimension $n \geq 2$, with field of representatives $k$ and endowed with 
a moderately ramified action of the cyclic group $G$ of order $p$. Then $A$, as a $k$-algebra with $G$-action, is isomorphic to 
$$
k[[x_1,\ldots,x_n]][u_1,\ldots,u_n]/(u_1^p-a_1^{p-1}u_1-x_1,\ldots,u_n^p-a_n^{p-1}u_n-x_n),
$$
where $x_1,\ldots,x_n$ are indeterminates, the elements  $a_1,\ldots,a_n\in k[[x_1,\ldots,x_n]]$
form a  system of parameters of $k[[x_1,\ldots,x_n]]$, and  the natural automorphism $\sigma$ of order $p$, which sends $u_i$ to $u_i+a_i$ for $i=1,\dots, n$
and fixes $k[[x_1,\ldots,x_n]]$, induces under this isomorphism a generator of $G$.

The elements $u_1,\dots, u_n$ form an admissible regular system of parameters for the action of $\left< \sigma\right>$, and the subring $R:=k[[x_1,\ldots,x_n]]$ is identified with the norm subring for the action of $G$. The Galois group $H$ of the extension $\Frac(A)/\Frac(R)$ 
is generated by the automorphisms $\sigma_i$, $i=1,\dots, n$, where $\sigma_i(u_j)=u_j$ if $i\neq j$, and $\sigma_i(u_i)=u_i+a_i$. 
Let ${G_i^\perp} $ be the subgroup generated by the elements $\sigma_j$ with $j \neq i$. Then
the invariant subring
$A^{G_i^\perp}$ introduced in {\rm \ref{flat, system parameters}} is isomorphic to $R[u_i]/(u_i^p-a_i^{p-1}u_i-x_i)$.

\end{theorem}

Before we give the proof of Theorem  \ref{structure moderately ramified} in \ref{proofstructure},
let us review the notion of {\it fixed scheme of the action} and mention two corollaries.
Let a finite group $G$ act on a ring $A$. The {\it fixed point scheme} of the $G$-action is the closed subscheme $\Spec A/I $ of $ \Spec A$, 
where $I$ is the ideal of $A$ generated by the elements of the form $\sigma(a)-a$, for all $a \in A$ and all $\sigma \in G$. 
Assume now that  $A=k[[u_1,\dots,u_n]]$  and $G$ acts by $k$-automorphisms. Then 
$I$ is the ideal of $A$ generated by the elements
$\sigma(u_i)-u_i$, for $i =1,\dots, n$ and $\sigma \in G$.
When the action of $G=\left< \sigma \right>$ on $k[[u_1,\dots, u_n]]$ is linear, with $\sigma(u_i):=\zeta_i u_i$ for 
$i =1,\dots, n$ and  $\zeta_i$ is a non-trivial $m$-th root of unity for some $m$ coprime to the characteristic of $k$, we find that the fixed point scheme is smooth, with ideal $I=(u_1,\dots, u_n)$.
\if false 
Suppose that $G$ is cyclic of order $p$  and ${\rm char}(k)=p$. When the ideal $I$ is principal, then the ring $A^G$ is regular (\cite{K-L}, Theorem 2), and it is conjectured in \cite{K-L}, Conjecture 9,
that if $A^G$ is regular, then $I$ is principal.
\fi
Let us record now the following immediate consequences of Theorem \ref{structure moderately ramified}.
 
\begin{corollary}
Let $A$ of dimension $n$ be endowed with a moderately ramified $G$-action, as in {\rm \ref{structure moderately ramified}}. 
Let  $I$ be the ideal of $A$ defining the fixed point scheme of the action.  Then  
$I \subseteq \maxid_A^p$ and the length of $A/I$ is a multiple of $p^n$.
\end{corollary}

\proof 
Consider the normal form for the ring $A$ obtained in Theorem \ref{structure moderately ramified}.
It is clear from the explicit form of the action of $G$ that $I=(a_1,\dots,a_n)A$.  Since  $a_i\in\maxid_R$ and  $u_i^p-a_i^{p-1}u_i - x_i =0$, we see
that $x_i\in\maxid_A^p$, so that $I \subseteq \maxid_A^p$.
We computed in \ref{norm elements} that the length of $A/(x_1,\dots, x_n)A$ is equal to $p^n$. 
It follows from the facts that $a_1,\dots,a_n \in R $ belong to $\maxid_R=(x_1,\dots, x_n)R$ and that $A/R$ is flat,
that the length of $A/(a_1, \dots, a_n)A$ as an $A$-module is a multiple of the length of $A/(x_1,\dots, x_n)A$ (\cite{Nag}, 19.1). 
\qed

\medskip
The proof for the following consequence is left to the reader:

\begin{corollary} 
Let $A_1$ and $A_2$ be two regular complete local rings 
of positive dimension $n_1$ and $n_2$, respectively, of characteristic $p>0$, and field of representatives $k$. 
Consider two moderately ramified actions of $G:=\ZZ/p\ZZ$, on $A_1$ and $A_2$. 
Then the natural diagonal action of $G$ on the completed tensor product $A_1 \hat{\otimes}_k A_2$ is also moderately ramified.
\end{corollary}
 
Next we explain how the torsor condition on $Y_i \to S$ in (\MRfive) leads to a simple equation for $Y_i$.

\begin{lemma} \label{normalformMR4} 
Consider a moderately ramified $G$-action on $A$ and keep the notation introduced in this section. 
Suppose that $Y_i \to S$ is a $\shG_i/S$-torsor.
Then $A^{G_i^\perp}$ is isomorphic to $R[r_i]/(f_i)$ for some polynomial $f_i(r_i)=r_i^p-a_i^{p-1}r_i - \xi_i$ with coefficients
 $a_i, \xi_i \in \maxid_{R}$.
\end{lemma}
\proof 
It follows
from the Tate--Oort classification recalled in \ref{tate--oort} and \ref{TateOortReview} that
$\shG_i/S$ is isomorphic to the group scheme $\shG_{\alpha_i,\beta_i}/S$, with $\shG_{\alpha_i,\beta_i}= R[T]/(T^p-\alpha_iT)$
for some elements $\alpha_i,\beta_i\in R$ with $\alpha_i\beta_i=0$ (since  $w_p=0$ when $R$ has characteristic $p$).  
By construction, the generic fiber of $\shG_i \to S$ is a constant group scheme.  Since $R$ is a normal domain, we conclude from 
Proposition \ref{generically split group scheme}   that 
$\beta_i=0$, $\alpha_i\neq 0$, and $\alpha_i $ is a $(p-1)$-th power in  $ R $, say $\alpha_i=a_i^{p-1}$ with $a_i \in R$.

It follows from \ref{naturalaction} and \ref{effective model torsor}
that the torsor $Y\ra S$ is isomorphic to the spectrum of $R[r_i]/(f_i)$ for some polynomial $f_i(r_i)=r_i^p-a_i^{p-1}r_i - \xi_i$.
The coefficient $a_i$ is not a unit, because $R\subset A^{G_i^\perp}$ is ramified at the origin.
The element $\xi_i$ becomes a $p$-th power in the residue field $k=R/\maxid_R$, because the
extension $R\subset A^{G_i^\perp}$ has trivial residue field extension.
Changing $\xi_i$ by such $p$-th power from the field of representatives $k\subset R$, we obtain $\xi_i\in \maxid_R$.
\qed

\begin{emp} \label{proofstructure} {\it Proof of Theorem \ref{structure moderately ramified}.}
Suppose that the $G$-action on $A=k[[u_1,\ldots,u_n]]$ is moderately ramified,
that is, Conditions   (\MRone) -- (\MRfive)  hold. We find  
from Proposition \ref{tensor product} that the map 
$f:A^{G_1^\perp}\otimes_R\cdots\otimes_R A^{G_n^\perp}\ra  A$ introduced in \eqref{map} is an isomorphism, and that each 
$A^{G_i^\perp}$ is in fact regular (\ref{flat, system parameters}). We may then apply  \ref{normalformMR4}  
to write  $A^{G_i^\perp}=R[r_i]/(f_i)$ for some  polynomial $f_i(r_i)=r_i^p-a_i^{p-1}r_i - \xi_i$ 
with coefficients $a_i, \xi_i \in \maxid_{R}$. 
\end{emp}

\begin{lemma}
\label{new regular system}
Keep all the preceding assumptions. Then the following hold:
\begin{enumerate}[\rm (a)]
\item The elements  $a_1,\ldots,a_n $ form a system of parameters of $R$.
\item The elements  $\xi_1,\ldots,\xi_n $ form a regular system of parameters of $R$.
\item The elements  $r_1,\ldots,r_n$ form  a regular system of parameters of $A$.
\end{enumerate}
\end{lemma}

\proof 
Part (a) is immediate from \ref{flat, system parameters}. 
To proceed, note that the canonical map
$$
R[r_1,\ldots,r_n]/(f_1,\ldots,f_n)\lra R[[r_1,\ldots,r_n]]/(f_1,\ldots,f_n)
$$
from polynomial rings to power series rings is bijective. This is because the polynomials $f_i(r_i)$
are monic, so  both sides are free $R$-modules of the same rank, with corresponding
bases of monomials.
Write  $\ideala\subset R[[r_1,\ldots,r_n]]$ for the ideal  generated by the polynomials $f_1,\ldots,f_r$. 
Using Proposition \ref{tensor product} we get an identification $A=R[[r_1,\ldots,r_n]]/\ideala$.
The  partial derivatives are
\begin{equation}
\label{partial derivatives}
\begin{split}
\frac{\partial}{\partial x_j}(r_i^p-a_i^{p-1}r_i-\xi_i)	&= 
  \frac{\partial a_i}{\partial x_j}a_i^{p-2}r_i - \frac{\partial\xi_i}{\partial x_j}, \quad {\rm and }\\
\frac{\partial}{\partial r_j}(r_i^p-a_i^{p-1}r_i-\xi_i)	&= 
  a_i^{p-1}.
\end{split}
\end{equation}
Consider the Jacobian matrix $J\in\Mat_{n\times 2n}(k[[x_1,\ldots,x_n,r_1,\ldots,r_n]])$  of partial derivatives.
Since $A$ is regular, this matrix
has maximal rank when its entries are reduced modulo the maximal ideal (use the Jacobian Criterion as in \cite{NagJac}, Theorem on page 429, and Remark 1 on page 431). 
So the same holds for the $n\times n$-matrix $(\partial \xi_i/\partial x_j)_{1 \leq i,j \leq n} \in\Mat_{n\times n}(R)$, because $a_i\in\maxid_R$.
By the Implicit Function Theorem (\cite{A 4-7}, Chapter IV, \S 4, No.\ 7, Corollary to Proposition 10),  
 the elements $x_1,\ldots,x_n$ can be expressed as formal power series in the $\xi_1\ldots,\xi_n$, which proves (b).
Using the relation $r_i^p-a_i^{p-1}r_i-\xi_i=0$, the elements $x_1,\ldots,x_n$
can be expressed as formal power series in the $r_1,\ldots,r_n$.
Thus $r_1,\ldots,r_n$ for a regular system of parameters of $A$, and (c) is proved.
\qed

\begin{emp} We can now conclude the proof of Theorem \ref{structure moderately ramified}.
We  use the   identification 
$$ k[[x_1,\ldots,x_n]][r_1,\ldots,r_n]/(f_1,\ldots,f_n) \lra A$$
to obtain an explicit isomorphism $(\ZZ/p\ZZ)^n \to H$: Define an injective group  homomorphism
\begin{gather*}
(\ZZ/p\ZZ)^n  \lra   \Aut_R(A) \quad \subseteq \Gal(\Frac(A)/\Frac(R))=H, \\
(h_1,\dots, h_n)   \longmapsto   \left( r_i\mapsto r_i+h_ia_i \right)_{1\leq i\leq n}. 
\end{gather*}
In the definition above, we identify $\ZZ/p\ZZ$ with the prime subfield of $k$.
Since $|H|=p^n$, we find then that this homomorphism is an isomorphism.

The subgroup $G$ of $H$  corresponds under this isomorphism to a subgroup 
of order $p$ generated by some $(h_1,\dots, h_n) \in  (\ZZ/p\ZZ)^n$.  Lemma \ref{no invariant parameter}
shows that every coefficient $h_i$ is non-zero. We thus have $h_i^{p-1}=1$ for $i=1,\dots, n$.
After replacing our choice of $a_i\in R$ by $h_ia_i$, we can assume that the subgroup $G$ of $H$ corresponds under the above bijection
with the diagonal subgroup,
generated by $(1,\ldots,1)$. With such choice, the generator of $G$ corresponding to $(1,\ldots,1)$ acts via
$$
r_i\longmapsto r_i+ a_i.
$$
It is easy to compute the norms of $r_i\in A$:
$$
N_{A/A^G}(r_i)=\prod_{\ell=0}^{p-1}(r_i+\ell a_i) = r_i^p-a_i^{p-1}r_i = \xi_i.
$$
\if false
Since  $a_i\in\maxid_R$ and  $r_i^p-a_i^{p-1}r_i - \xi_i =0$, we see
that $\xi_i\in\maxid_A^p$. Since $a_i$ can be expressed in terms of $\xi_1,\ldots,\xi_n$, 
it follows that $a_i \in\maxid_A^p$. Therefore, the action induced by $H$ on the space $\maxid_A/\maxid_A^p$ is trivial.
It follows in particular that the action induced by $G$ on $\maxid_A/\maxid_A^2$ is trivial, showing that every regular system of parameters
in $A$ is admissible (see \ref{lem.Borel}). 
\fi 
We replace our original regular system of parameters $u_1,\ldots,u_n$ by the regular system of parameters $r_1,\ldots,r_n$.
The norm elements $x_1,\ldots,x_n\in R$ for the initial regular system of parameters are then replaced  by $\xi_1,\ldots,\xi_n$.
Both regular systems of parameters produce the same norm subring  $R\subset A$. 
After this change of parameters, we  can write that $A=k[[u_1,\ldots,u_n]] $ and $R=k[[x_1,\ldots,x_n]]$, with 
$$
A=R[u_1,\ldots,u_n]/(u_1^p-a_1^{p-1}u_1-x_1,\ldots,u_n^p-a_n^{p-1}u_n-x_n).
$$
Moreover, the elements $(h_1,\dots,h_n)\in(\ZZ/p\ZZ)^n=H$ act on $A$ via
$u_i\mapsto u_i + h_ia_i$ and $G$ is generated by the element corresponding to $(1,\dots,1)$. 
\qed
\end{emp}

We now state the expected converse to Theorem \ref{structure moderately ramified}.
\begin{theorem}
\label{structure moderately ramified2}
Let $n\geq 2$. Let $k$ be a field of characteristic $p>0$. Let $R:=k[[x_1,\ldots,x_n]]$ and let $a_1,\dots, a_n \in \maxid_R \setminus \{0\}$.
 Consider the ring 
$$
A:=R[u_1,\ldots,u_n]/(u_1^p-a_1^{p-1}u_1-x_1,\ldots,u_n^p-a_n^{p-1}u_n-x_n).
$$
\begin{enumerate}[\rm (a)]
\item
Then $A$ is a regular complete local ring with maximal ideal $(u_1,\dots, u_n)$. The extension $\Frac(A)/\Frac(R)$ is Galois with Galois group $H$ isomorphic to $({\mathbb Z}/p{\mathbb Z})^n$, 
generated by the automorphisms $\sigma_i$, $i=1,\dots, n$, with $\sigma_i(u_j):=u_j$ if $j \neq i$, and $\sigma_i(u_i):=u_i +a_i$. We have $R=A^H$.
\item
Assume now that the elements $a_1,\ldots,a_n$
form a system of parameters in $R$. Let $\sigma$ be the automorphism of order $p$ of $A$ which sends $u_i$ to $ u_i+a_i$, for $i=1,\dots, n$
and fixes $R$ (in other words, $\sigma=\sigma_1 \cdots \sigma_n$). Then the action of $\langle \sigma \rangle$ on $A$ is moderately ramified.
\end{enumerate}
\end{theorem}
\proof
(a) Clearly, $R\subset A$ is a finite flat extension of degree $p^n$, so $\dim(A)=\dim(R)=n$ by Going-Up and Going-Down.
The fiber ring for this   extension is  $A/\maxid_RA=k[u_1,\ldots,u_n]/(u_1^p,\ldots,u_n^p)$, which is local, so $A$ is local by Hensel's Lemma.
Furthermore, the residue classes of the $x_i$ vanish in the cotangent space    $\maxid_A/\maxid_A^2$, so the
latter is generated by the residue classes of the $u_i$.
Hence $\edim(A)\leq\dim(A)$, so $A$ is regular.

It is clear that we can exhibit $p^n$ automorphisms of $A$ leaving $R$ fixed: for each $(c_1,\dots, c_n) \in ({\mathbb Z}/p{\mathbb Z})^n$, 
define $\sigma(c_1,\dots, c_n)$ to be the $R$-automorphism of order $p$ of $A$ which sends $u_i$ to $ u_i+c_ia_i$, for $i=1,\dots, n$.
Hence, $\Frac(A)/\Frac(R)$ is Galois, and we obtained an explicit group homomorphism
$H:=(\ZZ/p\ZZ)^n \to {\rm Gal}(\Frac(A)/\Frac(R))$.

Let us denote by $G_i$ the subgroup generated by $\sigma_i$, with $\sigma_i(u_i):=u_i+a_i$ and $\sigma_i(u_j):=u_j$ if $j \neq i$.
Let $G_i^{\perp}$ denote the subgroup of $H$ generated by the groups $G_j$, $j \neq i$.
Clearly, the ring of invariants $A^{G_i^\perp}$ contains $u_i$, and we obtain a natural ring homomorphism
$R[u_i]/(f_i) \to A^{G_i^\perp}$, which sends the class of $u_i$ to the class of $u_i$. 
As in the preceding paragraph, one sees that the source ring is a regular complete local ring
with field of fractions of degree $p$ over $\Frac(R)$. 
It follows that this natural homomorphism is an isomorphism. 
 
(b) Let $G:=\langle \sigma \rangle$.
The fixed scheme of the $G$-action on $X=\Spec(A)$ is defined by the ideal
$(a_1,\ldots,a_n)\subset A$. Since we assume in (b) that $a_1,\ldots,a_n\in R=k[[x_1,\ldots,x_n]]$ is
a system of parameters, it follows that the action of $G$ is ramified precisely at the origin.
The regular system of parameters $u_1,\ldots,u_n\in A$ is   admissible 
since the group $G$  induces the trivial representation on the cotangent space $\maxid_A/\maxid_A^2$ (\ref{lem.Borel}).
Clearly, $x_i=N_{A/A^G}(u_i)$, so that $R$ is indeed the norm subring of $A$ associated with the admissible system of parameters $u_1,\dots, u_n$.
It follows from (a) that $\Frac(A)/\Frac(R)$ is Galois, showing that (\MRone) holds.

Let us show now that each $G_i$ is the inertia group of a height one prime ideal of $A$. 
Indeed, let $\primid $ be minimal prime ideal in $A$ which contains $a_i$. 
This ideal $\primid$ has height one, and it is clear that $G_i \subseteq I_\primid$.
Since $a_1,\dots, a_n,$ is a system of parameters of $A$, we find that $\primid$ cannot contain 
$a_j$, for any $j \neq i$. It follows that no automorphism in $H$ but those in $G_i$ can induce the identity 
on $A/\primid$. Hence, $G_i = I_\primid$, as desired.

Consider now a prime ideal $\primid$ of height one in $A$ with $I_\primid\neq \{\rm id\}$.
We are going to show that $I_\primid=G_j$ for some $j$, which in particular implies that Conditions (\MRtwo) and (\MRthree) hold. 
Let $\tau \in I_\primid$ be a non-trivial element. Then the ideal $I(\tau)$ of the fixed scheme is contained in $\primid$. 
Since $\primid$ has height one, it follows that $I(\tau)$ can contain at most one of $a_1,\dots, a_n$. We find using our explicit description in (a) 
of the automorphisms of $\Frac(A)$ over $\Frac(R)$ that 
the ideal $I(\tau)$ is not trivial and contains at least one of of $a_1,\dots, a_n$ since $\tau$ is not trivial. It follows that we can assume  that $a_i \in I(\tau)$.
Considering again the descriptions of the automorphisms, we see that this implies that $\tau \in \left< \sigma_i\right>=G_i$. It follows that $I_\primid=G_i$.

Consider now the isomorphism $R[u_i]/(f_i) \to A^{G_i^\perp}$ found in (a). 
It is clear that $R\subset A^{G_i^\perp}$ is finite and flat, so (\MRfour) holds. 
The scheme  $\Spec A^{G_i^\perp} \to \Spec R$ 
is a torsor under the group scheme ${\mathcal G}_i:=\Spec R[z]/(z^p-a_i^{p-1}z)$ 
and this group scheme is the effective model of the action on $\Spec A^{G_i^\perp}$,
according to Proposition \ref{effective model torsor}. Thus (\MRfive) holds.
\qed

\medskip
There is one situation in which every action ramified precisely at the origin is moderately ramified, namely when $n=2$ and $p=2$, and $k$ does not admit any separable quadratic extension.
That $\mathbb Z/2\mathbb Z$-actions on $k[[x,y]]$ are very structured  was already recognized by Artin in \cite{Artin 1975},
leading to a description of the ring of invariants $A^G$ that we recall in \ref{2-dimensional invariant ring}.

\begin{proposition}
\label{artin case} Let $k$ be a field of characteristic $p=2$ which does not admit any separable quadratic extension.
Then, in dimension $n=2$, any $G$-action on  $A$
that is ramified precisely at the origin is moderately ramified.
\end{proposition}

\proof Proposition \ref{exist admissible} shows that 
 $A$ contains an admissible regular system of parameters $u_1,u_2$, with norm ring $k[[x_1,x_2]]$.
According to Theorem \ref{galois extension} (i), the field extension
$\Frac(A^G)/k((x_1,x_2)) $, which has  degree $p^{n-1}=2$,  is not purely inseparable. Hence  it must be Galois. We use our hypothesis on $k$ with \ref{galois extension} (ii) to find that  the extension $\Frac(A)/k((x_1,x_2))$ must be Galois as well.
So Condition (\MRone) holds. Then \ref{n=2} shows that (\MRtwo) holds.
Hence, the Galois group $H$ of order $p^n=4$ is $\ZZ/2\ZZ \times \ZZ/2\ZZ$, and 
this group has precisely three nontrivial cyclic subgroups, one of which must be $G$.
Since the action of $G$ is ramified precisely at the origin, we find that $G$ cannot 
be the inertia group in $H$ of a prime of height $1$. 
Proposition \ref{generated by inertia} 
shows that $H$ is generated by cyclic groups which are inertia groups of primes of height $1$. 
Thus $H$ contains exactly two subgroups which are inertia groups of primes of height $1$, and Condition (\MRthree) holds.  
In dimension $2$, Condition (\MRfour) is always satisfied.
Lemma \ref{MR1-2impliesMR5} ensures that 
(\MRfive) holds.
\qed

\if false
\begin{lemma} \label{lem.p=2modram}
\label{case p=2} Let $p=2$ and $n=2$. 
Let $A$ be endowed with a $G$-action 
that is ramified precisely at the origin and such that 
{\rm (\MRone)},  {\rm (\MRtwo)},  and {\rm (\MRthree)}, hold. Then  {\rm (\MRfour)} also holds.
\end{lemma}
\proof 
Keep the notation introduced in \ref{MR4}. Recall that $G_i$ denotes the inertia subgroup of a prime ideal ${\mathfrak p}_i \in \Spec A$, 
and we denote by $s_i$ the image of this prime under the natural map $\Spec A \to S$. By construction,  the branch locus of $Y_i \to S$
is the closure of the point $s_i$, of codimension $1$ in $S$. Since $n=2$, we find that each morphism $Y_i \to S$ is flat.
Let $\shG_i/S$ be the effective model of the $G_i$-action on $Y_i$.
This effective model is a finite flat group scheme over $S$, which is constant outside
the branch locus of $Y_i \to S$, and comes with an $S$-group scheme action
$\shG_i\times_S Y_i\ra Y_i$. 
For each $i=1,\dots, n$, we need to show that the scheme $Y_i/S$ is a torsor under $\shG_i/S$. 
Proposition \ref{torsors} shows that for $Y_i/S$ to be a torsor under $\shG_i/S$, it suffices 
that for each point $s\in S$ of codimension at most $ 1$, the fiber $Y_{i,s}$ is a $\shG_{i,s}$-torsor.
When $p=2$ and ${\mathcal O}_{S,s}$ is a complete discrete valuation ring, Theorem \ref{characterization moderately ramified}
shows that the latter condition is always satisfied.
\qed
\fi

\begin{remark} \label{pseudoreflection}
Let $B$ be any   noetherian local ring.
A ring automorphism $\sigma: B \to B$ of finite order is called a {\it generalized reflection}
if $(\sigma-{\rm id})(B) \subseteq (xB)$ for some regular non-unit $x\in B$.
In other words, the scheme of fixed points of $\sigma$ in $\Spec(B)$ contains an effective Cartier divisor.
The automorphism  $\sigma$ is called a {\it pseudo-reflection} if $(\sigma-{\rm id})(B)= (xB)$ for some regular non-unit $x\in B$.
In this case, the scheme of fixed points of $\sigma$ in $\Spec(B)$ is an effective Cartier divisor.
When the order  of $\sigma$ is invertible in $B$, this latter definition is equivalent 
to the more classical definition where one requires that $\sigma$ acts on 
the cotangent space $\maxid_B/\maxid_B^2$ through pseudo-reflections in the sense of linear algebra (see \cite{Avr}, (12) on page 168).

Note that if   $\sigma : B \to B$ is a generalized reflection (resp., a pseudo-reflection), 
then the conjugates  $g\sigma g^{-1}$ and   powers $\sigma^m\neq {\rm id}$,
are  also  generalized reflections (resp., pseudo-reflections), where $g:B \to B$ is any automorphism.
Indeed, $I(g\sigma g^{-1})=g(I(\sigma))$, and when $n\geq 2$ is the order of $\sigma$,
we have  $I(\sigma^{n-1}) \subseteq  \dots \subseteq I(\sigma^2) \subseteq I(\sigma)$.
Considering the same chain of inclusions with $\sigma $ replaced by $\sigma^{-1}$, we find that all these inclusions are in fact equalities. 

Let now $A$ be a complete regular local ring of dimension $n \geq 2$ endowed with a moderately ramified action of ${\mathbb Z}/p{\mathbb Z}$, as in Theorem \ref{structure moderately ramified}. Let $u_1, \dots, u_n \in A$ be as in the statement of \ref{structure moderately ramified}, and let $R:=k[[x_1,\dots, x_n]]$.
As in \ref{structure moderately ramified}, $\Frac(A)/\Frac(R)$ is Galois of order $p^n$, with Galois group $H$ generated by elements $\sigma_i$, $i=1,\dots, n$,
such that $\sigma_i(u_i)=u_i+a_i$, and $\sigma_i(u_j)=u_j$ if $i \neq j$. It is clear from this expression for $\sigma_i$ that the ideal $I(\sigma_i)$ is equal to $(a_iA)$, and is thus principal. Since we assume that $(a_1,\dots, a_n)$ is a system of parameters in $R$, each of these ideals is proper. It follows that  the group $H$ is generated by pseudo-reflections. 
The next proposition shows that when only Condition (\MRone) holds (i.e., when $\Frac(A)/\Frac(R)$ is Galois), then it is still possible to show that
$H$ is generated by generalized reflections.  This proposition
is nothing but a reformulation of a result of Serre (\cite{Serre 1967}, Th\'eor\`eme 2'). 
\end{remark}
\begin{proposition} \label{Serre} 
Let $B$ be a regular 
noetherian local ring with maximal ideal $\maxid_B$. Let $H$ be a finite 
group of ring automorphisms of $B$, and let $R:=B^H$. 
Assume that 
$B$ is a finitely generated $R$-module. 
The ring $R$ is then a noetherian local ring with maximal ideal $\maxid_R$, and we assume also that the natural map $R/\maxid_R \to B/\maxid_B$ is an isomorphism. Then $H$ is generated by 
generalized reflections 
if $R$ is regular.
\end{proposition}
\if false
\proof Let $H' $ denote the subgroup of $H$ generated by the pseudo-reflections in $H$. It is easy to check that 
in fact $H'$ is a normal subgroup of $H$. We have $R:=B^H \subseteq R':=B^{H'}$. To show that $H'=H$, it suffices
to show that $R=R'$. If $R \neq R'$, then the hypothesis that the natural map $R/\maxid_R \to B/\maxid_B$ is an isomorphism
implies that $\Spec R' \to \Spec R$ is ramified at $\maxid_{R'}$. By construction, $R'$ is integrally closed, and by hypothesis, $R$ is regular.
It follows from the Zariski--Nagata Purity Theorem that the ramification locus of  $\Spec R' \to \Spec R$ has codimension $1$. Thus, we will obtain a contradiction 
by showing the following claim: {\it $\Spec R' \to \Spec R$ does not ramify in codimension $1$}. Let ${\mathfrak q}'$ be an ideal of height $1$ in $R'$.
Let $\mathfrak p$ be a prime ideal of $B$ above $\mathfrak q$. We can consider the inertia group $I_{\frak p}$ in $H$, and the inertia group $I'_{\frak p}$ in $H'$.
Clearly, $I_{\frak p} \supseteq I'_{\frak p}$, and the image of  $I_{\frak p}/ I'_{\frak p}$ in $H/H'$ is the inertia group of $\frak q$. 
Thus our claim follows if we show that $I_{\frak p} \subseteq I'_{\frak p}$. Since $B$ is regular, every prime ideal of height $1$ is principal. 
Thus there exists $c \in B$ such that $\frak p=(cB)$. It follows that for each $\sigma \in I_{\frak p}$, we have $I(\sigma) \subseteq (cB)$, and so 
$\sigma \in I_{\frak p} \subseteq   I_{\frak p} \cap H' = I'_{\frak p}$, as desired. \qed
\fi
\begin{remark} 
Conjecture 9 in \cite{K-L} strengthens Proposition \ref{Serre} in certain cases. Indeed, 
let $B$ be a regular local ring with ${\rm char}(B/\maxid_B)=p>0$, and let $G=\left< \sigma\right>$ be a cyclic group of order $p$ acting
on $B$ by local automorphisms. 
With the definition of pseudo-reflection given in this article, \cite{K-L}, Conjecture 9,
states that: {\it If $B^G$ is regular, then $\sigma$ is a pseudo-reflection}. The authors of \cite{K-L} prove this conjecture when $p=2$ and $p=3$.

\if false
We  note here that \cite{K-L} assumes that $B^G$ is noetherian (this hypothesis is made in the abstract of the paper \cite{K-L})
but does not assume that $B$ is a finitely generated $B^G$-module. The authors claim that when $B^G$ is regular, then $B$ is a finitely generated $B^G$-module
of rank $|G|$ on page 69 (Theorem 2, proof of (d) implies (c)) and on page 70 (proof of Conjecture 9 in the cases $p=2$ and $p=3$). This statement seems incorrect since even in dimension $1$, it is possible for a Dedekind domain $A$ to have an integral closure $B$ in a finite extension which is not a finitely generated $A$-module. The arguments of the authors are correct once $B$ is assumed to be a finitely generated $B^G$-module, since then $B^G$ regular implies that $B$ is free over $B^G$ of rank $|G|$.
\fi

Theorem 2, (a) implies (d), in   \cite{K-L} 
proves the converse of the statement of Conjecture 9 above: {\it If $\sigma$ is a pseudo-reflection, then $B^G$ is regular}. We note below an example where 
$\sigma$ is a generalized reflection, but $B^G$ is not regular (see also \ref{generalizedreflectionCM}). We do not have an example of a generalized reflection acting 
on a regular ring $B$ such that $B^G$ is not Cohen--Macaulay, and it would be interesting to determine whether such example exists.
In the case where the order of $\sigma$ is coprime to the residue characteristic of $B$, the invariant ring $B^G$ is always Cohen--Macaulay (\cite{H-E}, Proposition 13). 

Theorem \ref{complete intersection subring} lets us consider an action of $G$ on $A:=k[[u,v]]$ which is not ramified precisely at the origin, 
but such that the ring $A^G$ can still be completely described. Indeed, 
let $R:=k[[x,y]]$. Let $a,b  \in {\maxid}_R$ with $(a,b)$ a ${\maxid}_R$-primary ideal. Let $\mu \in \maxid_R$ such that  
$(a,\mu)$  and $(b,\mu)$ are ${\maxid}_R$-primary  ideals.
Consider the ring $$A:=R[u,v]/(u^p-(\mu a)^{p-1}u-x, v^p-(\mu b)^{p-1}v-y).$$
Then $A$ is a regular complete local ring with maximal ideal $(u,v)$ (\ref{structure moderately ramified2} (a)). The extension $\Frac(A)/\Frac(R)$ is Galois with Galois group $H$ isomorphic to $({\mathbb Z}/p{\mathbb Z})^2$, 
generated by the automorphisms $\sigma_1$ and $\sigma_2$,  with 
$\sigma_1(u):=u+\m a$ and $\sigma_1(v)=v$, and $\sigma_2(v):=v+\m b$ and $\sigma_2(u)=u$. We have $R=A^H$.
The morphisms $\sigma_1$ and $\sigma_2$ are both pseudo-reflections, but the composite $\sigma:=\sigma_1 \sigma_2$   is only a generalized reflection: $I(\sigma)=(\m  a,\m  b)$. Let $G:=\left<\sigma\right>$.
It follows that  every prime ideal $\mathfrak p$ of $A$ containing $\m$
is ramified over ${\mathfrak p} \cap A^G$. 
 We claim that the ring $A^G$ is not regular. 
Indeed,
consider $z:=bu-av$, with the relation $f(z):= z^p - (\m ab)^{p-1}z -a^py+b^px=0$. Then
$A^G= k[[x,y]][z]/(f)$, and  this latter ring is singular at $(x,y,z)$ (\ref{complete intersection subring}).
Preliminary computations seem to indicate that the graph of the minimal desingularization $X \to \Spec A^G$ does not depend on $\mu$.
\end{remark}

\if false
  Let $k$ be a field of characteristic $p>2$. Let $R:=k[[x,y]]$ and let $\mm \in \maxid_R \setminus \{0\}$, coprime to both $x$ and $y$.
 Consider the ring 
$$
A:=R[u,v]/(u^p-(\mm x)^{p-1}u-x, v^p-(\mm y)^{p-1}v-y).
$$
Then $A$ is a regular complete local ring with maximal ideal $(u,v)$. The extension $\Frac(A)/\Frac(R)$ is Galois with Galois group $H$ isomorphic to $({\mathbb Z}/p{\mathbb Z})^2$, 
generated by the automorphisms $\sigma_i$, $i=1,2$, with $\sigma_i(u_j):=u_j$ if $j \neq i$, and $\sigma_1(u):=u+\mm x$, and $\sigma_2(v):=v+\mm y$. We have $R=A^H$.
The morphisms $\sigma_1$ and $\sigma_2$ are both pseudo-reflections, but the composite $\sigma:=\sigma_1 \sigma_2$   is only a generalized reflection: $I(\sigma)=(\mm x,\mm y)$.
Let $G:=\left<\sigma\right>$. We claim that the ring $A^G$ is not regular. 
Indeed, we find that the element $z:=yu-xv$ is $G$-invariant, and satisfies the equation $f(z):=z^p-(\mm xy)^{p-1}z-(y^px+x^py)=0$. 
The ring $A^G$ is isomorphic to the ring $k[[x,y]][z]/(f(z))$, and  this latter ring is singular at $(x,y,z)$ (\ref{complete intersection subring2}).
In this example, the action of $G$ is not ramified precisely at the origin. In fact, for any prime ideal $\frak p$ containing $\mm$,  
$I_{\frak p}=\left< \sigma_1, \sigma_2\right>$ is not cyclic.

\end{remark}
\fi
\if false

\begin{remark} 
Let $\sigma$ be a $k$-automorphism of finite order of the ring $A=k[[u_1,\dots, u_n]]$. Then $\sigma$ is called a {\it pseudo-reflection} in \cite{K-L}, Definition 1,
if there exists a regular system of parameters $v_1,\dots, v_n$ in $A$ such that $\sigma$ fixes $v_2,\dots,v_n$. We find that when $A$ is endowed with a moderately ramified action of ${\mathbb Z}/p{\mathbb Z}$, the Galois group of the associated extension $\Frac(A)/\Frac(R)$ is generated by pseudo-reflections.
\end{remark}
\fi

\section{Complete intersection subrings}
\label{complete intersection subrings}

Let as usual $A$ be a complete regular local noetherian ring of dimension $n\geq 2$,  
characteristic $p>0$, with field of representatives $k$, and 
endowed with the action of a cyclic group $G$ of order $p$.
\if false
When the action is ramified precisely at the origin (\ref{ramifiedprecisely}), Artin in \cite{Artin 1975} showed that when $p=2$ and $n=2$, the ring of invariants $A^G$ can be completely described in terms of generators and relations.
Peskin in \cite {Peskin 1983}  
considered certain special classes of actions when $p=3$ and $n=2$, and for these classes, was also successful in  describing $A^G$
 in terms of generators and relations. 
 \fi
 The main result in this section is Theorem \ref{2-dimensional invariant ring}, which shows that for the class of moderately ramified actions
 introduced in \ref{definition moderately ramified}, the ring $A^G$ can 
 be described  when $n=2$ in terms of generators and relations. 

\begin{emp} \label{normalformrecalled}
Let $R:=k[[x_1,\ldots,x_n]]$. Let $a_1,\ldots,a_n \in \maxid_R$
be such that $a_i$ and $a_j$ are coprime if $i \neq j$.
Let $\m \in R \setminus \{0\}$ be such that $\m$ and $a_i$ are coprime, for $i=1,\dots, n$. Consider the ring
$$
A:=R[u_1,\ldots,u_n]/(u_1^p- (\m a_1)^{p-1}u_1-x_1,\ldots,u_n^p-(\m a_n)^{p-1}u_n-x_n).
$$
Theorem \ref{structure moderately ramified2} (a) shows 
$A$ is a regular complete local ring with maximal ideal $(u_1,\dots, u_n)$. The extension $\Frac(A)/\Frac(R)$ is Galois with Galois group $H$ isomorphic to $({\mathbb Z}/p{\mathbb Z})^n$, 
generated by the automorphisms $\sigma_i$, $i=1,\dots, n$, with $\sigma_i(u_j):=u_j$ if $j \neq i$, and $\sigma_i(u_i):=u_i +\m a_i$. We have $R=A^H$.

Let $\sigma$ be the automorphism of order $p$ of $A$ which sends $u_i$ to $ u_i+\m a_i$, for $i=1,\dots, n$
and fixes $R$ (in other words, $\sigma=\sigma_1 \cdots \sigma_n$). 
When 
$(\m a_1, \dots, \m  a_n)$ is a system of parameters of $R$, 
then
 the action of $G:=\langle \sigma \rangle$ on $A$ is moderately ramified (\ref{structure moderately ramified2} (b)). Theorem \ref{structure moderately ramified} shows that any moderately ramified action of 
 $G$ corresponds to a ring $A$ as above with $\m=1$ and  $(\m a_1, \dots, \m  a_n)$  a system of parameters of $R$. When $\m \notin R^*$, the action is not ramified precisely at the origin. Our goal is to describe 
 the ring of invariants $A^G$ explicitly in all cases when $n=2$.
 
 Identify 
the Galois group $H$ of  $\Frac(A)/\Frac(R)$ with the elementary abelian $p$-group
$({\mathbb Z}/p{\mathbb Z})^n$,  where an element $(\nu_1,\nu_2,\ldots,\nu_n)$ corresponds to the $R$-algebra automorphism given by $u_{j}\longmapsto u_{j} + \nu_j \m a_j$.
Under our identification, $(1,\dots, 1)$ corresponds to the generator $\sigma$ of $G$.
 
Let $s=(s_1,\dots, s_n) \in ({\mathbb Z}/p{\mathbb Z})^n$  
be a non-zero element, and consider the subgroup $H_s$ of $H$ of all elements of $H$ perpendicular to $s$. 
In other words, 
$$H_s := \{ (\nu_1,\nu_2,\ldots,\nu_n) \mid \sum_{i=1}^n \nu_i s_i =0\}.$$
Let $$z_s:=(\prod_{\{ i \mid  s_i \neq 0\} } a_i) (s_1\frac{u_1}{a_1}+\ldots+s_n\frac{u_n}{a_n} ) \in A.$$
It is easy to verify that $z_s$ satisfies the following relation:
$$ z_s^p - (\m \prod_{\{ i \mid  s_i \neq 0\} } a_i)^{p-1} z_s - (\prod_{\{ i \mid  s_i \neq 0\} } a_i)^p  (s_1\frac{x_1}{a_1^p}+\ldots+s_n\frac{x_n}{a_n^p} )=0.
$$
The reader will check that $h \in H$ is such that $h(z_s)=z_s$ if and only if $h \in H_s$. 
In particular, $\Frac(A^{H_s})=\Frac(R)(z_s)$ has degree $p$ over $\Frac(R)$. 

For convenience, let 
$\alpha:=\prod_{\{ i \mid  s_i \neq 0\} } a_i$, and when $s_i \neq 0$, let $\alpha_i:=\alpha/a_i$. 
By definition, $\alpha \in R$, and $\alpha_i \in R$ when $s_i \neq 0$.
 With this notation, define $f(z) \in R[z]$ by
$$f(z):=  z^p - (\m \alpha)^{p-1} z - \sum_{\{ i \mid  s_i \neq 0\} } s_i\alpha_i^p x_i.$$
Since $f(z_s)=0$ and the degree of $z_s$ over $R$ is $p$, we find that $f(z)$ is irreducible in $\Frac(R)[z]$, and also in $R[z]$ by Gauss' Lemma. Since $R[z]$ is factorial because $R$ is, we find that $R[z]/(f(z))$ is a domain. 
We have a natural ring homomorphism $j:R[z]/(f(z)) \to A^{H_s}$, which sends the class of $z$ to $z_s$. 
After tensoring with $\Frac(R)$, the resulting homomorphism is surjective.
It is thus also bijective, since the source and target have the same dimension over $\Frac(R)$. 
Since both rings $R[z]/(f(z))$ and $A^{H_s}$ are integral domains,
the map $j$
is injective.
 \end{emp}
 
 \begin{theorem}\label{maximal subgroups}
Let $A$ be as in {\rm \ref{normalformrecalled}}.
 Let $H$ denote the Galois group of the associated 
extension $\Frac(A)/\Frac(R)$. 
Let $H_s$ denote a subgroup of order $p^{n-1}$ in $H$. 
Then the extension $A^{H_s}$ 
 is isomorphic to the ring $R[z]/(f(z))$, where $f(z)
 \in R[z]$ is given above. In particular, it is flat
over $R$. 
The group $H/H_s$ acts on the ring $A^{H_s}$ through  pseudo-reflections.
\end{theorem}
\proof
Since $R[z]/(f(z))$ is finitely generated and free over $R$, it is Cohen--Macaulay and, hence, satisfies the property $S_2$. 
We show below that $R[z]/(f(z))$ is regular in codimension $1$. Then from Serre's criterion, $R[z]/(f(z))$ is normal, and injects into $A^{H_s}$, which is also normal, with same field of fractions. It follows that $R[z]/(f(z)) \to A^{H_s}$ is an isomorphism.

Since $R$ is complete, we can identify $R[z]/(f)$ with $k[[x_1,\dots,x_n,z]]/(f)$. Consider the partial derivatives of $f$:
$$
f_z = -(\m \alpha)^{p-1}\quadand  
f_{x_i} = (\m \alpha)^{p-2}z(\m \alpha)_{x_i} - s_i\alpha_i^p,\quad i=1,\dots,n.
$$

Suppose that $\primid$ is a prime ideal in $R[z]$ containing $(f)$.
Assume that the localization of $R[z]/(f)$ at the prime ideal $\primid/(f)$ is not regular.
Then, according to the Jacobian Criterion, 
all the derivatives of $f$ belong to $\primid$. In particular,
$f_z \in \primid$, and so $\m \alpha\in\primid$.
We want to show that the prime $\primid/(f)$ has height at least $2$ in $R[z]/(f)$. 

Assume first that $p\geq 3$, and that $s_i \neq 0$. Then the condition $f_{x_i} \in \primid$ immediately implies  that  $\alpha_i \in\primid$.
In particular, $\alpha_i$ is not a unit.
Since $\alpha=a_i\alpha_i= \prod_{\{ j \mid  s_j \neq 0 \} } a_j$, and so there exists $\ell \neq i$
such that $s_\ell \neq 0$ and $a_\ell \in \primid$. But then   from $s_\ell \neq 0$ and
$f_{x_\ell} \in \primid$, it follows that $\alpha_\ell \in \primid$. Hence, among the factors of $\alpha_\ell$, we can find 
$a_m \in \primid$ for some $m \neq \ell$. By hypothesis, the ideal $(a_\ell, a_m)$ has height at least 2 since 
$a_\ell$ and $a_m$ are coprime
and, thus, $\height(\primid)\geq 2$.

 Assume now that $p=2$.  
In view of the fact that $\m \alpha\in\primid$, we will consider two cases: when $\alpha \in \primid$, and when $\m \in \primid$. 
{\it Let us start with the case where $\alpha \in \primid$.}
If two or more of the $a_j$'s that divide $\alpha$ belong to $\primid$, then as before $\height(\primid)\geq 2$ and we are done.
Thus we are reduced to consider only the case where $a_i\in\primid$ for some $i$ with $s_i \neq 0$ and $a_j \notin \primid$ for all $j \neq i$ such that $s_j \neq 0$. We show now that this case cannot happen. Recalling that $\alpha= \alpha_i a_i$, we find using the product rule that
$f_{x_i}=  z(\m \alpha)_{x_i} - s_i\alpha_i^2= z(\m \alpha_i)_{x_i}a_i +z(\m \alpha_i)(a_i)_{x_i} - s_i\alpha_i^2 $.
From $f_{x_i}, a_i \in \primid$ and $\alpha_i \notin \primid$, we obtain that   
\begin{equation} \label{eq.1}
z \m (a_i)_{x_i}+s_i \alpha_i \in\primid.
\end{equation} 
We conclude in particular from this last expression that $z \notin \primid$. 

Since $a_i$ divides $\alpha_j$ when $j \neq i$ and $s_j \neq 0$, 
we find that $s_j \alpha_j^2x_j \in \primid$.
By hypothesis, 
$f(z):=  z^2 - (\m \alpha) z - \sum_{\{ j \mid  s_j \neq 0\} } s_j\alpha_j^2 x_j$ belongs to $\primid$, 
and so 
\begin{equation}
\label{eq.2}
z^2 +s_i\alpha_i^2x_i \in \primid.
\end{equation} 
We are now ready to conclude as follows. First, \eqref{eq.1} shows that $(z \m (a_i)_{x_i}+s_i \alpha_i)^2 \in\primid$. 
Using that $s_i=1=s_i^2$ and \eqref{eq.2}, we obtain  that
$z^2(1+ (\m (a_i)_{x_i})^2 x_i) \in \primid$. 
This is a contradiction, since $(1+ (\m (a_i)_{x_i})^2 x_i)$ is a unit, and we noted above that $z \notin \primid$.

{\it Consider now the case where $\m \in \primid$.} If $a_i$ belongs to $\primid$ for some $i$, then $\height(\primid)\geq 2$ by our hypothesis 
that $\m$ is coprime to $a_i$. 
So let us assume that $ a_i \notin \primid$ for all $i$ with $s_i \neq 0$. 
From $f_{x_i}\in \primid$ and $\alpha_i \notin \primid$, we conclude that $z \m_{x_i} a_i + s_i\alpha_i \in \primid$ whenever $s_i \neq 0$. 
In particular, $z \notin \primid$.
From the relation $f=0$
we find that $z^2 + \sum_{\{ i \mid  s_i \neq 0\} } s_i\alpha_i^2x_i \in \primid$. It follows that 
$z^2(1+ \sum_{\{ i \mid  s_i \neq 0\} } (\m_{x_i} a_i)^2 x_i) \in \primid$. 
This is a contradiction as before, since $(1+ \sum_{\{ i \mid  s_i \neq 0\} } (\m_{x_i} a_i)^2 x_i)$ is a unit, and $z \notin \primid$.
This concludes the proof that $R[z]/(f(z)) \to A^{H_s}$ 
 is an isomorphism.

Pick any standard basis vector $e_i$ of $({\mathbb Z}/p{\mathbb Z})^n$ which does not belong to 
$H_s$ (i.e., such that $s_i \neq 0$). Then the group $H/H_s$ is cyclic of order $p$, generated by the image $\overline{e}_i$ of $e_i$, and acts on $A^{H_s}$. It is easy to check that $
\overline{e}_i(z_s) - z_s = s_i \m \prod_{s_j \neq 0} a_j$, so that the ideal of the fixed scheme of $\overline{e}_i$ is $I(\overline{e}_i)=\m (\prod_{s_j \neq 0} a_j)A^{H_s}$, showing that $\overline{e}_i$ is a pseudo-reflection as defined in \ref{pseudoreflection}.
\qed
 
\medskip
In the definition of a moderately ramified action on a  complete local regular noetherian ring $A$ of dimension $n \geq 2$, Condition (\MRfive) imposes 
some structure requirement on  $n$ subrings of $A^G$, denoted $A^{G_i^\perp}$ in \ref{MR4}, of rank $p$ over $R$. Our next corollary implies that all subrings
of the form $A^{H_s}$ of rank $p$ over $R$ satisfy the same structure requirement.

\begin{corollary} \label{maximal subgroups 2}
Let $A$ be a  complete local regular noetherian ring of dimension $n \geq 2$ and  characteristic $p>0$, with field of representatives $k$.
Assume that  $A$ is
endowed with a moderately ramified  action of a cyclic group $G$ of order $p$. Let $H$ denote the Galois group of the associated 
extension $\Frac(A)/\Frac(R)$. 
Let $H_s$ denote a subgroup of order $p^{n-1}$ in $H$. Then the extension $A^{H_s}$ 
 is isomorphic to a ring of the form $R[z]/(f(z))$, where $f(z) = z^p - \alpha^{p-1} z - \beta $, with $\alpha, \beta \in R$. In particular, it is flat
over $R$. The group $H/H_s$ acts on the ring $A^{H_s}$ through   pseudo-reflections.
\end{corollary}
\proof 
Theorem \ref{structure moderately ramified}  lets us identify $A$  with a ring of the form 
$$
R[u_1,\ldots,u_n]/(u_1^p- a_1^{p-1}u_1-x_1,\ldots,u_n^p-a_n^{p-1}u_n-x_n),
$$
where  $R:=k[[x_1,\ldots,x_n]]$ and $a_1,\ldots,a_n\in R$
is a  system of parameters of $R$. Theorem \ref{maximal subgroups} can then be applied.  
\qed

\medskip Keep the notation introduced in \ref{normalformrecalled}.
Recall that $ x_i:=N_{A/A^G}(u_i)=u_i^p-(\m a_i)^{p-1}u_i$ is a \emph{norm element}.
In light of the $G$-action $u_i\mapsto u_i+\m a_i$, the  elements
$$
z_{ij}:=a_iu_j - a_ju_i, \quad i \neq j
$$
are also clearly $G$-invariant. We call these elements   \emph{minor elements}. Denoting by $e_1,\dots, e_n$ the standard vectors of $({\mathbb Z}/p{\mathbb Z})^n$
and setting $s=e_j-e_i$, we find that $z_{ij}=z_s$ (notation as in \ref{normalformrecalled}). 
The minor elements satisfy the obvious relations:
\begin{equation}
\label{minor relations}
\begin{array}{rl}
z_{ij}^p -(\m^{p-1} a_i^{p-1}a_j^{p-1}z_{ij} + a_i^px_j - a_j^px_i)=0, & \quad 1\leq i<j\leq n, {\rm \ and } \\
a_iz_{jk} - a_jz_{ik} + a_kz_{ij}=0, &\quad 1\leq i<j<k\leq n.
\end{array}
\end{equation}

When $n=2$, there is only one interesting element $z_{ij}$, namely $z_{12}$, and Theorem \ref{maximal subgroups} shows that it generates $A^G$. 
Our next theorem shows that $A^G$ in this case can be explicitly described.
\begin{theorem}
\label{2-dimensional invariant ring}
Let $A$ be a  complete local regular noetherian ring of dimension two and characteristic $p>0$, with field of representatives $k$.
Assume that  $A$ is
endowed with a moderately ramified  action of a cyclic group $G$ of order $p$. 
Then there exists a system of parameters $a,b$ in $k[[x,y]]$ such that $A^G$ is isomorphic to the domain
$$
k[[x,y,z]]/(z^p-a^{p-1}b^{p-1}z -a^py+b^px).
$$
Conversely, for any system of parameters $a,b\in k[[x,y]]$, the above
ring is the ring of invariants of a moderately ramified
$G$-action on some complete regular local noetherian ring $A$ of dimension two.
\end{theorem}
\proof The first part of the statement follows immediately from \ref{maximal subgroups 2}.
For the converse, we use the ring $A:=
k[[x,y]][u,v]/(u^p- a^{p-1}u-x,v^p-b^{p-1}v-y)
$ and apply \ref{structure moderately ramified2} to find that $A$ is regular.
\qed
 
\begin{remark}
When $p=2$ and $n=2$, we noted in \ref{artin case} that every action ramified precisely at the origin is in fact already moderately ramified
when $k$ has no separable quadratic extensions.
Putting this result together with Theorem \ref{2-dimensional invariant ring}, we recover Artin's description in \cite{Artin 1975} of the ring $A^G$ when $n=2$, $p=2$, and
$k$ has no separable quadratic extensions. Note that this hypothesis on $k$ does not appear in Artin's description in \cite{Artin 1975}, and we have not been able to provide a proof of this description without such hypothesis on $k$. 
\end{remark}

\begin{remark} Let $A$ be a  complete local regular noetherian ring of dimension two and characteristic $p>0$, with field of representatives $k$.
Assume that  $A$ is
endowed with a moderately ramified  action of a cyclic group $G$ of order $p$, and consider the ring $A^G$ as described in Theorem  \ref{2-dimensional invariant ring}.
Let $I$ denote the ideal of $A^G$ generated by $x$ and $y$. Then $I \neq IA \cap A^G$. Indeed, the element $z \in A^G$ does not belong to $I$
since $A^G/I$ is isomorphic to $k[z]/(z^p)$. On the other hand, since $(a,b) \subseteq I$ and $z=av-bu$, we find that $z \in IA \cap A^G$. 
This example generalizes Example 2 in \cite{Gil}. Note that it follows that $A^G$ is not a direct summand of the $A^G$-module $A$ (see \cite{Hoc}, Proposition 1, or \cite{H-E}, Proposition 10).
\end{remark}

\begin{remark} Axioms (\MRone)  and (\MRtwo) in the definition of moderately ramified action 
specify the existence of a regular system of parameters $u_1,\dots,u_n$ in the complete regular local ring $A$ such that the associated norm subring
$R:=k[[x_1,\dots,x_n]]$ is such that $\Frac(A)/\Frac(R)$ is a Galois extension of degree $p^n$ with elementary abelian Galois group (see \ref{elementary abelian}). In particular, 
any moderately ramified action $\sigma:A \to A$ comes equipped with a subgroup of ${\rm Aut}_k(A)$ isomorphic to $({\mathbb Z}/p{\mathbb Z})^n$. We note below that when $p=2=n$, much more is true. 
 
Assume that $n=p=2$ and that $k$ is algebraically closed. Let $\sigma:A \to A$ be a moderately ramified action. Consider a regular system of parameters $u,v$ such that $A=k[[u,v]]$, $R:=k[[x,y]]$ is the norm subring,
and  $a,b \in k[[x,y]]$ are such that $\sigma(u)=u+a$ and $\sigma(v)=v+b$. Clearly, there are then two non-trivial involutions in ${\rm Aut}_k(A)$
which commute with $\sigma$, namely the involution which fixes $v$ and sends $u$ to $u+a$, and the involution which fixes $u$ and sends $v$ to $v+b$.
We show below that in fact, {\it the centralizer of $\left< \sigma \right>$ in the group 
${\rm Aut}_k(A)$ contains infinitely many non-trivial involutions}.  

For each $c \in k^*$, consider the regular system of parameters $u+cv$, $v$, with norms
$X$ and $Y=y$ with respect to $\sigma$, where
$$X:=(u+cv)(\sigma(u)+c\sigma(v))= x + c^2y+ c(u\sigma(v) + v \sigma(u)).$$ 
For the regular system of parameters $u+cv$, $v$, the norm subring is $k[[X,Y]]$.
Since $u,v$ is an admissible regular system of parameters (\ref{admissible}) and $a,b \in (u,v)^2$, one checks that
$u+cv$, $v$, is also an admissible regular system of parameters, so the homomorphism $k[[X,Y]] \to A$ has degree $4$. 
Theorem \ref{galois extension} (i) shows that the field extension
$\Frac(A^G)/k((X,Y)) $, which has  degree $p^{n-1}=2$,  is not purely inseparable. Hence  it must be Galois. Since $k$ is algebraically closed,
\ref{galois extension} (ii) implies that  the extension $\Frac(A)/k((X,Y))$ must be Galois as well.
 We find that $k[[x,y]]=k[[X,Y]]$ if and only if $(u\sigma(v) + v \sigma(u))\in k[[x,y]]$. In our case, we can compute explicitly that $u\sigma(v) + v \sigma(u)= av+bu$, which is nothing but the minor element $z=z_{12}$  generating $A^G$ over $k[[x,y]]$.  Hence, the extension $\Frac(A)/\Frac(R)$ and $\Frac(A)/ k((X,Y))$ are distinct, and thus correspond to two distinct elementary abelian subgroups $H_0$ and $H_c$  in ${\rm Aut}_k(A)$ of order $4$, intersecting in $\left< \sigma \right>$. It is not hard to check that   the groups $H_c$, $c \in k$, are pairwise distinct.
\end{remark}

\if false
\begin{remark} A moderately ramified action of a cyclic group $G$ of order $p$ on $A$ of dimension $n$ comes along with 
$(p-1)^{n-1}-1$ other actions of $G$ which are also moderately ramified. When $n=2$, choose a system of parameters $u,v$
so that the action of $G$ on $A=k[[u,v]]$ is in normal form as in \ref{normalformrecalled}. Then the $p-1$ actions are labeled each by an element 
$c \in {\mathbb F}_p^*$, 
with automorphism of order $p$ given by $u \mapsto u+a$ and $v \mapsto v+cb$. 
Associated to this automorphism 
is the ring of invariants $$
B_c:=k[[x,y,z]]/(z^p-a^{p-1}b^{p-1}z -a^py+cb^px).
$$
When $p>2$, it is natural to wonder whether these rings are isomorphic as $k$-algebras.

Let $d \in {\mathbb F}_p^*$. The ring $B_{cd^2}$ is easily shown to be isomorphic to the ring $B_{c}$ over the separable closure $k_s$ of $k$
in two cases. In both cases, we will use  a change of variables $\tau$ given by $x \mapsto \delta x$, $y \mapsto \mu y$, and $z \mapsto z$,
with $\delta, \mu \in k_s$. 
Assume that $(a,b)=(x,y)$. Let $\delta \in k_s$
satisfy the equation $\delta^{p-1}=d$ and set $\mu:=1/(d\delta)$. It is easy to see that with these choices for $\delta$ and $\mu$, 
we have $cd^2 \mu^p \delta =c$, $\delta^p \mu=1$, and $\delta^{p-1}\mu^{p-1}=1$, and so $B_{cd^2}$ is isomorphic as a $k(\delta)$-algebra to $B_c$.
Assume now that $(a,b)=(y,x)$. Let $\delta :=1/d$ and set $\mu:=1$. Then $B_{cd^2}$ is isomorphic as a $k$-algebra to $B_c$.

\if false
Both cases are similar, and so we present below only the case where $(a,b)=(x,y)$.
If the polynomial $ z^p-a^{p-1}b^{p-1}z -a^py+cb^px$ is sent to the polynomial $z^p-a^{p-1}b^{p-1}z -a^py+b^px$ under the map $\tau$, 
then $c \mu^p \delta =1$, $\delta^p \mu=1$, and $\delta^{p-1}\mu^{p-1}=1$. In particular, $(\delta\mu)^{p-1}=1$ implies that $\delta\mu \in k^*$.
Combining $c \mu^p \delta =1$ and $\delta^p \mu=1$, we find that $c (\delta \mu)^{p+1}  =1$, and so $c$ is a square in $k^*$, say $c=d^2$ for some $d \in k^*$. 
From $c (\delta \mu)^{p+1}  =1$ and $\delta^{p-1}\mu^{p-1}=1$, we find that $c(\delta \mu)^2=1$. 
\fi

\end{remark}
\fi
\begin{emp} \label{Aci}
Let $A_\ci$ denote the   $R$-subalgebra of $A^G$   generated
by the $n-1$ minor elements $z_{12},\ldots,z_{1n}$. We show in \ref{complete intersection subring} that
this subring is a complete intersection which gives us a useful approximation of $A^G$.
Call the $R$-subalgebra generated by
all minor elements $z_{ij}$, $1\leq i<j\leq n$, the  \emph{minor subring} $A_\mnr$ of $ A^G$. This subring captures the regular locus
of $\Spec(A^G)$ (\ref{minor subring}).

For the definition of the ring homomorphism below, regard $z_{ij}$ as an indeterminate, 
and set $r_{ij}:= z_{ij}^p-(\m a_i a_j)^{p-1}z_{ij}  -a_i^px_j+a_j^px_i$, viewed as an element of the polynomial ring $R[z_{ij}]$.
Set $B:=R[z_{12},\ldots,z_{1n}]/(r_{12},\ldots,r_{1n})$. 
Then we have a natural homomorphism of $R$-algebras:
\begin{equation}
\label{subring ci}
B=R[z_{12},\ldots,z_{1n}]/(r_{12},\ldots,r_{1n})\lra A,
\end{equation}
whose image is $A_\ci$. We show in Theorem \ref{complete intersection subring} that the homomorphism $B \to A_\ci$ is an isomorphism.
\end{emp}
\if false
We will need the following facts.
\begin{lemma} \label{integra}
The ring $R[z_{ij}]/(r_{ij})$ is a domain. 
\end{lemma}
\proof The kernel of the natural map $R[z_{ij}] \to A$, with $z_{ij} \mapsto z_{ij}$ is a prime ideal which contains $r_{ij}$. 
Since $R$ is a UFD, so is $R[z_{ij}]$. Thus, to show that $(r_{ij})$ is a prime ideal, it suffices to show that the polynomial $r_{ij}$
is irreducible in $R[z_{ij}]$. If it is not irreducible, then any prime factor has degree $1$ and, thus, the image of $R[z_{ij}] \to A$ would be the subring $R$ in $A$. This is a contradiction, since $\Frac(R[z_{ij}])$ has degree $p$ over $\Frac(R)$. \qed
 
\begin{lemma} \label{codimension1}
Let $a,b$ be part of a system of parameters of $R$. Let $\m \in R \setminus \{0\}$, coprime to both $a$ and $b$. 
Let $f:=z^p-(\m a b)^{p-1}z-a^px_j+b^px_1$, with $j \neq 1$.
Then the ring $R[z]/(f)$ is regular in codimension one.
\end{lemma}
\begin{proof} 
Without loss of generality, we can assume that we are in the situation of \ref{integra}, and so $R[z]/(f)$ is a domain of rank $p$ over $R$. 
Since $R$ is complete, we can identify $R[z]/(f)$ with $k[[x_1,\dots,x_n,z]]/(f)$. Relabel $x:=x_1$, 
and consider the following  partial derivatives of $f$:
\begin{align*}
f_z &= -(\m a b)^{p-1}, \text{ and }   \\
f_x &= (\m ab)^{p-2}z(\m ab)_x + b^p, \text{ and } \\ f_{x_j} &= (\m ab)^{p-2}z(\m ab)_{x_j} - a^p.
\end{align*}
Suppose that $\primid$ is a prime ideal of $k[[x_1,\dots,x_n,z]]/(f)$ such that the localization $(k[[x_1,\dots,x_n,z]]/(f))_\primid$ is not regular.
Then, according to the Jacobian Criterion, we have
$f_z \in \primid$, and in particular $\m ab\in\primid$. 
When $p\geq 3$, the condition $f_x,  f_{x_j} \in \primid$ immediately implies  that $a, b\in\primid$, and thus $\height(\primid)\geq 2$ if $\primid$ is not regular, as desired.

Assume now that $p=2$.  Consider first the case where $a\in\primid$. 
If we also have $b\in\primid$, then as before $\height(\primid)\geq 2$. 
From $f_x\in \primid$, we find that either $b\in\primid$ or $z(\m a)_x+b \in\primid$. 
Thus, if $b\in\primid$, the proof is complete. Assume then that $b\notin\primid$. It follows that $z(\m a)_x+b \in\primid$, and so $z\notin\primid$.
Since $f+\m a b z=z^2+b^2x \in \primid$, we find that $z^2(1+(\m a)_x^2x) \in\primid$. This is not possible since $1+(\m a)_x^2x$ is a unit in $R$.

The second case where $b\in\primid$ is similar. Consider then the last case where $\m \in \primid$. If either $a$ or $b$ belongs to $\primid$,
then $\height(\primid)\geq 2$ by our hypothesis in $\m$. So let us assume that $a$ and $b$ are not in $\primid$. It follows that 
$z\m_x a   + b$ and $z\m_{x_j} b   + a$ both belong to $\primid$. In particular, $z \notin \primid$. Since $z^2-a^2x_j+b^2x \in \primid$, we conclude that 
$z^2(\m_x^2 a^2x+\m_{x_j}^2 b^2y+1) \in \primid$. As before, this is a contradiction since $(\m_x^2 a^2x+\m_{x_j}^2 b^2y+1)$ is a unit in  $R$.
\end{proof}
\fi

\begin{lemma} \label{AciNormal}
 The ring $B$ is a complete intersection, free of rank $p^{n-1}$ over $R$. The ring $B[1/(\m a_1)]$ is normal.
\end{lemma}
\begin{proof}
Consider the natural isomorphism
$$
\bigotimes_{j=2}^n R[z_{1j}]/(r_{1j}) \longrightarrow B:= R[z_{12},\ldots,z_{1n}]/(r_{12},\ldots,r_{1n})
$$
Since the polynomial rings $R[z_{1j}]$ are domains and each $r_{1j}$ is non-zero,
each tensor factor on the left and, hence, also $B$,  is a complete intersection (\cite{Maj}, Theorem 2).
Each tensor factor is free of rank $p$ over $R$, so $B$ is free of rank $p^{n-1}$ over $R$.

Consider now the fibers of the map $\Spec(R[z_{1j}]/(r_{1j}))\ra\Spec(R)$.
The Jacobian Criterion tells us that the fibers over points outside of the closed subset $V(\m a_1a_j)$ are etale.
Theorem \ref{maximal subgroups}
shows that $\Spec(R[z_{1j}]/(r_{1j}))$ is regular in codimension $1$.
Let $\mathfrak q$ be a prime ideal of height $1$ in $B$ that does not contain $\m a_1$. Then $\mathfrak q$ can contain at most one $a_j$ with $j \neq 1$.
When $\mathfrak q$ does not contain any $a_j$, $j \neq 1$, the map $\Spec B \to \Spec R$ is etale at $\mathfrak q$, and thus $B_{\mathfrak q}$ is regular.
When $\mathfrak q$ contains $a_j$ for some $j \neq 1$, let $B_j$ denote the natural subring of $B$ isomorphic to $R[z_{1j}]/(r_{1j})$.
The map $\Spec B \to \Spec B_j$ is etale at $\mathfrak q$, and thus $B_{\mathfrak q}$ is regular since $B_j$ is regular in codimension $1$.
We conclude that $B[1/(\m a_1)]$ is regular in codimension $1$. 
Since $B$ is a complete intersection, it is Cohen--Macaulay and, hence, satisfies Condition $S_2$. It follows that  $B[1/(\m a_1)]$ is normal.
\end{proof}

\if false
The group $(\ZZ/p\ZZ)^{n-1}$ acts on $B$ as follows: the element $(\nu_2,\ldots,\nu_n) \in (\ZZ/p\ZZ)^{n-1}$ corresponds to the $R$-algebra automorphism $B \to B$ 
given by $$
z_{1j}\longmapsto z_{1j} + \nu_ja_1a_j.
$$
Let $r:=\prod_{i=1}^n a_i$.
Since $\Spec(R[z_{1j}]/(r_{1j}))\ra\Spec(R)$ is a torsor under ${\mathbb Z}/p{\mathbb Z}$ outside of the closed subset $V(a_1a_j)$, we find that
 $\Spec(B[1/r])$ is a torsor over $\Spec(R[1/r])$ under the action of
the constant group scheme $(\ZZ/p\ZZ)^{n-1}$.

Recall that we also have an action of $(\ZZ/p\ZZ)^{n}$ on $A$, where an element $(\nu_1,\nu_2,\ldots,\nu_n)$ corresponds to the $R$-algebra automorphism 
given by $
u_{j}\longmapsto u_{j} + \nu_j a_j$. The ring $A$ is thus also equipped with an action of $(\ZZ/p\ZZ)^{n-1}$, by viewing $(\ZZ/p\ZZ)^{n-1}$ as the subgroup of   $(\ZZ/p\ZZ)^{n}$ consisting of the elements of the form 
$(0,\nu_2,\ldots,\nu_n)$, and endowing $A$ with the induced action. It is easy to check that the morphism $B \to A$ is equivariant with respect
to these actions.
\fi
\begin{theorem}
\label{complete intersection subring}
The   homomorphism  $B \to A_\ci$ induced by {\rm \eqref{subring ci}} is an isomorphism. 
The ring $A_\ci$ is a complete intersection domain,
with $A_\ci[1/(\m a_1)]=A^G[1/(\m a_1)]$. When $n=2$, $A_\ci = A^G$.
When $n\geq 3$, the ring $A_\ci$ is not regular in codimension one, and $A_\ci \neq A^G$.
\end{theorem}
\proof
Let us show first that the morphism $B\ra A$ 
is injective. For this, let $B_{ij}:= R[z_{ij}]/(r_{ij}) $. 
We will show that both natural maps
$$
B_{12} \otimes_R \dots \otimes_R  B_{1n} \longrightarrow \Frac(B_{12}) \otimes_{\Frac(R)} \dots \otimes_{\Frac(R)}  \Frac(B_{1n})
$$
and
$$\Frac(B_{12}) \otimes_{\Frac(R)} \dots \otimes_{\Frac(R)}  \Frac(B_{1n}) \lra \Frac(A)$$
are injective.
The first map is injective because each $B_{1j}$ is free (and thus flat) over $R$. 
To show that the second map is injective, we consider its source and target as finitely generated vector spaces over $\Frac(R)$. 
We find that the dimension of the source is $p^{n-1}$, and that the image is the smallest subfield of $\Frac(A)$ generated by $\Frac(R)$ and $z_{12}, \dots, z_{1n}$. It is easy to check that this subfield has dimension $p^{n-1}$, since $H_{12} \cap \dots \cap H_{1n}=\left< \sigma \right>$.
\if false
Let us show first that the morphism $B\ra A$ 
is injective. Since $R$ is a domain and $B$ is a free $R$-module, multiplication  in $B$ by $r \in R \setminus \{0\}$ is injective.
Thus $B \to A$ is injective if $B[1/r] \to A[1/r]$ is injective. We apply this remark to the case where $r:=\m \prod_{i=1}^n a_i$, 
and prove now that $B[1/r] \to A[1/r]$ is injective.

 *******
 
Recall from XX? that we have a tensor decomposition 
$$
A =    \bigotimes_{j=1}^n A^{G_j^\perp} = \bigotimes_{j=1}^n R[u_j]/(u_j^p-(\m a_i)^{p-1}u_j - x_j);
$$
our tensor factor $R[z_{1j}]/(r_{1j})$ maps into the subring $A^{G_1^\perp}\otimes A^{G_j^\perp}$.
As explained above, the map $B\ra A$ is equivariant.
Using that $B$ becomes a torsor after inverting $r:=a_1\ldots a_n$,
we infer that $B[1/r]\ra A[1/r]$ is injective, so that $B \to A$ is injective. 

*******

\fi

Since $A$ is a domain and $B \to A$ is injective, we find that $B$ is a domain. 
The injection $B\subset A^G$ induces a bijection on field of fractions.
As $A^G$ and $B[1/\m a_1]$ are both normal (\ref{AciNormal}), we find that the inclusion $B[1/\m a_1] \to A^G[1/\m a_1]$ is an isomorphism.

The case $n=2$ is treated in Theorem \ref{maximal subgroups}.
To prove the last statement in \ref{complete intersection subring}, suppose that $n\geq 3$. Then Proposition \ref{properties invariant ring}
shows that $A^G$ is not Cohen--Macaulay. Being a complete intersection, the ring $B$ is Cohen--Macaulay.
If $B$ were regular in codimension one, then it would be normal since it is $S_2$,  and the inclusion $B\subset A^G$
would be an equality, a contradiction.
\qed

\begin{corollary}
\label{minor subring}
The morphism of schemes $\Spec(A^G)\ra \Spec(A_\mnr)$ is an isomorphism outside the closed points.
\end{corollary}

\proof
By definition, we have $A_\ci\subset A_\mnr\subset A$. Using Theorem \ref{complete intersection subring},
we infer that $A_\mnr$ becomes normal after inverting $a_1$.
But it also contains the other complete intersection subrings,
defined with $z_{i1},\ldots,\widehat{z}_{ii},\ldots,z_{in}$ for  
$1\leq i\leq n$. Whence $A_\mnr$ becomes normal after inverting any of the $a_i$.
Since $a_1,\ldots,a_n\in R$ is a system of parameters, we conclude that the localizations
$(A_\mnr)_\primid$ are normal for any non-maximal prime ideal $\primid$.
It follows that $\Spec(A^G)\ra \Spec(A_\mnr)$ is an isomorphism outside the closed points.
\qed

\section{The invariant ring in higher dimension}
\label{linear situation}

Let $k$ be a ground field of characteristic $p>0$, and consider the polynomial ring in $2n$ variables
$$
B:=k[u_1,\ldots,u_n,a_1,\ldots,a_n],
$$
endowed with the action of the cyclic group $G$ of order $p$ given by identifying a generator of $G$ 
with the $k$-linear automorphism $\sigma$ of order $p$ defined by
$$
u_i\longmapsto u_i + a_i \quadand a_i\longmapsto a_i, \quad1\leq i\leq n.
$$
The subring $B^G$ of 
$ B$ is an
object extensively studied in modular representation theory. In this section, we review some known results on the structure of $B^G$ in \ref{generators vector invariants}, and  use
them to obtain information on the invariant subring $A^G$ for certain moderately ramified group actions on complete local rings $A$ obtained as quotients of the completion $\widehat{B}$ of $B$ at $(u_1, \dots, u_n, a_1,\dots, a_n)$.
We also provide in 
\ref{generalizedreflectionCM} an example where the integral closure of a 
Cohen--Macaulay local ring $R$ in a Galois extension $L/\Frac(R) $ generated by a generalized reflection of prime order $p$
is not  Cohen--Macaulay (see \ref{IntegralClosureAbelianExt} for a related example). 

Generators for  the invariant ring $B^G$ have been determined.
As in the previous section, the norm elements and  the minor elements 
\begin{gather*}
x_i:=N_{B/B^G}(u_i)=u_i^p-a_i^{p-1}u_i, \quad 1 \leq i \leq n,\\
z_{ij}:=a_iu_j-a_ju_i, \quad 1 \leq i<j\leq n
\end{gather*}
are clearly $G$-invariant. Additional natural $G$-invariant elements are  the traces
$$
t_\epsilon: =\Trace_{B/B^G}(u_1^{\epsilon_1}\cdot \ldots \cdot u_n^{\epsilon_n}) =
\sum_{\nu=0}^{p-1} (\prod_{i=1}^n(u_i + \nu a_i)^{\epsilon_i}),
$$
where $\epsilon=(\epsilon_1,\ldots,\epsilon_n) \in \NN^n$. 
For the purpose of generating $B^G$, we recall below that it suffices to consider only the $n$-tuples
 $(\epsilon_1,\ldots,\epsilon_n)$
subject to the conditions
$$
0\leq \epsilon_i\leq p-1 \quadand \sum_{i=1}^n\epsilon_i > 2p-2.
$$
Let us call such a tuple \emph{relevant}.
Note that there are no relevant tuples  when $n=2$. 
When $n=3$ and $p=2$, there is only one relevant tuple, namely $\epsilon=(1,1,1)$.
The element $t_\epsilon\in B^G$ attached to a relevant tuple $\epsilon$
is called \emph{trace element}. 

Write $\widehat{B}$ for the formal completion of the polynomial ring $B$ with respect to the maximal 
ideal $\maxid=(u_1,\dots, u_n,a_1,\dots, a_n)$. Then  
$\widehat{B}^G$ coincides with the formal completion of $B^G$ with respect
to $\maxid\cap B^G$.
Recall that the \emph{embedding dimension} of a local noetherian ring $C$ is the
vector space dimension of the \emph{cotangent space} $\maxid_C/\maxid_C^2$ over
the residue field $\kappa=C/\maxid_C$.

\begin{proposition}
\label{generators vector invariants}
The ring $B^G$ is generated as a $k$-algebra by the indeterminates $a_i$, 
the norm elements $x_i$, the minor elements $z_{ij}$, and the  
trace elements $t_\epsilon$. These elements yield a basis for
the cotangent space $\maxid_{\widehat{B}^G}/\maxid^2_{\widehat{B}^G}$, which has dimension
\begin{equation}
\label{formula embedding dimension}
\edim(\widehat{B}^G) = 2n + \binom{n}{2} + p^n - \binom{2p+n-2}{n} + n\binom{p+n-2}{n}.
\end{equation}
Furthermore, the ring $\widehat{B}^G$ has depth $n+2$.
\end{proposition}

\proof
First note that the grading on the polynomial ring $B$ induces a grading
on the ring of invariants $B^G\subset B$. According to \cite{Kemper 2002},
Proposition 2.1, together with Lemma 2.2, these gradings ensure that the minimal number
of generators for the $k$-algebra $B^G$ coincides with $\edim(\widehat{B}^G)$.

Richman   conjectured that the indeterminates together with the norms, minors and traces  
generate the ring of invariants  
(\cite{Richman 1990}, page 32). This was established in full generality by Campbell and Hughes \cite{C-H}.
Later, Shank and Wehlau showed that the elements form a minimal set of generators,
if one discards the non-relevant traces (\cite{Shank; Wehlau 2002}, Corollary 4.4).
In turn, these yield a basis for the cotangent space $\maxid_{\widehat{B}^G}/\maxid^2_{\widehat{B}^G}$.

The indeterminates $a_i$, together with the norm elements $x_i$ and minor elements $z_{ij}$ contribute
$2n+\binom{n}{2}$ members of the basis. It remains to count the number
of monomials $u_1^{\epsilon_1}\ldots u^{\epsilon_n}_n$ corresponding to relevant tuples $\epsilon\in\NN^n$.
There  are $p^n$ monomials that have degree $\leq p-1$ in each variable, and 
there are 
$\sum_{d=0}^{2p-2}\binom{d+n-1}{n-1}=\binom{2p+n-2}{n}$
monomials with total degree $\leq 2p-2$. Among the latter one sees that there are
$\sum_{d=0}^{p-2}\binom{d+n-1}{n-1}=\binom{p+n-2}{n}$
excess monomials that have degree $\geq p$ in a fixed variable $u_i$. 
Our  formula for $\edim(\widehat{B}^G)$ follows.

The statement about the depth of the local ring of $B^G$ at the origin
is proved  by Ellingsrud and Skjelbred (\cite{Ellingsrud; Skjelbred 1980}, Theorem 3.1),
under the assumption that the ground field $k$ is algebraically closed.
In the notation of their result, we interpret $F_{-1}$ to be the empty set, and   note that $F_0$ is the closed set
$V(a_1,\dots, a_n)$ in $\Spec B$, which as dimension $n$.
Since depths are invariant under ground field extensions and formal completions, the formula holds in general. 
\qed

\medskip
Recall that $x_i$ denote the norm from $B$ to $B^G$ of the element $u_i$.
We now write  $\widehat{B}=k[[u_1,\dots, u_n,a_1,\dots, a_n]]$  and   identify the norm of $u_i$ from $\widehat{B}$ to $\widehat{B}^G$ with $x_i$.
Thus,  $\widehat{B} $ contains  $R:=k[[x_1,\dots, x_n]]$ as subring.
Choose  $\alpha_1,\dots, \alpha_n \in \maxid_R$, and consider the elements 
$$
b_i:=a_i -\alpha_i \in \widehat{B}.
$$ 
Since $(u_1,\dots,u_n,a_1,\dots,a_n)=  (u_1,\dots,u_n,b_1,\dots,b_n)$ we conclude that
the $b_1,\dots,b_n$ are part of a regular system of parameters of $\widehat{B} $
(use \cite{Matsumura 1989}, 17.4). 
Let $\mathfrak b$ denote the ideal of $\widehat{B}$ generated by $b_1,\dots, b_n$. It follows from the equality of ideals just mentioned that 
the ring $A:=\widehat{B}/\mathfrak b$ is regular of dimension $n$, with maximal ideal generated by the classes of $u_1,\dots, u_n$. Clearly, $b_i \in \widehat{B}^G$.
The ring $A=\widehat{B}/\mathfrak b$ has thus an induced action of $G$.

\begin{lemma} \label{lem. modram}
The ring $A$ is isomorphic to the ring $k[[x_1,\dots,x_n]][u_1,\dots,u_n]/I$, where the ideal $I$ is generated
by $u_i^p-\alpha_i^{p-1} u_i -x_i$ for $1\leq i\leq n$.
When $\alpha_1,\dots,\alpha_n$ is a system of parameters in $R$, 
then the $G$-action on $A$ is moderately ramified.
\end{lemma}
\proof
Follows from the Implicit Function Theorem, the definition of a moderately ramified action and Theorem \ref{structure moderately ramified2}.
\qed

\smallskip
The elements $b_1,\ldots,b_n$ are $G$-invariant, so they define ideals in both  $\widehat{B}^G$ and $\widehat{B}$.
We have a natural homomorphism
$$
\varphi: \widehat{B}^G/(b_1,\dots, b_n)\widehat{B}^G \lra (\widehat{B}/ (b_1,\dots, b_n)\widehat{B})^G=A^G.
$$
Under suitable assumptions, this map is bijective, as we now show.

\begin{proposition}
\label{base-change map bijective}
Keep the above notation and assumptions.
Assume that $b_1,\ldots,b_n$ is a regular sequence in $\widehat{B}^G$, 
and that at least one $\alpha_i\in\maxid_R$ is non-zero. Then the following holds.
\begin{enumerate}[\rm (i)]
\item The  map $\varphi: \widehat{B}^G/(b_1,\ldots,b_n)\widehat{B}^G \ra A^G$ is an isomorphism.
\item We have  $
\edim(A^G)= n +\binom{n}{2} + p^n   - \binom{n+2p-2}{n}  + n\binom{n+p-2}{n}$ .
\end{enumerate}
\end{proposition}

\proof
(i) 
The  ring  $\widehat{B}^G$ 
has dimension $2n$ and, according to \ref{generators vector invariants}, it has  depth $n+2$. Hence, the quotient ring $\widehat{B}^G/(b_1,\ldots,b_n)$
acquires dimension $n$ and  depth $2$.
Since the ring extension $B^G\subset B$ is finite, the induced map on affine schemes
$\Spec(B)\ra\Spec(B^G)$ is surjective.  This continues to hold
for 
$$
\Spec(\widehat{B}/(b_1,\ldots,b_n))\lra\Spec(\widehat{B}^G/(b_1,\ldots,b_n)).
$$
The scheme on the left is irreducible, whence the same holds for the scheme on the right.
Now consider the exact sequence
$$
0\lra N\lra \widehat{B}^G/(b_1,\ldots,b_n) \xrightarrow{\ \varphi \ } A^G \lra F\lra 0
$$
defining the kernel $N$ and cokernel $F$ for our map in question.
Clearly, the image  of $\varphi$ contains the norm and minor elements, thus the cokernel
$F$ has finite length, by Proposition \ref{minor subring}.
Moreover, its kernel $N$ contains only nilpotent elements, because 
$\widehat{B}^G/(b_1,\ldots,b_n) $ is $n$-dimensional with only one 
minimal prime ideal. We conclude that $N$ equals this minimal prime ideal.
If we could show that $N=0$, then the morphism of schemes
$$
\Spec(A^G)\lra \Spec(\widehat{B}/(b_1,\ldots,b_n)) 
$$
is an isomorphism outside the closed points. Since $\widehat{B}/(b_1,\ldots,b_n)$ 
has depth $2$, this ring must be normal, and the map 
$\varphi$
is bijective by Zariski's Main Theorem.

To see that $N=0$, we regard the ring of invariants $B^G$ as a finite algebra over the polynomial ring
$k[x_1,\ldots,x_n,a_1,\ldots,a_n]$ generated by the norms $x_i$ and the variables $a_i$.
The scheme $\Spec(B^G)$ is regular outside the image of the fixed scheme in $\Spec(B)$,
the latter being  defined by the ideal generated by $a_1,\ldots,a_n$. It follows
that 
$$
\Spec(B^G)\ra\AA^{2n}=\Spec(k[x_1,\ldots,x_n,a_1,\ldots,a_n])
$$ 
is finite and flat, of degree $p^{n-1}$,
at least over the complement of the subscheme defined by $a_i$, $1\leq i\leq n$.
The subscheme defined by $b_i=a_i-\alpha_i$, $1\leq i\leq n$ is not contained in this, 
because some $\alpha_i$ is non-zero, and
it follows that $\widehat{B}^G/(b_1,\ldots,b_n)$ has degree $p^{n-1}$
as module over 
$$
R=k[[x_1,\ldots,x_n]]=k[[x_1,\ldots,x_n,a_1,\ldots,a_n]]/(b_1,\ldots,b_n).
$$
The same holds for $A^G$. Tensoring with the fraction field of $R$,
we deduce that $N\otimes_R\Frac(R)=0$ by counting vector space dimensions.
Since $\widehat{B}^G/(b_1,\ldots,b_n) $ has no embedded components, this
already ensures $N=0$.

(ii) In light of (i), we merely have to check that 
the $b_i=a_i-\alpha_i$, $1\leq i\leq n$ are linearly independent
in the cotangent space of $B^G$. This indeed holds, because  the $x_i=u_i^p - a_iu_i$ vanish 
and the $a_i$ are linearly independent in the cotangent space of the ring $B$.
\qed

\medskip
In dimension $n=3$, we obtain the following unconditional result
when the formal power series $\alpha_1, \alpha_2, \alpha_3,$ are polynomials.

\begin{theorem} \label{thm.n=3}
Let $n=3$.  Choose polynomials  $\alpha_i=\alpha_i(x_1,x_2,x_3) \in k[x_1,x_2,x_3]$, $i=1,2,3$,  such that $\alpha_1, \alpha_2, \alpha_3$ form a system of parameters in $R=k[[x_1,x_2,x_3]]$.
Then the induced $G$-action on $A:=\widehat{B}/\idealb$ is moderately ramified
and the natural homomorphism $\widehat{B}^G/(b_1,b_2,b_3)\ra A^G$ is bijective.
In particular, $A^G$ is generated
as a complete local $k$-algebra by the norm elements $x_1,x_2,x_3$, the minor elements $ z_{12},z_{13},z_{23}$
and the $(p^3-p)/6$   trace elements $t_\epsilon\in A^G$. 
Moreover, we have  $\edim(A^G)= 6 + (p^3-p)/6$.
\end{theorem}

\proof
In view of Proposition \ref{base-change map bijective}, we start by proving that   $b_1,b_2,b_3\in B^G$ form a regular sequence.
Using the case $r=1$ of Theorem 1.4 in \cite{Kemper 1999},
we have to check that the map 
$$
B\lra B^{\oplus 3},\quad x\longmapsto (b_1x,b_2x,b_3x)
$$
induces an injective map $H^1(G,B)\ra H^1(G,B^{\oplus 3})$ on group cohomology.
Using the decomposition $H^1(G,B^{\oplus 3})= H^1(G,B)^{\oplus 3}$ and symmetry, is suffices 
to verify that the map $b_3:H^1(G,B)\ra H^1(G,B)$ is injective.
 
Let $B^\flat\subset B$ be the  $k[a_1,a_2,a_3]$-submodule generated by 
all monomials $ u_1^{n_1}u_2^{n_2}u_3^{n_3}$ with exponents $n_1,n_2,n_3\leq p-1$.
Clearly, $B^\flat\subset B$ is $G$-invariant.
According to \cite{Shank; Wehlau 2002}, Proposition 6.2, the canonical map
$$
k[x_1,x_2,x_3]\otimes_k H^1(G,B^\flat)\lra H^1(G,B)
$$
is bijective (in loc.\ cit.\ the symbols $k[W]^\flat$ are defined at the end of Section 2 on page 310, 
and the symbol $B$ is defined on page 318 before Lemma 6.1).
Multiplication by the element $b_3=a_3-\alpha_3(x_1,x_2,x_3)$ on the right
corresponds to the $k$-linear mapping $f=\id\otimes a_3 - \alpha_3(x_1,x_2,x_3)\otimes \id$
on the left.
Let $\Fil_i\subset k[x_1,x_2,x_3]$ by the  vector subspace of all polynomials of
total degree $\leq i$. This is an ascending filtration that is exhaustive and discrete,
in the sense that  $\bigcup_i\Fil^i=k[x_1,x_2,x_3]$ and $\Fil_i=0$ for $i=-1$.
By  abuse of notation
we denote the induced filtration   $ \Fil^i\subset  k[x_1,x_2,x_3]\otimes_k H^1(G,B^\flat)$ 
by the same symbol.
Clearly, $f(\Fil^i)\subset \Fil^{i+d}$, where $d\geq 1$ is the total
degree of $\alpha_3(x_1,x_2,x_3)$.
Moreover, the induced map $f:\Gr^i\ra \Gr^{i+d}$ on the associated graded  coincides with the 
multiplication by $\alpha_3(x_1,x_2,x_3)\otimes\id$, and the latter is   injective because $\alpha_3\in k[x_1,x_2,x_3]$
is  a regular element.
It follows that our original multiplication map $b_3:H^1(G,B)\ra H^1(G,B)$ is injective
(compare for example \cite{Sjoedin 1973}, Lemma 1 (e)).
This shows that $b_1,b_2,b_3\in B^G$ form a regular sequence.
In turn, they form a regular sequence in the flat extension $\widehat{B}^G$.

Now it follows from  Proposition \ref{base-change map bijective} that the   map 
$\widehat{B}^G/(b_1,b_2,b_3)\widehat{B}^G \ra A^G$ is bijective, with
indicated embedding dimension. The generators of $\widehat{B}^G$ are given in Proposition \ref{generators vector invariants}. According to Lemma \ref{lem. modram}, the ring $A:=\widehat{B}/\mathfrak{b}$ equals  
$$
k[[x_1,x_2,x_3,u_1,u_2,u_3]]/(u_1^p-\alpha_1^{p-1}u_1-x_1,u_2^p-\alpha_2^{p-1}u_2-x_2,u_3^p-\alpha_3^{p-1}u_3-x_3),
$$
and the $G$-action given by $u_i\mapsto u_i+\alpha_i$ is moderately ramified.
\qed

\begin{example}\label{generalizedreflectionCM} 
Let  $A$ denote a  complete  noetherian regular local ring of dimension $n$ and  characteristic $p>0$,
endowed with a moderately ramified action of a cyclic group $G$ of order $p$.
{\it Then the ring $A^G$ is endowed with a group of automorphisms of order $p^{n-1}$ generated by  generalized reflections of order $p$.}
Indeed, $A$ is endowed with an elementary abelian group of automorphisms $H$, and by construction, $H/G$ acts on $A^G$. 
In the presentation of $A$ given in \ref{structure moderately ramified}, we find that the ideal $I(\sigma_i)$ of the fixed scheme of the pseudo-reflection $\sigma_i:A \to A$ is $a_iA$. 
Thus the image $\overline{\sigma}_i: A^G \to A^G$ of $\sigma_i$ in $H/G$ has an ideal $I(\overline{\sigma}_i)$ contained in $a_iA^G$ (since $a_i \in A^G$).
In particular, $\overline{\sigma}_i$ is a  generalized reflection of order $p$. We note in the example below that the ideal $I(\overline{\sigma}_i)$
is not principal in general.

Consider the case where  $n=3$. In this case, $A^G$ is not Cohen--Macaulay (use \ref{properties invariant ring} (i)), and Theorem \ref{thm.n=3} gives us a set of generators that we can use to 
obtain information about the ideal $I(\overline{\sigma}_1)$. We find that $\overline{\sigma}_1(z_{12})-z_{12}=-a_1a_2$ and 
$\overline{\sigma}_1(z_{13})-z_{13}=-a_1a_3$, so that $(a_1a_2, a_1a_3) \subseteq I(\overline{\sigma}_1)$. 
From the relation $z^p_{23}-(a_2a_3)^{p-1}z_{23}-a_2^px_3+a_3^px_2$, we find that $a_1z^p_{23} \in (a_1a_2, a_1a_3)$. 
Note now that $I(\sigma_2\sigma_3)=(a_2,a_3)A$. 
Since $\overline{\sigma}_1= (\overline{\sigma}_2\overline{\sigma}_3)^{-1}$, we find that $I(\overline{\sigma}_1) \subseteq (a_1)A \cap (a_2,a_3)A \cap A^G$ 
(with $(a_2,a_3)A \cap A^G \supseteq (a_2,a_3,z_{23})A^G$).

Assume now that in addition $p=2$. By considering the last generator of $A^G $
found in \ref{thm.n=3}, namely $t:=u_1u_2u_3+\sigma(u_1u_2u_3)$, we find that $\overline{\sigma}_1(t)-t=a_1z_{23}+a_1a_2a_3$.
Thus in this case $I(\overline{\sigma}_1)=(a_1)(a_2,a_3,z_{23})$, which shows that $I(\overline{\sigma}_1)$ is not principal.  

We note that Theorem \ref{maximal subgroups} shows that 
the invariant ring $(A^G)^{\left< \overline{\sigma}_1\right>}$  is a complete intersection and, thus, is Cohen--Macaulay. 
Consider a Cohen--Macaulay local ring $R$ and a Galois extension $L/\Frac(R) $ of prime order. 
Let $B$ denote the integral closure of $R$ in $L$. When the Galois group of $L/\Frac(R) $ is generated by a generalized reflection $\sigma: B \to B$, 
one may wonder whether $B$ is 
always Cohen--Macaulay. This question is given a partial positive answer when the order of $\sigma$ is coprime to the residue characteristic of $R$ in \cite{H-E}, Proposition 15. The above example with $(A^G)^{\left< \overline{\sigma}_1\right>} \subset A^G$ shows that the answer to this question is negative in general.
We note however that  $(A^G)^{\left< \overline{\sigma}_1\right>} $ is Gorenstein, and that $ A^G$ is shown in 
Theorem \ref{canonical class trivial} to be quasi-Gorenstein.
\end{example}

\section{Class group and canonical class}
\label{class group}

Let $A$ be a complete local   noetherian ring that is regular of dimension $n\geq 2$
and  characteristic $p>0$, with a field of representatives $k$, and endowed with
an action of a cyclic group $G$
of order $p$ that is ramified precisely at the origin.  
Let $A^G$ be the ring of invariants, and denote by
$\Cl(A^G)$  its  \emph{class group} of Weil divisors modulo linear equivalence.
Since $A^G$ is normal, the group $\Cl(A^G)$ is trivial if and only if $A^G$ is factorial. 

The canonical module $K_{A^G}$ of $A^G$ is reflexive of rank one. Since $\Cl(A^G)$ is naturally isomorphic to the rank one reflexive class group, 
we obtain using this isomorphism a canonical class $[K_{A^G}]\in\Cl(A^G)$.
The main result of this section, Theorem \ref{canonical class trivial},  states that the canonical class $[K_{A^G}]$ is trivial  
when the $G$-action is moderately ramified.
 
Let us start with a description of the group $\Cl(A^G)$, 
which follows from a result of Waterhouse \cite{Waterhouse 1971}, 
generalizing work of Samuel \cite{Samuel 1964}, Th\'eor\`eme (1). 

\if false
This group can 
also be viewed as the group of isomorphism classes of finitely generated reflexive rank one
modules over $A^G$,
and the Picard group of the complement  $W$ of the closed point in $Y=\Spec(A^G)$.
Write $U=U(A)=A^\times$ for the multiplicative group of 
invertible power series, and $U_1\subset U$ for the subgroup of     power series
with constant term $1$. It turns out that the class group of $A^G$ can be expressed  
entirely in terms of the   $G$-module $U_1$:
\fi

\begin{proposition}
\label{description class group}
Let $A $ be as above, endowed with a $G$-action ramified precisely at the origin. 
Let $U_1$ denote the kernel of the map $A^* \to k^*$.
Then $\Cl(A^G)$ is isomorphic to $H^1(G,A^*)$, and there is a short exact sequence
$$
0\lra U_1^G/(U_1^G)^p\lra (U_1/U_1^p)^G \lra \Cl(A^G)\lra 0.
$$
\end{proposition}

\proof
That $\Cl(A^G)$ is isomorphic to $H^1(G,A^*)$ follows from \cite{Waterhouse 1971}, Corollary 1, page 545, applied to the Galois extension 
${\rm Frac}(A)/{\rm Frac}(A^G)$ of degree $p$ with Galois group $G$. Indeed, it is clear that $A$ is the integral closure of $A^G$ in ${\rm Frac}(A)$, 
 and that every minimal prime of $A^G$ (and by this, \cite{Waterhouse 1971} means the primes in $A^G$ of height $1$, see proof of Theorem 1)
are unramified in $A$, since we assume that the $G$-action is ramified precisely at the origin. The definition of ${\mathcal P}(A^G,A)$ as the kernel of 
the natural map $\Cl(A^G) \to \Cl(A)$ is given in the abstract of the paper, 
and we find that since $A$ is regular and, hence, factorial, ${\mathcal P}(A^G,A) =\Cl(A^G)$.

Recall that  the $G$-action is $k$-linear. Hence $1\ra U_1 \ra A^*\ra k^* \ra 1$
is an exact sequence 
of $G$-modules, and induces an exact sequence
$$
1 \lra H^1(G,U_1) \lra H^1(G,A^*)\lra H^1(G,k^*).
$$
Since the action on $k^*$ is trivial, $H^1(G,k^*)={\rm Hom}(G,k^*)=(0)$.
It follows that $H^1(G,U_1) \ra H^1(G,A^*)$ is bijective. 
The  short exact sequence
$(1)\ra U_1\stackrel{p}{\ra} U_1\ra U_1/U_1^p\ra (1)$ of $G$-modules gives an exact sequence
$$
H^0(G,U_1) \stackrel{p}{\lra} H^0(G,U_1) \lra H^0(G,(U_1/U^p_1))\lra H^1(G,U_1)\stackrel{p}{\lra} H^1(G,U_1).
$$
The map on the right is the trivial map, because the cohomology group $H^1(G,U_1)$ is annihilated by $\ord(G)=p$. The assertion follows.
\qed
 
\medskip
The group  $\Cl(A^G)$ is expected to be frequently  non-trivial.
When $n=2$, this expectation  is implied by Theorem 25.1 in \cite{Lip}. A direct proof that  $\Cl(A^G)$
is not trivial for certain moderately ramified actions is presented in our next lemma. When $n>2$, examples of non-Cohen--Macaulay unique factorization domains are known (see \cite{LipUFD}, section 4, for a survey on this matter, and \cite{Mori}). In the context of ring of invariants for an action of $G:={\mathbb Z}/p{\mathbb Z}$ on $A:=k[[x_1,\dots,x_n]]$, Theorem 1.3 in \cite{B-W} exhibits examples of complete local rings $A^G$ with $n\leq p$ which are not unique factorization domains. 
 
Let $n=3$.  Choose polynomials  $\alpha_1, \alpha_2 \in k[x_1,x_2]$, and $\alpha_3 \in k[x_1,x_2,x_3]$
such that $\alpha_1, \alpha_2, \alpha_3$ form a system of parameters in $R=k[[x_1,x_2,x_3]]$.
Consider the associated action of $G$ on $A:=k[[u_1,u_2,u_3]]$ as in Theorem \ref{thm.n=3}.
We obtain from Theorem \ref{thm.n=3} a description of a minimal set of generators $x_1,x_2,x_3,w_1,\dots, w_s$ for $A^G$, which includes
the minor element $z_{12}:=\alpha_1 u_2-\alpha_2 u_1$. Write $A^G=k[[x_1,x_2,x_3,w_1,\dots, w_s]]/I$, where $I$ is the ideal of relations
between the generators of $A^G$. Since we are considering a minimal set of generators, we find that $I \subset J^2$, where 
$J=(x_1,x_2,x_3,w_1,\dots, w_s)$.
Recall \eqref{minor relations} that $z_{12}$ satisfies the relation
$$ z_{12}^p-(\alpha_1 \alpha_2)^{p-1}z_{12} -\alpha_1^px_2+\alpha_2^px_1=0.$$

\begin{lemma} 
\label{not-factorial} Keep the above notation. 
Suppose that 
$\alpha_1^px_2-\alpha_2^px_1$ can be factored in $k[[x_1,x_2]]$ into a product of at least $p+1$ non-units. 
Then the ring  $A^G$ is not factorial.
\end{lemma}
\proof 
In the ring $A^G$, we have a factorization 
\begin{equation} \label{factorial}
\alpha_1^px_2-\alpha_2^px_1= \prod_{i \in {\mathbb F}_p} (z_{12}-i\alpha_1\alpha_2).
\end{equation}
We claim that each factor $z_{12}-i\alpha_1\alpha_2$ is irreducible in $A^G$. Indeed, if $z_{12}-i\alpha_1\alpha_2$ is reducible,
there exists power series $f,g,h\in k[[x_1,x_2,x_3,w_1,\dots, w_s]]$ such that 
$$
z_{12}-i\alpha_1\alpha_2 = f g + h,
$$
with $f,g \in J$ and $h \in I \subseteq J^2$, since the images of $x_1,x_2,x_3,w_1,\dots, w_s$ form a minimal system of generators. It follows from this expression for $z_{12}$ that $z_{12} \in  J^2$, which is a contradiction.  

Now  assume that $A^G$ is factorial. Then the elements $z_{12}-i\alpha_1\alpha_2 \in A^G$, $i\in\FF_p$, are prime elements,
and the existence and uniqueness of prime factorization shows that 
$\alpha_1^px_2-\alpha_2^px_1$
does not factor into
more than $p$ non-units.
\qed

\medskip
It is easy to produce examples of elements $a,b\in k[[x,y]]$ such that $a^py-b^px$ factors into a
product of $p+1$ elements.
For instance, set $a=x$ and $b=y$, to obtain $a^py-b^px = yx^{p}-xy^{p}= xy\prod_{i \in {\mathbb F}_p^*} (x-iy) $.
When $k$ contains  a primitive $p+1$-root  of unity $\zeta$, we can take $a=y$ and $b=x$, 
to obtain $a^py-b^px = y^{p+1}-x^{p+1}= \prod_i( y-\zeta^i x) $. 

\bigskip
We now briefly review some  facts regarding   canonical modules.
Let $S$ be any   complete local noetherian ring of dimension $n$ with maximal ideal $\maxid$ and residue field $k:=S/\maxid$. 
Let $E_{S}(k)$ be the injective hull of the residue field $k$. If $M$ is any $S$-module, let $H^i_\maxid(M)$ denote the $i$-th local cohomology group.
Then the functor 
$$
M\longmapsto \Hom_{S}(H^n_\maxid(M),E_{S}(k))
$$
is representable by a module $K_S$. This module is called the \emph{canonical module} of $S$ \cite{Herzog; Kunz 1971},
or the \emph{dualizing module} \cite{BouChap10}. By Yoneda's Lemma, it is unique up to isomorphism.
The $S$-module $K_S$ is finitely generated, of dimension $\dim(K_S)=n$, and satisfies
Serre's Condition $(S_2)$.
These statements are proved for instance in \cite{Herzog; Kunz 1971}, 5.2, and 5.16. 
Note that these hold without the requirement that the ring be Cohen--Macaulay.

When $S$ is Cohen--Macaulay, the canonical module   is isomorphic to $S$
if and only if $S$ is Gorenstein (\cite{Herzog; Kunz 1971}, 5.9). 
In particular, if $S=k[[x_1,\dots,x_n]]$, then $K_S$ is isomorphic to $S$.

\begin{emp} \label{useful.fact}
Suppose that $T$ is a complete local ring and $\phi:S \to T$ is an injective  
local homomorphism  of local rings, such that $T$ is finite over $S$.
Then $K_T:={\rm Hom}_S(T, K_S)$ is the  canonical module of $T$ (\cite{Herzog; Kunz 1971}, 5.14). 
\end{emp}

Assume now that $P \in \Spec S$ is such that $\dim S= \dim S/P + {\rm height}(P)$. Then  $K_{S_P}$ is isomorphic to $(K_S)_P$ (\cite{Herzog; Kunz 1971}, 5.22).
The prime ideal $P$ belongs to the support of $K_S$ if and only if $\dim S= \dim S/P + {\rm height}(P)$ (\cite{Aoyama 1983}, 1.9). 
In particular, if $S$ is normal and $P$ has height one, we find that $S_P$ is regular and, thus, $K_{S_P}$ is trivial. Hence, 
$(K_S)_P$ is free of rank $1$ for every prime $P$ of height one.
The module $K_S$ satisfies Serre's condition $S_2$ (\cite{Aoyama 1983}, 1.10). 

Let $M$ be a finitely generated $S$-module, and recall the functor $M \mapsto M^*:={\rm Hom}_{S}(M,S)$. 
The module $M$ is called {\it reflexive} if the natural map $M \to (M^*)^*$ is an isomorphism. 
\if false
When $S$ is normal, we can apply Theorem (1.4) in \cite{Vas}
and find that $M$ is reflexive if and only if every $R$-sequence of two or less elements be also an $M$-sequence. If $R$ is a domain, a torsion-free $S$-module $M$ is reflexive if and only if it is the intersection of all its localizations at height one primes (\cite{Vas}, bottom of page 1351).
\fi
It follows from the facts in the previous paragraph and \cite{Hartshorne 1994}, Theorem 1.9, that when $S$ is normal, then $K_S$ is reflexive.

Recall that a prime ideal ${\mathfrak p}$ of height $1$ in a normal ring $S$ is reflexive, and that there is a natural isomorphism of groups between the class group 
$\Cl(S)$ and the rank one reflexive class group, which sends the class of $\mathfrak p$ to the class of ${\mathfrak p}$ (\cite{Yua}, Theorem 1).


We return now to the ring $A$ with a $G$-action as at the beginning of this section, and denote by $K_{A^G}$ the  canonical module for the complete local ring $A^G$. Since it follows from above that $K_{A^G}$ is reflexive, we can consider the element $$[K_{A^G}]\in\Cl(A^G),$$ called the \emph{canonical class} of $A^G$.

\begin{remark} \label{rem.notGorenstein}
In general, the canonical class $[K_{A^G}]$ need not be trivial. 
Indeed, it is possible to exhibit in dimension $n=2$  diagonal actions on products of
ordinary algebraic curves over an algebraically closed field $k$ such that the associated quotient singularities are rational singularities (\cite{Lorenzini 2011c}, 4.1). In general, these singularities are   
not Gorenstein since the only Gorenstein rational singularities are the double points (\cite{Wah}, 2.5). 
\end{remark}

Our next theorem shows that when $A^G$ is obtained from a moderately ramified action, then its canonical class is trivial.

\begin{theorem}
\label{canonical class trivial} 
Let $A$ be a complete local regular noetherian ring of dimension $n\geq 2$,
characteristic $p>0$,  and endowed with
a moderately ramified action of a cyclic group $G$
of order $p$.
Then  the canonical class
$[K_{A^G}]$ is trivial in $ \Cl(A^G)$.
\end{theorem}

The proof of Theorem \ref{canonical class trivial} uses some methods from non-commutative algebra
and is presented after the proof of Proposition \ref{symmetric}.
Recall that the free $A$-module on the elements of $G$ can be endowed  with 
an associative multiplication, turning it into a ring called the  \emph{skew group ring} $A*G$.
On elements of the form $a\sigma$ and $  b\eta$  
with $a,b\in A$ and $\sigma,\eta\in G$, the multiplication is
$$
a \sigma \cdot b \eta := (a\sigma(b) ) (\sigma\eta),
$$
and this multiplication is extended to all elements by bilinearity. 
Since $G$ is abelian, the center of this associative algebra is $A^G*G$. Using the inclusion $A^G e \subset A^G*G$, 
where $e\in G$ is the neutral element, we may
regard the associative ring $A*G$ as an  algebra over the central subring $A^G$, or over any subring $R$ of $A^G$. 

The ring $A$ is endowed with the structure of a left $A*G$-module as follows: for $c, a \in A$ and $\sigma \in G$, let $(c\sigma)a := c \sigma(a)$.
The following injective homomorphism of $A^G$-algebras
 \begin{equation}
A^G\lra\End_{A*G}(A),\quad a\longmapsto (x\mapsto ax), \label{iso2}
\end{equation}
is easily seen to be an isomorphism. Indeed, given $\varphi: A \to A$ in $\End_{A*G}(A)$, we have $\varphi(\sigma \cdot 1)=\sigma(\varphi(1))=\varphi(\sigma(1))$, so that $\varphi(1)\in A^G$. We also have $\varphi(a)=a \varphi(1)$ for all $a \in A$. 
We show below in Proposition \ref{comparison maps bijective}  that in the case of moderately ramified actions
 the natural  homomorphism  of $A^G$-algebras
\begin{equation}
A*G\lra\End_{A^G}(A),\quad a\sigma\longmapsto (x\mapsto a\sigma(x)), \label{iso1}
\end{equation}
is also an isomorphism.

Let $R\subset A^G$ be a normal noetherian  subring such that $A^G$ is a finite $R$-algebra.
Let $\primid\in \Spec R$. Let $M$ be a finitely generated left $A*G$-module. 
Then $M_\primid$ and $(A*G)_\primid$ are well-defined $R_\primid$-modules.
It turns out that $(A*G)_\primid$ is in fact an $R_\primid$-algebra 
when endowed with the natural multiplication $(\lambda/s) \cdot (\lambda'/s'):= (\lambda \lambda')/(ss')$, 
where $\lambda, \lambda' \in A*G$ and $s,s' \in R \setminus \primid$. In addition, $M_\primid$ is a left $(A*G)_\primid$-module.
Recall that the $A*G$-module $M$ is projective if the functor $N\mapsto {\rm Hom}_{A*G}(M,N)$ is exact, and that 
homological algebra in module categories is the same over any ring, commutative or not (see for instance  \cite{Rot}).
 
\begin{proposition} 
\label{comparison maps bijective} 
Keep the above notation. Assume that the $G$-action on $A$ is  unramified in codimension one.
Let $R\subset A^G$ be a normal noetherian  subring such that $A^G$ is a finite $R$-algebra.
\begin{enumerate}[ \rm (i)]
\item  The  $R$-modules $A^G$, $A$, $A*G$, and  $\End_{A^G}(A)$ 
are reflexive.
\item The  homomorphism  $A*G\ra\End_{A^G}(A)$ in \eqref{iso1}
is an isomorphism. 
\item For each  $\primid \subset R$ of height
one, the left $(A*G)_\primid$-module $A_\primid$ is projective.
\end{enumerate}
\end{proposition}

\proof 
(i) 
According to \cite{Hartshorne 1994}, Theorem 1.9, a finitely generated $R$-module $M\neq 0$ is
reflexive if and only if it satisfies Serre's Condition $(S_2)$. Since $R$ is local, this means  $\depth(M)\geq 2$.
It follows that   the reflexivity of $M$ does not depend on the chosen subring $R$, as remarked in
\cite{Iyama; Reiten 2008}, page 1094.
Applying this principle to $R\subset A$, we see that the $R$-module $A$ is reflexive,
and it follows that  $A*G$ is reflexive as well.
To proceed, let  $g_1,\ldots,g_r\in A^G$ be a system of module generators over $R$.
Then we have a short exact sequence
$$
0\lra \End_{A^G}(A) \lra \End_R(A) \lra \bigoplus_{i=1}^r\End_R(A),
$$
where the map on the right sends an $R$-linear map  $f:A\ra A$ to the tuple of commutators 
$fg_i-g_if$.
The two terms on the right are reflexive $R$-modules, that is,  satisfy Serre's Condition $(S_2)$.
Using local cohomology, we infer that the term on the left
satisfies $(S_2)$, whence is reflexive. Analogous arguments apply to $\End_{A*G}(A)$.

(ii) For this, we regard $A*G$, $\End_{A^G}(A)$ and $A^G$ as modules over $A^G$.
These  modules are reflexive by (i). Hence Theorem 1.12 in \cite{Hartshorne 1994} shows that it suffices to check that the maps are bijective after 
localizing at each prime ideal $\primid\subset A^G$ of height one. 
So we may replace $A^G$ by the localization $(A^G)_\primid$.
By faithfully flat descent
we even may replace it by the strict henselization $(A^G)_\primid\subset C$.
Since the $G$-action is free in codimension one, we get 
an isomorphism $A\otimes_{A^G}C\simeq C\times G=C^p$ of $C$-algebras, with the permutation action of $G=\ZZ/p\ZZ$.

The first map in question takes the form
$(C\times G)*G \ra \End_C(C\times G)$. Since domain and range are free $C$-modules of rank $p^2$,
it suffices to check that the map is surjective, by Nakayama's Lemma.
Let $\sigma,\eta\in G$.
Since matrix rings are generated by elementary matrices, it suffices to check that
the endomorphism  that is zero on all standard basis vectors except that  $(1,\sigma)\mapsto (1,\eta)$ 
lies in the image. Indeed, it is the image of $(1,\eta)\sigma^{-1}\in (C\times G)*G$.

\if false
In light of the preceding paragraph, the second map in question, after tensoring with the strict henselization
$(A^G)_\primid\subset C$,
can be viewed as  $C\ra\End_{\Mat_p(C)}(C^p)$.
This map is bijective, because the matrices that commute with all matrices are 
exactly the scalar matrices.
\fi

(iii) We have to show that $\Ext^1(A_\primid,M)$ vanishes for each module $M$ over the ring $(A*G)_\primid$.
This Ext group can be computed with injective resolutions of $M$, or alternatively with 
free resolutions of $A_\primid$. Being a finitely generated $R_\primid$-module,
the associative ring $(A*G)_\primid$ is noetherian, whence we may choose a free resolution 
with finitely generated terms.
Consequently, the formation of $\Ext^1(A_\primid,M)$ commutes with flat base-change in $R_\primid$.
Thus we can reduce to the strict henselization   $R_\primid\subset D$.
Let $\primid_1,\ldots,\primid_r\subset A^G$ be the prime ideals lying over  $\primid\subset R$.
Then $C=(A^G)\otimes_RD$ decomposes as $C=C_1\times\ldots \times C_r$, where $(A^G)_{\primid_i}\subset C_i$ are 
the strict henselizations. As above, we get $A\otimes_RD=C\times G=C^p$,
and $(C\times G)*G=\End_C(C^p)$. As a left module, $C^p$ is a direct summand
of $\End_C(C^p)$, whence a projective module.
\qed

\medskip
Let $\SR$ be a commutative ring, and let $\Lambda$ be an associative $\SR$-algebra. Then the $\SR$-module
$\Lambda$ carries the structure of a  $\Lambda$-bimodule via $a\cdot x\cdot b :=axb$, for all $a, b,x \in \Lambda$.
Similarly, the  dual $\SR$-module $\Hom_\SR(\Lambda,\SR)$ becomes a $\Lambda$-bimodule, via
$$
(a\cdot \phi \cdot b)(x) := \phi(bxa),
$$
for all $\phi \in \Hom_\SR(\Lambda,\SR)$ and $a, b, x \in \Lambda$.
Any homomorphism of left $\Lambda$-modules
$\Lambda\ra\Hom_\SR(\Lambda,\SR)$ is of the form
$
y\longmapsto y\phi
$
for some $\phi\in\Hom_\SR(\Lambda,\SR)$. We leave the proof of our next lemma to the reader.

\if false 
Each element $\phi\in\Hom_\SR(\Lambda,\SR)$ induces two maps
$\Lambda\ra\Hom(\Lambda,\SR)$, given by
$$
y\longmapsto y\phi\quadand y\longmapsto \phi y.
$$
The first map is a homomorphism of left $\Lambda$-modules, and the second is a homomorphism
of right $\Lambda$-modules. Clearly, each left-module or right-module homomorphism
is of the above respective forms. This shows:
\fi

\begin{lemma} \label{bimodule}
Any homomorphism $\Lambda\ra\Hom_\SR(\Lambda,\SR)$ of $\Lambda$-bimodules
is of the form $y\mapsto y\phi$ for some element $\phi\in \Hom_\SR(\Lambda,\SR)$ with
$\phi(xy)=\phi(yx)$ for all $x,y\in\Lambda$.
\end{lemma}

As in \cite{Iyama; Reiten 2008}, page 1095, the $\SR$-algebra $\Lambda$  is called \emph{symmetric} if the $\Lambda$-bimodules $\Hom_\SR(\Lambda,\SR)$ and
$\Lambda$ are isomorphic. In other words, the $\SR$-algebra $\Lambda$  is  symmetric if there exists some $\phi\in \Hom_\SR(\Lambda,\SR)$ such that
$\phi(xy)=\phi(yx)$ for all $x,y\in\Lambda$, and such that $\Lambda\ra \Hom_\SR(\Lambda,\SR)$, $y\mapsto y\phi$,
is bijective.

\begin{proposition}
\label{symmetric} Let $k$ be a field of characteristic $p>0$, and let $R:=k[[x_1,\ldots,x_n]]$. Let $a_1,\ldots,a_n\in\maxid_R$
be arbitrary elements, not all zero. Let $A$ be the      regular (see {\rm \ref{structure moderately ramified2}})  complete local noetherian ring  
given by 
$$
A:=R[u_1,\ldots,u_n]/(u_1^p-a_1^{p-1}u_1-x_1,\ldots,u_n^p-a_n^{p-1}u_n-x_n).
$$
Let $G$ denote the subgroup of automorphisms of $A$ fixing $R$ generated by 
the natural automorphism of order $p$ which sends $u_i$ to $u_i+a_i$ for $i=1,\dots, n$. Then the $R$-algebra $A*G$ is symmetric.
\end{proposition}

\proof We abused notation in the above statement and did not distinguish between $u_i \in R[u_1,\ldots,u_n]$ and its class in $A$. 
We  will continue to do so in the proof below to lighten our notation.
First observe that the skew group ring $A*G$ is a free $R$-algebra of rank $p^{n+1}$  because the group $G$ has order $p$ and the $R$-algebra
$A$ is free of rank $p^n$. In fact, the elements
$$
u_1^{i_1}   \cdots   u_n^{i_n}\sigma^i\in A*G,\quad 0\leq i_1,\ldots,i_n,i\leq p-1,
$$
form an $R$-basis. Following Braun \cite{Braun 2011}, who studied the linear case, we consider  the element $\phi\in\Hom_R(A*G,R)$ given on the basis elements by
$$
u_1^{i_1}\cdots   u_n^{i_n}\sigma^i\longmapsto
\begin{cases}
1 	& \text{if } i_1=\ldots=i_n=p-1 \text{  and } i=0;\\
0	& \text{else.}
\end{cases}
$$

Consider the map 
\begin{equation} \label{Hombimod}
A*G\lra\Hom_R(A*G,R),\quad y\longmapsto y\phi.
\end{equation}
We claim this map is an isomorphism of $A*G$-bimodules. Let us start by showing that
it is a homomorphism of  bimodules. For this, it suffices to verify that
 $\phi(xy)=\phi(yx)$ for all $x,y\in A*G$ (\ref{bimodule}). 
 
Obviously it suffices to show that for any $x \in A*G$ and any basis element $y=u_1^{j_1}\cdots u_n^{j_n}\sigma^j$, we have $\phi(xy)=\phi(yx)$.
We proceed  using   induction on $j_1+\dots+j_n+j$. Assume that the conclusion holds 
for $j_1+\cdots+j_n+j < t$ and let us prove it for $j_1+\dots+j_n+j = t$. When $t>1$, it is always possible to write 
$y=y_1 y_2$, where $y_1$ and $y_2$ are two basis vectors to which we can apply the induction hypothesis.
Then $\phi(xy)=\phi(x(y_1y_2))= \phi(y_2(xy_1))= \phi(y_1(y_2x))= \phi(yx)$. It remains to prove the case   
where $j_1+\dots+j_n+j=1$. 
This means $y=\sigma$ or $y=u_s$ for some $1\leq s\leq n$.
To show that $\phi(xy)=\phi(yx)$, 
it is easy to see that it suffices to treat  the case where  
$x:=u_1^{i_1}\cdots u_n^{i_n}\sigma^i$ is itself a basis vector.

Consider first the case $y=\sigma$. We have $\sigma(u_r)=u_r+a_r$, 
which implies that
$$
xy=u_1^{i_1}\cdots u_n^{i_n}\sigma^{i+1}\quadand 
yx=(u_1+a_1)^{i_1}\cdots (u_n+a_n)^{i_n}\sigma^{i+1}.
$$
Expanding the factors of $yx$ shows that $yx=xy + h$ with $h$ a linear combination of
basis vectors $u_1^{i'_1}\cdots u_n^{i'_n}\sigma^{i+1}$ with $0\leq i_r'<i_r$ for $r=1,\dots, n$.
In particular $0\leq i_1',\ldots,i_n'<p-1$.
Consequently $\phi(h)=0$, and we find that $\phi(yx)=\phi(xy)$.

Now let $y=u_s$ for some $1\leq s\leq n$. Using $\sigma^i(u_s)=u_s+ia_s$, we get
$$
xy=u_1^{i_1}\cdots u_n^{i_n}(u_s+ia_s)\sigma^i\quadand
yx=u_1^{i_1}\cdots u_s^{i_s+1}\cdots u_n^{i_n}\sigma^i.
$$
By definition, both  $\phi(xy)$ and $\phi(yx)$ vanish if $i\neq 0$.
In the remaining case where $i=0$, one gets $xy=yx$, and we again find that $\phi(xy)=\phi(yx)$.

It remains to check that the homomorphism \eqref{Hombimod}
is bijective. Since the domain and range are both finitely generated free $R$-modules of the same rank,
it suffices to verify that this map of $R$-modules is surjective.
By Nakayama's Lemma, it is enough to show that the induced map
\begin{equation} \label{Naka}
(A*G)\otimes_R R/\maxid_R \lra \Hom_R(A*G,R)\otimes_R R/\maxid_R
\end{equation}
is surjective. Let $\xoverline{A}:=A\otimes_R R/\maxid_R=k[[u_1,\ldots,u_n]]/(u_1^p,\ldots,u_n^p)$.
Since the action of $G$ is given by  $\sigma(u_s)=u_s+a_s$ with $a_s\in\maxid_R$, 
we find that $G$ acts trivially on $\xoverline{A}$, and $(A*G)\otimes_R R/\maxid_R$ is isomorphic to the commutative algebra 
$$
\xoverline{A}*G=k[u_1,\ldots,u_n,\sigma]/(u_1^p,\ldots,u_n^p,\sigma^p-1).
$$
The right-hand side of \eqref{Naka} is isomorphic to the $k$-vector space 
$\Hom_k(\xoverline{A}*G,k)$, because   the $R$-module $A*G$ is finitely generated.
Let $\bar{\phi}\in \Hom_k(\xoverline{A}*G,k)$ denote the class of $\phi$.

Let $x:=u_1^{i_1}\cdots u_n^{i_n}\sigma^i$ be any basis vector, and consider 
the linear form  $\varphi_x:\xoverline{A}*G\ra k$ such that
$\varphi_x(x)=1$ and  $\varphi_x(z)=0$ for any other basis vector $z$.
We exhibit now an element $y$ such that $y\bar{\phi} = \varphi_x$.
Consider the element $y:=u_1^{j_1}\cdots u_n^{j_n}\sigma^j$
with the  complementary exponents
$$
j_1:=p-1-i_1, \ \ldots, j_n:=p-1-i_n, \quadand j=p-i.
$$
Then 
$(y\bar{\phi})(x)=\bar{\phi}(xy)=1$ by definition of $\phi$. For every other basis vector $z=u_1^{j'_1}\cdots u_n^{j'_n}\sigma^{j'}$,
  we have $(y\bar{\phi})(z)=\bar{\phi}(zy) = 0$.
It follows that $\xoverline{A}*G\ra\Hom_k(\xoverline{A}*G,k)$ is surjective.
\qed

\medskip
\noindent
\emph{Proof of Theorem \ref{canonical class trivial}.}
By hypothesis,  $A$ is a complete regular local domain of dimension $n \geq 2$ with field of representatives $k$ and endowed with 
a moderately ramified action of the cyclic group $G$ of order $p$. 
Then Theorem \ref{structure moderately ramified} shows that the ring $A$, as a $k$-algebra with $G$-action, is isomorphic to 
$$
k[[x_1,\ldots,x_n]][u_1,\ldots,u_n]/(u_1^p-a_1^{p-1}u_1-x_1,\ldots,u_n^p-a_n^{p-1}u_n-x_n),
$$
where $x_1,\ldots,x_n$ are indeterminates, the elements  $a_1,\ldots,a_n\in k[[x_1,\ldots,x_n]]$
form a  system of parameters of $k[[x_1,\ldots,x_n]]$, and  the natural automorphism of order $p$ which sends $u_i$ to $u_i+a_i$ 
and fixes $k[[x_1,\ldots,x_n]]$ induces under this isomorphism a generator of $G$. As usual, we let $R:=k[[x_1,\ldots,x_n]]$.
Clearly $R \subset A^G \subset A$, and we may thus consider the skew group ring $A*G$ 
as an $R$-algebra.

Using \ref{useful.fact} and the fact that $K_R=R$ because $R$ is regular, we find that $K_{A^G}=\Hom_R(A^G,R)$. 
To prove Theorem \ref{canonical class trivial}, it thus suffices to show that there exists an isomorphism of $A^G$-modules  $A^G \to \Hom_R(A^G,R)$.
For this we need the following fact (Proposition 2.4 (3) in  \cite{Iyama; Reiten 2008}):

\begin{emp} \label{useful.fact2}
{\it Suppose that $\SR$ is a normal noetherian ring,
and $\Lambda$ is an associative $\SR$-algebra that is finitely generated as $\SR$-module.
Let $M$ be a finitely generated left $\Lambda$-module such that $M$ is reflexive as an $\SR$-module, and such that
given any prime ideal $\primid\subset \SR$ of height one, 
$M_\primid $ is a projective $\Lambda_\primid $-module.
If the $\SR$-algebra $\Lambda$ is symmetric, then the $\SR$-algebra
${\rm End}_{\Lambda}(M)$ is symmetric as well.}
\end{emp}

Let us now check that we can apply \ref{useful.fact2} to the case where $S:=R$, $\Lambda:=A*G$ and $M:=A$. 
Clearly $R$ is normal since it is regular, and $A*G$ is a finitely generated $R$-module.
According to Proposition \ref{comparison maps bijective}, the $R$-module $A$ is reflexive, 
and for each prime ideal $\primid\subset R$ of height one, the $(A*G)_\primid$-module $A_\primid$
is projective.
Finally,  Proposition \ref{symmetric} shows that $A*G$ is a symmetric $R$-algebra.
Thus \ref{useful.fact2}  can be applied and we find that 
the $R$-algebra
${\rm End}_{A*G}(A)$ is symmetric: In other words, we have an isomorphism of ${\rm End}_{A*G}(A)$-bimodules
between ${\rm End}_{A*G}(A)$ and $\Hom_R({\rm End}_{A*G}(A), R)$. 
Using  
\eqref{iso2}, we find an isomorphism of $A^G$-algebras $A^G \to {\rm End}_{A*G}(A)$.
Hence, the $A^G$-modules $A^G$ and $K_{A^G}=\Hom_R(A^G,R)$ are isomorphic, as desired.
\qed

\begin{remark} Recall that when a complete local ring is Cohen--Macaulay, its canonical module is free of rank $1$
if and only if the ring is Gorenstein (\cite{Herzog; Kunz 1971}, 5.9). 
Local rings with trivial canonical class are called \emph{quasi-Gorenstein}
by  Platte and Storch (\cite{Platte; Storch 1977}, page 5). 
Theorem \ref{canonical class trivial} 
provides a rich supply of quasi-Gorenstein rings
which are not  Cohen--Macaulay when they have dimension bigger than $2$ (use \ref{properties invariant ring} (i)).

Recall also that if a finite group $G$ acts on   a Cohen--Macaulay ring $A$
and $|G|$ is invertible in $A$,
then the ring $A^G$ is Cohen--Macaulay (\cite{H-E}, Proposition 13). 
A necessary and sufficient condition for $A^G$ to be Gorenstein when $A$ is regular is given in \cite{WatII}, Theorem 2. 
Theorem 4 in \cite{WatI} shows
that $A^G$ is Gorenstein if the image of the associated map $\lambda: G \to {\rm GL}(\maxid_A/\maxid_A^2)$ is in ${\rm SL}(\maxid_A/\maxid_A^2)$.
See also Conjecture 5 in \cite{Kemperetal} when $|G|$ is not invertible in $A=k[V]$, where $V$ is a $k$-vector space with an action of $G$.
\end{remark}


\end{document}